\documentclass[11pt,a4paper]{amsart}
\usepackage[foot]{amsaddr}

\usepackage{caption}
 \usepackage{geometry}
\usepackage{amsmath,amsfonts,amssymb,amsthm,mathrsfs}
\usepackage{graphicx}
\usepackage{stmaryrd}
\usepackage{hyperref}
\usepackage{tikz}
\usetikzlibrary{patterns,calc,arrows.meta} 
\usepackage{dsfont}
\usepackage{mathtools}
\usepackage{subcaption}

\newtheorem{theorem}{Theorem}[section]
\newtheorem{lemma}[theorem]{Lemma}
\newtheorem{proposition}[theorem]{Proposition}

\newtheorem{assumption}[theorem]{Assumption}

\newtheorem{question}[theorem]{Question}
\newtheorem{questions}[theorem]{Questions}

\theoremstyle{remark}
\newtheorem{remark}[theorem]{Remark}
\newtheorem{definition}[theorem]{Definition}
\newtheorem{claim}[theorem]{Claim}

\numberwithin{equation}{section}

\newcommand{\R}{\mathbb{R}}
\newcommand{\N}{\mathbb{N}}
\newcommand{\Z}{\mathbb{Z}}
\def\P{\mathbb{P}} 
\newcommand{\E}{\mathbb{E}}

\newcommand{\id}{\mathds{1}}
\newcommand{\Var}{\mathrm{Var}}
\newcommand{\eps}{\varepsilon}

\newcommand{\arm}{\text{\textup{Arm}}}
\newcommand{\ind}{\mathds{1}}
\newcommand{\wick}[1]{\;{:}#1{:}\;}

\begin{document}

\title[Limit theorems for the cluster count of the GFF]{Limit theorems for the number of sign and level-set clusters of the Gaussian free field}
\author{Michael McAuley\textsuperscript{1}}
\address{\textsuperscript{1}School of Mathematics and Statistics, Technological University Dublin} 
\email{m.mcauley@cantab.net}
\author{Stephen Muirhead\textsuperscript{2}}
\address{\textsuperscript{2}School of Mathematics, Monash University}
\email{stephen.muirhead@monash.edu}
\thanks{}
\begin{abstract}
We study the limiting fluctuations of the number of sign and level-set clusters of the Gaussian free field on $\Z^d$, $d \ge 3$, that are contained in a large domain. In dimension $d \ge 4$ we prove that the fluctuations are Gaussian at all non-critical levels, while in dimension $d=3$ we show that fluctuations may be Gaussian or non-Gaussian depending on the level. We also show that the sign clusters experience a form of Berry cancellation in all dimensions, that is, the fluctuations of the sign cluster count are suppressed compared to generic levels.

Our proof is based on controlling the Wiener-It\^{o} chaos expansion of the cluster count using percolation theoretic inputs; to our knowledge this is the first time that chaos expansion techniques have been applied to analyse a non-local functional of a strongly correlated Gaussian field. 
\end{abstract}

\date{\today}

\maketitle


\section{Introduction}

The \textit{Gaussian free field} (GFF) on $\Z^d$, $d \ge 3$, is the centred stationary Gaussian field $f$ on $\Z^d$ with covariance kernel
\begin{equation}
    \label{e:g}
\E[f(x) f(y)] = G(x-y) , 
\end{equation}
where $G(x) = \sum_{k \ge 0} \P[X_k = x \, | \, X_0 = 0]$ is the Green's function for the simple random walk $(X_k)_{k \ge 0}$ on $\Z^d$. The GFF is a central object in probability theory and mathematical physics; see \cite{bis20,wp22} for recent introductions. One of its characteristic features is the presence of strong non-integrable correlations decaying as $G(x) \sim c_d |x|^{2-d}$ as $|x| \to \infty$.

\smallskip
The \textit{sign clusters} of the GFF are the connected components of the sets $\{f > 0\} := \{ x \in \Z^d: f(x) > 0\}$ and $\{ f < 0 \} := \{ x \in \Z^d: f(x) < 0\}$; equivalently, neighbouring vertices $x\sim y \in \Z^d$ are in the same sign cluster if and only if $f(x) f(y) > 0$. More generally, the \textit{level-set clusters} at level $\ell \in \R$ are the connected components of the excursion sets $\{ f > \ell\}$ and $\{ f < \ell\}$.

\smallskip
The sign clusters of the GFF exhibit deep connections to the simple random walk and related objects via the \textit{isomorphism theorems} of Dynkin-BFS and Le Jan (see \cite{lup16} for an overview). The geometry of the sign and level-set clusters of the GFF has been the object of extensive study over the last 40 years \cite{ls86,blm87}.  \cite{rs13,dpr18} have established the existence of a \textit{phase transition} in the connectivity of level-set clusters, namely that there exists a \textit{critical level} $\ell_c \in (0,\infty)$ such that, if $\ell > \ell_c$, all components of $\{f > \ell\}$ are bounded, whereas if $\ell < \ell_c$ then almost surely $\{ f > \ell\}$ contains a (unique) unbounded component. \cite{szn15,pr15,nit18,szn19,cn20,grs22} have studied fine properties of the so-called \textit{strongly subcritical and supercritical} phases. Recently \cite{dgrs23} has confirmed a long-standing prediction that the phase transition is \textit{sharp}, meaning that the boundaries of the strongly subcritical and supercritical phases coincide: as a consequence, if $\ell \neq \ell_c$ then all bounded components of $\{f > \ell\}$ have `small diameter' (see Section \ref{s:tae} for a precise statement).

\smallskip
In this paper we consider the \textit{number} of bounded sign and level-set clusters of the field that are contained in a large domain (the \textit{cluster count}). Precisely, we define $N_R^+(\ell)$ (resp.\ $N_R^-(\ell)$) to  be the number of connected components of $\{f > \ell\}$ (resp.\ $\{f < \ell\}$) that intersect $\Lambda_R$ but not $\partial \Lambda_R$, where $\Lambda_R = [-R,R]^d \cap \Z^d$, and $\partial D = \{x \in D: \exists y \notin D, x \sim y\}$ denotes the inner boundary of $D \subset \Z^d$. The cluster count is $N_R(\ell) = N_R^+(\ell) + N_R^-(\ell)$. See Section \ref{s:bc} for comments on our choice of boundary conditions. Note that $N_R(\ell)$ and $N_R(-\ell)$ have the same distribution by symmetry.

\smallskip
It is straightforward to establish (see Proposition \ref{p:mu}) that $N_R(\ell)$ satisfies a law of large numbers: as $R \to \infty$,
\begin{equation}
    \label{e:lln}
 \frac{ N_R(\ell) }{\mathrm{Vol}(\Lambda_R)}  \to \mu(\ell)  \qquad \text{a.s.\ and in $L^1$},
 \end{equation} 
where $\mu(\ell) \in (0,1)$ is the \textit{cluster density} at level $\ell$. The \textit{fluctuations} of the cluster count turn out to be more subtle. Our main result establishes the limiting distribution at all non-critical levels; perhaps surprisingly, the limiting fluctuations may be Gaussian or non-Gaussian depending on the dimension and the level.

\subsection{Limit theorems for the cluster count}
Our first result shows that, in dimensions $d \ge 4$, the cluster count  has asymptotically Gaussian fluctuations at all non-critical levels. Let $Z$ denote a standard Gaussian random variable, and $\Rightarrow$ convergence in law. 

\begin{theorem}
\label{t:fluc1}
Let $d \ge 4$ and $\ell \notin \{-\ell_c,\ell_c\}$. Then as $R \to \infty$,
\begin{equation}
    \label{e:gl}
\frac{N_R(\ell) - \E[N_R(\ell)]}{\sqrt{ \Var[N_R(\ell)]} } \quad \Longrightarrow \quad Z  .
\end{equation}
In particular \eqref{e:gl} holds for sign~clusters.
\end{theorem}

In dimension $d=3$, the limiting fluctuations exhibit a more complicated behaviour: two distinct limiting distributions are possible depending on the level. Recall that $\mu(\ell)$ denotes the cluster density defined in \eqref{e:lln}. Recently it was shown \cite{ps22} that $\ell \mapsto \mu(\ell)$ is real-analytic on $\R \setminus \{-\ell_c,\ell_c\}$ (in Proposition \ref{p:mudiff} we give an alternative proof of smoothness on $\R \setminus \{-\ell_c,\ell_c\}$, and also establish continuous differentiability at $\ell_c$). For $i\in\N$, we now define
 \begin{equation}
    \label{e:c}
 \mathcal{C}^{(i)}:= \big\{ \ell \in \R \;:\;\mu^{(1)}(\ell)=\dots=\mu^{(i-1)}(\ell)=0,\mu^{(i)}(\ell) \neq 0 \big\}.
 \end{equation}
That is, $\mathcal{C}^{(i)}$ consists of the levels at which the first non-vanishing derivative of the mean functional $\mu$ is of order $i$. In particular, $\mathcal{C}^{(2)}$ denotes the set of non-degenerate critical points of $\mu$. 

\begin{theorem}
\label{t:fluc2}
Let $d = 3$ and $\ell \notin \{-\ell_c,\ell_c\}$. Then as $R \to \infty$,
\[ \frac{N_R(\ell) - \E[N_R(\ell)]}{\sqrt{ \Var[N_R(\ell)]} } \quad \Longrightarrow  \quad \begin{cases}  Z & \text{if } \ell \notin \mathcal{C}^{(2)} , \\ Z' & \text{if } \ell \in \mathcal{C}^{(2)}  ,\end{cases}\]
where $Z'$ has a non-Gaussian order-$2$ Hermite distribution associated to the measure with density $\rho(\lambda) = |\lambda|^{-2}$ (see Definition \ref{d:ros}).
\end{theorem}

In Lemma \ref{l:high} we show that $\cup_{i\geq 2}\mathcal{C}^{(i)}$ is bounded, and hence, by analyticity, is finite outside any neighbourhood of $\{-\ell_c,\ell_c\}$. Also by symmetry we must have $\mu'(0) = 0$, but we are unable to rule out that $\mu^{\prime \prime}(0) = 0$, so we do not know which case the sign clusters fall into. Since a generic function does not have degenerate critical points, it is natural to expect that all critical points of $\mu$ (including $0$) are non-degenerate, but we are unable to confirm this.

\begin{questions}
In dimension $d=3$, is $Z$ or $Z'$ the limiting distribution of $N_R(0)$ after centring and rescaling? Is $\mathcal{C}^{(2)}\setminus\{-\ell_c,\ell_c\}$ non-empty? What is the limiting distribution of $N_R(\ell_c)$?
\end{questions}

\subsection{Order of the fluctuations}
\label{s:fluc}
In contrast to the limiting distribution, there are several possibilities for the order of the fluctuations in all dimensions. Recall that $G(x) = G_d(x)$ is the covariance kernel \eqref{e:g} of the GFF, and satisfies $G(x) \ge 0$ and $G(x) \sim c_d |x|^{2-d}$ as $|x| \to \infty$. For $k \ge 0$, let $\beta_{d,k} > 0$ be constants defined as
\begin{equation}
    \label{e:beta}
\beta_{d,k} := \lim_{R \to \infty} \frac{ \sum_{x,y \in \Lambda_R } G(x-y)^k }{   R^{\max\{2d - k(d-2), d \} } (\log R)^{\id_{k = d/(d-2)}} } .
\end{equation}
These exist by Lemma \ref{l:rv} (see also Remark \ref{r:beta}), and could be computed explicitly. 

\begin{theorem}
\label{t:var}
Let $\ell \notin \{-\ell_c,\ell_c\}$, then there exist constants $\sigma = \sigma_{\ell,d} > 0$ such that, as $R \to \infty$:

\begin{enumerate}
    \item  If $d \ge 5$, 
\[  \Var[N_R(\ell)]  \sim  \begin{cases}   \beta_{d,1} (\mu'(\ell))^2  R^{d+2} & \text{if }  \ell\in\mathcal{C}^{(1)} , \\ \sigma^2 R^d & \text{else};\end{cases}  \]
\item If $d = 4$,
\[  \Var[N_R(\ell)]  \sim  \begin{cases}  \beta_{4,1} (\mu'(\ell))^2  R^6 & \text{if } \ell\in\mathcal{C}^{(1)}, \\
\frac{ \beta_{4,2} (\mu''(\ell))^2}{2}  R^4 (\log R) & \text{if } \ell \in \mathcal{C}^{(2)},  \\ \sigma^2 R^4 & \text{else}; \end{cases}  \]
\item If $d = 3$,
\[  \Var[N_R(\ell)]  \sim  \begin{cases}  \beta_{3,1} (\mu'(\ell))^2  R^5 & \text{if } \ell\in\mathcal{C}^{(1)}, \\  \frac{ \beta_{3,2}   (\mu''(\ell))^2 }{2} R^4 & \text{if } \ell \in  \mathcal{C}^{(2)} , \\  \frac{ \beta_{3,3} (\mu'''(\ell))^2 }{3!} R^3 (\log R) & \text{if } \ell\in\mathcal{C}^{(3)}  , \\ 
\sigma^2 R^3 & \text{else}. \end{cases}  \]
\end{enumerate}
In particular if $d \ge 5$ the sign clusters satisfy $\Var[N_R(0)] \sim \sigma^2 R^d$ for some $\sigma = \sigma_d > 0$.
\end{theorem}

A consequence of Theorem \ref{t:var} is that the cluster count at levels at which $\mu'(\ell) = 0$ has \textit{suppressed fluctuations} compared to generic levels at which $\mu'(\ell) \neq 0$. In particular, since  $\mu'(0) = 0$ by symmetry, fluctuation suppression occurs for the sign clusters. This is an analogue of the fluctuation suppression (`Berry cancellation') known to occur for certain other geometric functionals of nodal sets of Gaussian fields, as first observed for the nodal length of random wave models by Berry \cite{ber02}, and later confirmed rigorously \cite{wig10,kkw13}. 

As in Theorem \ref{t:fluc2}, the case $\mu'(\ell) \neq 0$ in Theorem \ref{t:var} occurs at all sufficiently high levels, and the remaining cases occur at a (non-zero) finite number of levels except for possible accumulation points at $\{-\ell_c,\ell_c\}$. We believe, but cannot prove, that only the first two cases in Theorem \ref{t:var} occur in dimensions $d \in \{3,4\}$, and so in particular the sign clusters fall into the second case.

In Proposition \ref{p:mudiff} we show that the constants $\mu^{(m)}$ appearing in Theorem \ref{t:var} can also be expressed as sums over certain $m$-point `pivotal intensities'. An expression for the constants $\sigma_{\ell,d} > 0$ appearing in the final cases of Theorem \ref{t:var} is given in \eqref{e:sigma}; this includes a non-negligible boundary contribution in all dimensions (see Remark \ref{r:be}).

\begin{question}
Are all critical points of $\mu$ non-degenerate, so that only the first two cases in Theorem \ref{t:var} occur?
 \end{question}

In the critical case $\ell = \ell_c$ we establish bounds on the variance which match the orders of the extremal cases of Theorem \ref{t:var}:

\begin{theorem}
\label{t:var2}
Let $d \ge 3$. Then there exist $c_1,c_2 > 0$ such that, for all $R \ge 1$,
\begin{equation}
    \label{e:var41}
 c_1 R^d \le \Var[N_R(\ell_c)] \le c_2 R^{d+2} .
 \end{equation}
Moreover, if $\mu'(\ell_c) \neq 0$ then, as $R \to \infty$,
\begin{equation}
    \label{e:var42}
\Var[N_R(\ell_c)] \ge \beta_{d,1} (\mu'(\ell_c))^2  R^{d+2} (1 + o(1)).
\end{equation} 
\end{theorem}

Theorem \ref{t:var2} is much simpler to prove than Theorem \ref{t:var}. The upper bound in \eqref{e:var41} essentially follows from a simplified version of arguments developed in \cite{bmm24b} in the setting of smooth Gaussian fields. We refer to the volume-order lower bound in \eqref{e:var41} as the \textit{extensivity} of the variance, and here we give a new proof of this that is valid at every level. This is also an important ingredient in the proof of our stronger results for non-critical levels in the case that $\Var[N_R(\ell)] \sim \sigma^2 R^d$. The bound \eqref{e:var42} follows from the same arguments used to establish the case $\mu'(\ell) \neq 0$ of Theorem \ref{t:var} (see \cite{bmm22} for an alternative approach).


\begin{questions}
Is $\mu'(\ell_c) \neq 0$? Is the lower bound in \eqref{e:var42} asymptotically tight?
 \end{questions}

\subsection{Related work}
While to our knowledge the fluctuations of the cluster count of the GFF have not been considered in the literature, several works have studied the analogous problem for closely related models.

\smallskip
For an i.i.d.\ Gaussian field on $\Z^d$ the cluster count is equivalent, up to re-parameterisation, to the count of \textit{percolation clusters} in classical Bernoulli site percolation. In this case the limiting fluctuations are known to be Gaussian, first shown for non-critical levels \cite{cg84} and later at criticality \cite{pen01,zha01}. In \cite{bmm24} the martingale method of \cite{pen01} was extended to a class of weakly-dependent smooth Gaussian fields on $\R^d$. Since the methods of \cite{cg84,pen01,zha01,bmm24} only produce Gaussian fluctuations, one cannot hope to fully characterise the limiting distributions of fields with strong correlations, such as the GFF, using such methods.

\smallskip
For general Gaussian fields, including those with strong correlations, weaker results have previously been shown. The seminal work of Nazarov-Sodin \cite{ns09,ns16} established the law of large numbers \eqref{e:lln} for essentially all stationary Gaussian fields. Recently \cite{ns20, bmm22,bmm24b} have given upper and lower bounds on the variance. Translated to the setting of the GFF, \cite{bmm24b} showed that, at all levels,
\begin{equation}
    \label{e:var1}
\Var[N_R(\ell)] \le c_1 R^{d+2} ,
\end{equation} 
and \cite{bmm22} showed that, if $\mu'(\ell) \neq 0$,
\begin{equation}
    \label{e:var2}
 \Var[N_R(\ell)] \ge   c_2 R^{d+2}  
 \end{equation}
 for some $c_1,c_2 > 0$. \cite{ns20} gave the weaker lower bound $\Var[N_R(\ell)] \ge c_3 R^\delta$ at all levels for a general class of two-dimensional fields, where $\delta > 0$ is non-explicit but small. Note that although \cite{ns09,ns16,ns20,bmm22,bmm24b} studied smooth Gaussian fields, the proofs can be adapted to the (simpler) setting of discrete fields without much difficulty.

\subsection{Strategy of the proof}

Our proofs are inspired by the analysis of \textit{local additive functionals} of stationary Gaussian fields, which, by definition, take the form
\[
    \Psi(D,f):=\sum_{x\in D}\psi(f(x)),
\]
where $f$ is a stationary Gaussian field, $\psi:\R\to\R$ is measurable, and $D\subset\Z^d$ is finite. (One might also describe such a functional as stationary, to highlight the fact that the same function $\psi$ is applied to the field at each point.) The cluster count cannot be expressed in such a form, and so is described as \textit{non-local}. However it turns out that the cluster count can be well-approximated by a stationary sum
\begin{equation}
\label{e:stationary_approx1}
  N_R(\ell) \approx  \sum_{x\in D}\psi((f(x+\cdot))),
\end{equation}
where $\psi$ now depends on the whole field, but only weakly on points far away from the origin. Indeed one may express the cluster count as
\begin{equation}\label{e:stationary_approx}
N_R(\ell)=\sum_{x\in\Lambda_R}\frac{1}{\lvert C_x(\ell)\rvert}\ind_{A_x(R)}
\end{equation}
where $C_x(\ell)$ is the cluster of $\{f>\ell\}$ or $\{f<\ell\}$ containing $x$, $\lvert C\rvert$ denotes the number of elements of the finite set $C$, and $A_x(R)$ is the event that $C_x$ does not intersect $\partial\Lambda_R$. Approximating $\ind_{A_x(R)} \approx 1$ in the bulk of $D$ puts $N_R(\ell)$ into the form \eqref{e:stationary_approx1}.

We describe functionals that can be approximated in the form \eqref{e:stationary_approx1} as being \emph{semi-local}. The main thrust of our proof is to show that the classical limit theory of local additive functionals can be extended to semi-local functionals such as the cluster count.

\subsubsection{Limit theory for local additive functionals}
Our proofs are inspired by the analysis of \textit{local additive functionals} of stationary Gaussian fields, which, by definition, take the form
\[
    \Psi(D,f):=\sum_{x\in D}\psi(f(x)),
\]
where $f$ is a stationary Gaussian field, $\psi:\R\to\R$ is measurable and $D\subset\Z^d$ is finite. (One might also describe such a functional as stationary, to highlight the fact that the same function $\psi$ is applied to the field at each point, although this adjective is usually omitted in the literature.) The cluster count cannot be expressed in such a form, and so is described as non-local. However it heuristically corresponds to taking a roughly stationary sum of contributions over a domain. To make this intuition more rigorous, one may express the cluster count as
\begin{equation}\label{e:stationary_approx}
    N_R(\ell)=\sum_{x\in\Lambda_R}\frac{1}{\lvert C_x(\ell)\rvert}\ind_{A_x(R)}
\end{equation}
where $C_x(\ell)$ is the cluster of $\{f>\ell\}$ or $\{f<\ell\}$ containing $x$, $\lvert C\rvert$ denotes the number of elements of the finite set $C$ and $A_x(R)$ is the event that $C_x$ does not intersect $\partial\Lambda_R$. This makes it apparent that the cluster count can be viewed as a generalisation of a local additive functional $\Psi$ to the case that $\psi$ is non-stationary and depends on the values of the field on an arbitrarily large domain. Hence it is natural to consider the simpler class of local functionals with the hope of extending the corresponding results or analysis to the cluster count.

Local additive functionals of strongly-correlated stationary Gaussian fields have been well studied using methods which we now describe. Let $f$ be a centred stationary Gaussian field on $\Z^d$ with covariance kernel $K(x) = \E[f(0)f(x)]$ satisfying $K \ge 0$, $K(0)=1$ and  $K(x) \sim c |x|^{-\alpha}$ as $|x| \to \infty$ for some $c, \alpha > 0$. Let $\Psi: \R \to \R$ be a function such that $\Var[\Psi(Z)] \in (0,\infty)$, where $Z$ is a standard Gaussian random variable; such a function admits a \textit{Hermite expansion}
\begin{equation}
\label{e:hermite}
\Psi(x) = h_0 + \sum_{m \ge k} \frac{h_m}{m!} H_m(x) \, , \quad h_i \in \R 
\end{equation}
into Hermite polynomials $H_m$ of order $m \ge k$, where $k \ge 1$ is the \textit{Hermite rank} of $\Psi$ (meaning that $h_k \neq 0$). An alternative expression for the coefficients is
\begin{equation}
    \label{e:cmalt}
h_m = \E\big[ H_m(Z) \Psi(Z) \big] =  \frac{d^{(m)}}{d\ell^{(m)}} \E \big[\Psi(Z + \ell) \big] .
\end{equation}
We consider the local additive functional
\[ \Psi_R := \sum_{x \in \Lambda_R} \Psi( f(x) ) \ , \qquad   \widetilde{\Psi}_R := \frac{\Psi_R - \E[\Psi_R]}{\sqrt{\Var[\Psi_R]}} .\]
The possible limiting distributions of $\widetilde{\Psi}_R$ as $R \to \infty$ were first determined by Dobrushin-Major \cite{dm79} and Breuer-Major \cite{bm83}. To analyse these, one can consider the \textit{Wiener-It\^{o} chaos expansion} (see Section \ref{ss:ce} for details),
\begin{equation}
    \label{e:chaosloc}
 \Psi_R = \E[\Psi_R] + \sum_{m \ge 1} Q_m[\Psi_R]
\end{equation}
which decomposes $\Psi_R$ into the uncorrelated components
\begin{equation}\label{e:chaoslocal}
    Q_m[\Psi_R]=\frac{h_m}{m!}\sum_{x\in\Lambda_R}H_m(f(x)).
\end{equation}
Using the \textit{diagram formula} (Theorem~\ref{t:df}) one can show that 
\begin{equation}
    \label{e:chaosvar}
\Var[Q_m[\Psi_R]]  =    \frac{ \id_{m \ge k} h_m^2 }{m!}   \sum_{x,y \in \Lambda_R} K(x-y)^m
\end{equation}
and so in particular (c.f.\ \eqref{e:beta}), as $R \to \infty$
\begin{equation}
    \label{e:chaosvar2}
\frac{\Var[Q_m[\Psi_R]] }{ R^{\max\{2d-m\alpha,d\}} (\log R)^{\id_{m = d/\alpha} } }  \to  \frac{ \id_{m \ge k} h_m^2}{m!}  c_{K,m}
\end{equation} 
where $c_{K,m} > 0$ (recall that $K(x) \sim c |x|^{-\alpha}$). Moreover one can show that, provided $m \ge k$ and $h_m \neq 0$, as $R \to \infty$
\begin{equation}
\label{e:chaoslim}
\frac{Q_m[\Psi_R]}{\sqrt{\Var[Q_m[\Psi_R]]}} \quad \Longrightarrow  \quad \begin{cases}
   Z  & \text{if either } m = 1  \text{ or } m \ge d/\alpha, \\
    Z'_m & \text{if } 2 \le m  < d/\alpha  ,
\end{cases} 
\end{equation} 
where $Z'_m$ has an \textit{order-$m$ Hermite distribution} (see Definition \ref{d:ros}). In the case $m \ge d/\alpha$, the convergence in \eqref{e:chaoslim} may be established by appealing to the method of moments \cite{bm83} (or its refinement: the fourth moment theorem \cite{np05}), and in the case $2 \le m < d/\alpha$ by exploiting self-similarity.

\smallskip
Combining \eqref{e:chaosloc}--\eqref{e:chaoslim} (and using the fact that $\sum_mh_m^2m!<\infty$ for any Hermite expansion to control the variance of the tail in \eqref{e:chaosloc}) one can conclude that, as $R \to \infty$,
\begin{equation}
    \label{e:wclimit}
 \widetilde{\Psi}_R  \quad \Longrightarrow  \quad \begin{cases} Z &  \text{if either } k = 1 \text{ or } k \ge d/\alpha,\\ 
Z_k' & \text{if } 2 \le k  < d/\alpha.  \\ \end{cases}  
\end{equation} 

These arguments can be extended to more general correlation structures. However to conclude \eqref{e:wclimit} it can sometimes be necessary to independently verify that many chaotic components do not vanish simultaneously; this is often done by showing that the variance is \textit{extensive} ($\Var[\Psi_R] > c R^d$).

\smallskip
In recent decades there has been a substantial development of chaos expansion theory to study other local additive functionals of stationary Gaussian fields, starting from the work of Slud on the level-crossings of Gaussian processes \cite{slu91,slu94}. These methods have proved extremely flexible; without attempting to be exhaustive, let us mention applications to the level set geometry of smooth Gaussian fields with weak \cite{el16} and oscillating correlations \cite{mw11,mprw16,npr19,mrw20}, and the zeros of Gaussian analytic functions \cite{bn22}. In particular the impact of the \textit{asymptotic vanishing} of lower order chaoses on the limit theory of geometric functionals has been well understood \cite{mw11,mprw16,npr19,mrw20}.We remark that for a \emph{smooth} Gaussian field $f:\R^d\to\R$, local additive functionals take the form
\[
    \Psi(D,f)=\int_D\psi(f(x),\nabla f(x),\nabla^2 f(x))\;\mu(dx)
\]
where $\mu$ is some measure, $D$ is a compact domain and $\psi$ is a real-valued function. This framework includes examples like the length of level-set curves and the Euler characteristic of excursion sets.

\smallskip

To our knowledge, outside the special case of quadratic forms \cite{ft87,tt90}, this method has never been successfully implemented to study \textit{non-local} functionals such as the cluster count. Although one can generically decompose any $L^2$ functional of a Gaussian field into chaotic components, the difficulty lies in finding a tractable expression for asymptotic analysis. 

\subsubsection{Our contributions}
Our work develops this theory in several ways:

\vspace{-0.1cm}
\subsubsection*{Chaos expansion for smooth functionals}
For a generic smooth (non-local) functional $\Phi:\R^{\Lambda_R}\to\R$, we show (Proposition \ref{p:correlated_wce}) that the $m$-th chaotic component can be expressed as
    \begin{equation}
    \label{e:introce}
        Q_m[\Phi(f)]=\frac{1}{m!}\sum_{x_1,\dots,x_m\in\Lambda_R}\E[\partial_{x_1}\dots\partial_{x_m}\Phi(f)]\wick{f(x_1)\cdots f(x_m)}
    \end{equation}
    where $\wick{f(x_1)\cdots f(x_m)}$ denotes the Wick product of $f(x_1),\dots,f(x_m)$ (see Section~\ref{ss:ce} for the definition). Observe that for the local functional $\Phi=\Psi_R$ only the diagonal $x_1=\dots=x_m$ contributes to \eqref{e:introce}, so by passing the derivative through the expectation in \eqref{e:cmalt} and using that $\wick{Z^m}=H_m(Z)$ for a standard Gaussian variable $Z$, the above expression simplifies to \eqref{e:chaoslocal}. While it is straightforward to derive \eqref{e:introce}, to our knowledge this expression has not appeared before in the literature (c.f.\ Remark \ref{r:alt}).

\vspace{-0.1cm}
\subsubsection*{Chaos expansion for level-set functionals}
We extend the previous expression to (non-smooth, non-local) level-set functionals $\Xi(f)$ (i.e.\ functionals that depend only on the excursion set $\{ f > \ell\}$), showing that (Theorem \ref{t:cluster_wce})
    \begin{displaymath}
        Q_m[\Xi(f)]=\frac{1}{m!}\sum_{x_1,\dots,x_m\in\Lambda_R}P(x_1,\dots,x_m)\wick{f(x_1)\cdots f(x_m)}
    \end{displaymath}
    where $P$ is the \textit{(multi-point) pivotal intensity} of $\Xi$ given in Definition \ref{d:pivint}. In the case of distinct $(x_i)_i$, $P(x_1,\dots,x_m)$ is the expectation of $d_{x_1}\dots d_{x_m}\Xi$ conditionally on $f$ taking the value $\ell$ at $x_1,\dots,x_m$, where $d_x\Xi$ is the discrete derivative of $\Xi$ at $x\in\Lambda_R$ (i.e.\ the change to $\Xi$ upon adding the point $x$ to the excursion set). Repeated points are handled by passing derivatives onto a Gaussian density.

\vspace{-0.1cm}
\subsubsection*{Semi-localisation of the cluster count}
When $\Xi$ is the cluster count $N_R$, we use percolation theoretic inputs to show that the pivotal intensities decay rapidly away from the diagonal (Lemma \ref{l:decay}) and are well approximated (away from $\partial\Lambda_R$) by stationary counterparts (Lemma~\ref{l:conv}).

The proof of off-diagonal decay is based on the observation that, for $d_x\Xi$ to be non-zero, all of the points $x_1,\dots,x_m$ must be connected by bounded clusters of the upper/lower excursion sets (for the conditioned field). It has recently been shown that, for the unconditioned field at non-critical levels, such \emph{truncated arm events} exhibit rapid probability decay as the diameter of the cluster increases \cite{dgrs23}. We extend this to the conditioned field using a `de-pinning' argument (Section~\ref{s:tae}). A closely related \emph{stabilisation} argument yields approximation by stationary pivotal intensities.

\vspace{-0.1cm}
\subsubsection*{Limit theory for semi-local chaotic components}
We extend the classical limit theory for chaotic components of local functionals (outlined in \eqref{e:chaosvar}-\eqref{e:chaoslim}) to the semi-local case (Appendix \ref{a:slaf}). In particular, for $m\alpha<d$ (with the GFF corresponding to $\alpha = d-2$) we show that, as $R\to\infty$
    \begin{displaymath}
    \Var[Q_m[\Xi(f)]]\sim \frac{1}{m!}c_m\Big(\sum\nolimits_{x_2,\dots,x_m\in\Z^d}P_\infty(0,x_2,\dots,x_m)\Big)^2R^{2d-m\alpha},
    \end{displaymath}
where $c_m$ is the same constant given in \eqref{e:chaosvar2} for $K = G$. We give a separate argument that the sum in brackets above is equal to $(-1)^m\mu^{(m)}(\ell)$, explaining why \textit{fluctuation suppression} occurs at levels for which the first derivatives of $\mu$ vanish. 

We prove corresponding variance asymptotics for the higher order chaotic components (i.e.\ $m\alpha\geq d$), although the expressions for the leading constants are less simple. We also establish limiting distributions (i.e.\ the analogue of \eqref{e:chaoslim}) in all cases. The arguments are similar to those in the classical theory of local functionals, albeit with extra technicalities (including non-negligible boundary effects in some cases).

\vspace{-0.1cm}
\subsubsection*{Controlling the tail of the chaos expansion}
To prove limit theorems for the cluster count we also require control over the tail of the chaos expansion
\begin{equation}
\label{e:tailvar}
    \sum_{m^\prime\geq m}Q_{m^\prime}[\Xi(f)].
\end{equation}
For this we use an iterated interpolation formula to express the variance of \eqref{e:tailvar} in terms of certain (joint) pivotal intensities of \textit{fixed order} $m$ (Proposition \ref{p:wcerror}). This is crucial since, on the one hand the percolation theoretic inputs do not afford us sufficient uniform control on the higher chaoses directly, and on the other hand we cannot use classical tail inequalities (such as those in \cite{hp95}) due to the lack of differentiability of level-set functionals.

\vspace{-0.1cm}
\subsubsection*{Extensivity of fluctuations}
We give a separate argument to establish the \textit{extensivity} of fluctuations at all levels (Theorem \ref{t:var2}), which confirms that not all chaotic components have asymptotically vanishing variance.

\smallskip
Although we focus on the cluster count, our approach is of independent interest and paves the way to study other non-local functionals of strongly correlated Gaussian fields.

\subsection{Discussion and extensions}
\label{s:bc}
\subsubsection{Excursion set counts, boundary conditions}

Our proof extends in a straightforward way in several directions.

\smallskip
First, our proof applies to the functionals $N^+_R(\ell)$ and $N^-_R(\ell)$ which count the components of the excursion sets $\{ f > \ell\}$ and $\{f < \ell\}$. In this case all of our results remain true if the density $\mu(\ell)$ defined in \eqref{e:lln} is replaced by its analogue for $N^{\pm}_R(\ell)$. In fact, for $N^-_R(\ell)$ (resp.\ $N^+_R(\ell)$) our conclusions hold also at the level $\ell_c$ (resp.\ $-\ell_c$) which is not critical with respect to this excursion set.

\smallskip
Second, we could modify the boundary conditions in the definition of the functional $N_R$ without significant changes to the proof; possible alternatives are: (1) counting all the level-set clusters of the field $f$ restricted to $\Lambda_R$, or (2) counting all level-set clusters of the field $f$ which intersect $\Lambda_R$. Note however that if the variance has volume-order growth (i.e.\ the final cases of Theorem \ref{t:var}), the leading constant may be different for each of these choices (see Remark~\ref{r:be}). Note also that, with the latter choice, the functional is not measurable with respect to the field on $\Lambda_R$, but one could handle this difference by working inside a larger box $\Lambda_{2R}$ and using the percolation theoretic inputs to argue that there are very few clusters that intersect both $\Lambda_R$ and $\partial \Lambda_{2R}$.

\smallskip
Another natural choice would be to impose either \textit{Dirichlet} or \textit{periodic} boundary conditions on the GFF restricted to $\Lambda_R$ (in the latter case one considers the `zero-averaged' GFF); however, extra work would be required to adapt our arguments to these settings.

\subsubsection{Other semi-local level-set functionals}
We believe our analysis could eventually be extended to cover other functionals of the level/excursion sets which are `semi-local'. One important example is the density of the unbounded component of $\{f > \ell\}$ for supercritical levels $\ell < \ell_c$: 
\[ \theta_R(\ell) = \mathrm{Vol}( \Lambda_R \cap \mathcal{U}_\ell ) \] where $\mathcal{U}_\ell$ denotes the (unique) unbounded component of $\{f > \ell\}$. While we expect a variant of our proof to apply to this functional, there are also some important differences, for instance, since $\theta_R(\ell)$ is \textit{monotone} in $\ell$ we would only expect Gaussian limits.

\subsubsection{A more general class of fields}
While we focus our study on the GFF, our arguments make use of only a relatively small subset of its properties. Suppose $f$ is a stationary Gaussian field on $\Z^d$ with covariance kernel $K(x) = \E[f(0)f(x)]$. 

\begin{assumption}
    \label{a:rv}
    There exists $\alpha \in (0,d)$ and $c > 0$ such that 
    \begin{equation}
        \label{e:k}
     K(x) \sim c |x|^{-\alpha} + O \big(|x|^{-\alpha-2} \big) 
     \end{equation}
    as $|x| \to \infty$. Moreover $K \ge 0$ and $K$ is invariant under coordinate reflection and permutation.
\end{assumption}

\begin{assumption}
\label{a:sd}
 There exists a decomposition
 \begin{equation}
 \label{e:iiddecomp} 
 f \stackrel{d}{=} \kappa \tilde{f} + \hat{f} 
 \end{equation}
 where $\kappa > 0$, $\tilde{f}$ is an i.i.d.\ field of standard Gaussian random variables, and $\hat{f}$ is an independent centred Gaussian field.
 \end{assumption}

Assumption \ref{a:sd} is equivalent to the spectral measure $\mu = \mathcal{F}[K]$ having a density $\rho$ satisfying $\inf_{x \in \mathbb{T}^d} \rho(x) > 0$,  where $\mathcal{F}$ denotes the Fourier transform and $\mathbb{T}^d$ the torus. It is standard that Assumption \ref{a:rv} (with $\alpha=d-2$) and Assumption \ref{a:sd} hold true for the GFF (see \cite[Theorem 4.3.1]{ll12} for the error bound in \eqref{e:k}).

\smallskip
In fact we believe that only \eqref{e:k} is crucial to the results, and even then this could likely be weakened to an appropriate condition on the singularity of $\mu$ at the origin, and/or we could include slowly varying factors. Note that we only use the error bound in \eqref{e:k} to handle the case $\mu'(\ell) = 0$. Assumption \ref{a:sd} plays an important technical role in our proof (e.g.\ in the `de-pinning' arguments in Section \ref{s:tae}) and it would be interesting to remove it, especially as a step towards adapting the arguments to smooth fields.

\smallskip
Next we introduce the notion of \textit{truncated arm decay} from percolation theory. For a random subset $E \subseteq \Z^d$ let $\mathrm{Arm}_R(E)$ be the event that $E$ contains a component which is bounded and includes a path connecting a neighbour of $0$ and $\partial \Lambda_R$. For $\ell \in \R$, let $T_\ell$ be the property that both $\{f > \ell\}$ and $ \{ f < \ell\}$ have \textit{(super-polynomial) truncated arm decay} in the sense that
\[ \lim_{R \to \infty} \frac{-\log\P[ \mathrm{Arm}_R(E) ]}{\log R} = \infty  \]
for both $E = \{f > \ell\}$ and $ E = \{ f < \ell\}$.  It is known that, for a wide class of fields, $T_\ell$ holds for all sufficiently large $|\ell|$ (see, e.g., \cite{ms24}), and it is expected that in general $T_\ell$ holds for all $\ell \notin \{-\ell_c,\ell_c\}$, although so far this has only been shown for the GFF \cite{dgrs23}.

\begin{theorem}
Suppose $f$ satisfies Assumption \ref{a:rv} for $\alpha = d-2$ and Assumption \ref{a:sd}. Then the conclusions of Theorems \ref{t:fluc1}-\ref{t:fluc2} and Theorem \ref{t:var} hold at all levels $\ell \in \mathcal{L}$, where $\mathcal{L} \subseteq \R$ is any open subset in which $T_\ell$ holds uniformly. Moreover the first item of Theorem \ref{t:var2} holds at every level, and the second item of Theorem \ref{t:var2} holds at every level for which $\mu'(\ell) \neq 0$.

In general, the measure associated to the Hermite distribution appearing as a limit in Theorem \ref{t:fluc2} will depend on $f$ (see Definition \ref{d:ros}).
\end{theorem}

\smallskip
Most of the proof actually goes through at any level satisfying $T_\ell$; uniformity is only needed in the proof of Proposition \ref{p:mudiff}. Without assuming uniformity, the conclusions of Theorems \ref{t:fluc1}, \ref{t:fluc2} and \ref{t:var} would hold with $\mu'(\ell)$ replaced by the expression on the r.h.s.\ of~\eqref{e:mudiff}.

\smallskip
If $f$ satisfies Assumption \ref{a:rv} for $\alpha \in (0,d)$ different from $d-2$, then we believe that broadly analogous results will hold. However for $\alpha < d-2$ certain extra effects may appear compared to the $\alpha \ge d-2$ case: in general we expect that Hermite distributions of all orders $2 \le m < d/\alpha$ may appear as possible limits, and additional boundary effects will appear in the limit for small enough $\alpha$. We leave the investigation of such limits for future work.

\subsection{Outline of the paper}
In Section \ref{s:ce} we develop the theory of chaos expansions for level-set functionals such as the cluster count. In Section \ref{s:sl} we show how percolation theoretic inputs imply that the components of the chaos expansion of the cluster count can be semi-localised. In Section \ref{s:cdf} we study properties of the cluster density functional $\mu(\ell)$, and connect its derivatives to the asymptotics of the chaotic components. In Section \ref{s:vb} we give general variance bounds that hold at all levels. In Section \ref{s:mr} we complete the proof of the main results. Appendix \ref{a:mhp} establishes basic properties of Gaussian vectors and multivariate Hermite polynomials that are used in Sections \ref{s:ce} and \ref{s:sl}, Appendix \ref{a:slaf} extends the classical theory of local additive functionals of stationary Gaussian fields to semi-local additive functionals, and Appendix \ref{a:kc} contains computations used in Sections \ref{s:sl} and \ref{s:mr} and Appendix \ref{a:slaf}.

\subsection{Acknowledgements}
Part of this work was carried out while S.M.\ was a Research Fellow at the University of Melbourne, supported by the Australian Research Council (ARC) Discovery Early Career Researcher Award DE200101467. The authors also benefited from a research visit of M.M.\ to the University of Melbourne supported by this award. The authors thank Illia Donhauzer, Rapha\"{e}l Lachi\`{e}ze-Rey, Matthias Schulte, Franco Severo, Hugo Vanneuville, and Igor Wigman for comments on an earlier version, and I.W.\ for pointers to the literature on chaos expansions for geometric functionals of Gaussian fields. We also thank two anonymous referees for comments and corrections which helped improve this work.


\medskip
\section{Chaos expansion for the cluster count}
\label{s:ce}

In this section we derive an expression for the chaos projections of smooth functionals, which we then extend to the cluster count by an approximation argument. Finally we use a similar approach to derive a variance formula for the tail of chaos projections of the cluster count. Rather than restrict our attention to the GFF, we shall generalise the set-up by working with arbitrary centred Gaussian vectors.

\smallskip
For the remainder of the section we fix a finite subset $D \subset \Z^d$ and a non-degenerate centred Gaussian vector $f$ on $D$ with covariance matrix $K$, and we drop these from our notation.

\subsection{Preliminaries: Chaos expansions, Wick products, diagram formula}
\label{ss:ce}
We begin by recalling some fundamental facts about chaos expansions and Wick products which can be found in \cite[Chapters~1-3]{jan97}. Let $H$ be a real Gaussian Hilbert space defined on a probability space $(\Omega,\mathcal{F},\P)$. For $m\geq 0$, let $\mathcal{P}_m(H)$ be the set of random variables that can be expressed as a real polynomial of degree at most $m$ in finitely many elements of $H$. The \emph{$m$-th homogeneous chaos of $H$}, denoted $H^{{:}m{:}}$, is defined as the projection of $\overline{\mathcal{P}_m(H)}$ onto the orthogonal complement of $\overline{\mathcal{P}_{m-1}(H)}$, where $\overline{\cdot}$ denotes the closure in $L^2(\P)$. The Wiener-It\^o chaos expansion states that any square-integrable function that is measurable with respect to $H$ has a unique expansion in terms of elements of the homogeneous chaoses:
\begin{displaymath}
    \bigoplus_{m=0}^\infty H^{{:}m{:}}=L^2(\Omega,\sigma(H),\P)
\end{displaymath}
where $\sigma(H)$ denotes the $\sigma$-algebra generated by $H$. In other words, if $Q_m$ denotes projection in $L^2(\P)$ onto $H^{{:}m{:}}$, then for all $X\in L^2(\Omega,\sigma(H),\P)$
\begin{displaymath}
    X=\sum_{m=0}^\infty Q_m[X]
\end{displaymath}
where convergence occurs in $L^2$. In particular $Q_0[X] = \E[X]$, $\E[Q_m[X]] = 0$ for every $m \ge 1$, and $\Var[X] = \sum_{m \ge 1} \Var[Q_m[X]]$.

Given (centred) Gaussian variables $X_1,\dots,X_n\in H$, we define the \emph{Wick product}
\begin{displaymath}
    \wick{X_1\cdots X_n}:=Q_n[X_1\cdots X_n].
\end{displaymath}
It can be shown that $ : \! \! X_1\cdots X_n \! \! :$ is a polynomial of order $n$ in the variables $X_1,\dots,X_n$; e.g. $: \! \!X_1 X_2\!\! : \, = X_1 X_2 - \E[X_1 X_2]$ and $: \! \! X_1 X_2 X_3 \! \! : \, = X_1 X_2 X_3 - \E[X_1 X_2] X_3 - \E[X_1 X_3] X_2 - \E[X_2 X_3] X_1 $. By definition of the homogeneous chaoses, if $m\neq n$ then
\begin{displaymath}
    \E[\wick{X_1\cdots X_n}\wick{Y_1\cdots Y_m}]=0.
\end{displaymath}
More generally, the moments of Wick products can be computed by means of a diagram formula which we now describe.

A \emph{complete Feynman diagram} labelled by a collection of random variables $X_1,\dots,X_n$ is a graph with $n$ vertices (the $i$-th vertex is identified with $X_i$) such that each vertex is contained in precisely one edge. Clearly a complete Feynman diagram can only exist if $n$ is even. Suppose that a complete Feynman diagram $\gamma$ has edges $\{X_{i_k},X_{j_k}\}$ for $k=1,\dots, n/2$. Then the \emph{value} of $\gamma$ is
\begin{displaymath}
    v(\gamma):=\prod_{k=1}^{n/2}\E[X_{i_k}X_{j_k}].
\end{displaymath}
\begin{theorem}[Diagram formula {\cite[Theorem~3.12]{jan97}}]
\label{t:df}
    Let $k,I\in\N$ and for each $i=1,\dots,I$, let $\mathcal{X}_i=\wick{X_{i,1}\cdots X_{i,k}}$ where $\{X_{i,j}\}_{i,j}$ are centred jointly Gaussian variables. Then
\begin{displaymath}
    \E\left[\mathcal{X}_1\dots \mathcal{X}_I\right]=\sum_{\gamma}v(\gamma)
\end{displaymath}
where the sum is taken over all complete Feynman diagrams $\gamma$ labelled by $\{X_{i,j}\}_{i,j}$ such that no edge joins any $X_{i_1,j_1}$ and $X_{i_2,j_2}$ where $i_1=i_2$.
\end{theorem}

For example, in the case that $I=2$, if $X_1,\dots,X_k,Y_1,\dots,Y_k$ are centred jointly Gaussian variables, then
\begin{displaymath}
    \E\left[\wick{X_1\cdots X_k}\wick{Y_1\cdots Y_k}\right]=\sum_{\sigma\in S_k}\prod_{i=1}^k\E\left[X_iY_{\sigma(i)}\right]
\end{displaymath}
where $S_k$ denotes the group of permutations of $\{1,\dots,k\}$ (see Figure \ref{f:df}).

\vspace{0.3cm}
\begin{figure}[h]
    \centering
      \hspace{1.5cm}
    \begin{subfigure}[t]{0.45\textwidth}
    \begin{tikzpicture}[scale=1.2, every node/.style={font=\small}]
    \draw[fill] (0,0)  circle (1pt);
    \node[left] at (0,0) {$X_{1,1}$};
    \draw[fill] (1,0)  circle (1pt);
    \node[left] at (1,0) {$X_{1,2}$};
    \draw[fill] (2,0)  circle (1pt);
    \node[left] at (2,0) {$X_{1,3}$};
    \draw[fill] (3,0)  circle (1pt);
    \node[left] at (3,0) {$X_{1,4}$};

    \draw[fill] (0,-0.5)  circle (1pt);
    \node[left] at (0,-0.5) {$X_{2,1}$};
    \draw[fill] (1,-0.5)  circle (1pt);
    \node[left] at (1,-0.5) {$X_{2,2}$};

    \draw[fill] (2,-0.5)  circle (1pt);
    \node[left] at (2,-0.5) {$X_{2,3}$};
    \draw[fill] (3,-0.5)  circle (1pt);
    \node[left] at (3,-0.5) {$X_{2,4}$};

    \draw[thick] (0,0)--(0,-0.5);
    \draw[thick] (1,0)--(1,-0.5);

    \draw[thick] (2,0)--(2,-0.5);
    \draw[thick] (3,0)--(3,-0.5);
    \end{tikzpicture}    
    \end{subfigure}
    \hspace{-0.5cm}
    \begin{subfigure}[t]{0.45\textwidth}
        \begin{tikzpicture}[scale=1.2, every node/.style={font=\small}]
    \draw[fill] (0,0)  circle (1pt);
    \node[left] at (0,0) {$X_{1,1}$};
    \draw[fill] (1,0)  circle (1pt);
    \node[left] at (1,0) {$X_{1,2}$};
    \draw[fill] (2,0)  circle (1pt);
    \node[left] at (2,0) {$X_{1,3}$};
    \draw[fill] (2.5,0)  circle (1pt);
    \node[right] at (2.5,0) {$X_{1,4}$};

    \draw[fill] (0,-0.5)  circle (1pt);
    \node[left] at (0,-0.5) {$X_{2,1}$};
    \draw[fill] (1,-0.5)  circle (1pt);
    \node[left] at (1,-0.5) {$X_{2,2}$};

    \draw[fill] (2,-0.5)  circle (1pt);
    \node[left] at (2,-0.5) {$X_{2,3}$};
    \draw[fill] (2.5,-0.5)  circle (1pt);
    \node[right] at (2.5,-0.5) {$X_{2,4}$};

    \draw[thick] (0,0)--(0,-0.5);
    \draw[thick] (1,0)--(1,-0.5);

    \draw[thick] (2,0)--(2.5,0);
    \draw[thick] (2,-0.5)--(2.5,-0.5);
    \end{tikzpicture}
    \end{subfigure}
    \caption{Two complete Feynman diagrams on the vertices $\{X_{i,j}\}$ in the case $I=2,k=4$; only the first diagram contributes to the diagram formula.}
    \label{f:df}
\end{figure}
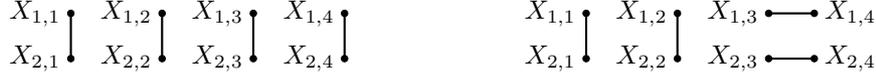

\subsection{Chaos expansion for smooth functionals}
We begin by establishing a chaos expansion for smooth functionals; this expansion is in terms of Wick products, and is different from previous approaches in the literature (see Remark \ref{r:alt}).

\begin{proposition}[Chaos expansion for smooth functionals]
\label{p:correlated_wce}
Let $\Phi:\R^D \to\R$ be a smooth function such that $\Phi$ and its derivatives of all orders have at most polynomial growth at infinity. Then $\Phi$ has the chaos expansion
\begin{displaymath}
\Phi(f) =\E[\Phi(f)]+\sum_{m=1}^\infty Q_m[\Phi(f)]
\end{displaymath}
where convergence occurs in $L^2$ and
\begin{equation}
        \label{e:qm}
        Q_m[\Phi(f)]=\frac{1}{m!}\sum_{x_1,\dots,x_m \in D}\E[\partial_{x_1}\dots\partial_{x_m} \Phi(f)]\wick{f(x_1)\cdots f(x_m)}.
    \end{equation}
\end{proposition}
\begin{proof}
    Let $\widetilde{Q}_m[\Phi(f)]$ denote the right-hand side of \eqref{e:qm}. By definition of the Wick product, it is clear that $\widetilde{Q}_m[\Phi(f)]$ is in the $m$-th homogeneous chaos of the Gaussian Hilbert space generated by $f$. Since the products $\wick{f(x_1)\cdots f(x_m)}$ for $x_1,\dots,x_m\in D$ form a total set for this chaos, it is enough to show that
    \begin{equation}\label{e:correlated_wce}
        \E \big[\Phi(f)\wick{f(x_1)\cdots f(x_m)} \big] = \E \big[\widetilde{Q}_m[\Phi(f)]\wick{f(x_1)\cdots f(x_m)}\big]
    \end{equation}
    for all $m\in\N$ and all $x_1,\dots,x_m\in D$.

    Let $K^{1/2}$ be the symmetric square root of $K$ and let $(Z_x)_{x \in D}$ be a standard Gaussian random vector such that $f=K^{1/2}Z$. By linearity of the Wick product
    \begin{displaymath}
        \E \big[\Phi(f)\wick{f(x_1)\cdots f(x_m)} \big]=\sum_{y_1,\dots,y_m \in D }\left(\prod_{l=1}^m(K^{1/2})_{x_l,y_l}\right)\E[\Phi(f)\wick{Z_{y_1}\cdots Z_{y_m}}].
    \end{displaymath}
    The term $\wick{Z_{y_1}\cdots Z_{y_m}}$ can be identified as a product of univariate Hermite polynomials evaluated at $Z$ \cite[Theorem~3.21]{jan97}. Then using an integration by parts property of Hermite polynomials \cite[Definitions 1.2.2 and 1.4.1]{np12} and linearity once more
    \begin{align*}
        \E \big[\Phi(K^{1/2}Z)\wick{Z_{y_1}\cdots Z_{y_m}} \big]&=\E\left[\frac{\partial}{\partial Z_{y_1}}\dots \frac{\partial}{\partial Z_{y_m}}\Phi(K^{1/2}Z)\right]\\
        &=\sum_{w_1,\dots w_m \in D}\left(\prod_{l=1}^m(K^{1/2})_{w_l,y_l}\right)\E\left[\partial_{w_1}\dots\partial_{w_m} \Phi(f)\right].
    \end{align*}
    Substituting this into the previous equation and using symmetry of $K^{1/2}$, we have
    \begin{equation}\label{e:correlated_wce2}
        \E \big[\Phi(f)\wick{f(x_1)\cdots f(x_m)} \big]=\sum_{w_1,\dots w_m \in D}\left(\prod_{l=1}^m K_{x_l,w_l}\right)\E\left[\partial_{w_1}\dots\partial_{w_m}\Phi(f)\right].
    \end{equation}
    Turning to the right hand side of \eqref{e:correlated_wce}, by the diagram formula (Theorem~\ref{t:df})
    \begin{align*}
        \E \big[\widetilde{Q}_m[\Phi(f)]\wick{f(x_1)\cdots f(x_m)} \big]&=\frac{1}{m!}\sum_{\sigma\in S_m}\sum_{w_1,\dots,w_m \in D}\left(\prod_{l=1}^m K_{x_l,w_{\sigma(l)}}\right) \E[\partial_{w_1}\dots\partial_{w_m} \Phi(f)].
    \end{align*}
    Since $w_1,\dots,w_m$ are summed over all indices, we see that the sum above takes the same value for each $\sigma\in S_m$ and hence the overall expression matches that of \eqref{e:correlated_wce2}. This verifies \eqref{e:correlated_wce} and so proves the proposition.
\end{proof}

\begin{remark}
\label{r:alt}
This chaos expansion can alternatively be expressed in terms of \textit{multivariate Hermite polynomials} for correlated Gaussian vectors \cite{rahman2017wiener}. For a centred Gaussian vector $X$ with non-degenerate covariance matrix $K$, and a multi-index $\alpha$, the multivariate Hermite polynomial of order $\alpha$ is defined as
\begin{equation} 
\label{e:mhp}
    H^\alpha_X(x) := H^\alpha_K(x) := (-1)^{|\alpha| }\frac{\partial^\alpha\varphi_X(x)}{\varphi_X(x)}
\end{equation}
where $\varphi_X(x)$ denotes the density of $X$. Then the $m$-th term of the chaos expansion for $\Phi(f)$ has the alternative expression
\begin{equation}\label{e:WienerChaosAlternative}
    Q_m[\Phi(f)] =\frac{1}{m!} \sum_{x_1,\dots,x_m \in D} \sum_{y_1,\dots,y_m \in D} \E[\partial^{\alpha^y} \Phi(f)]\left(\prod_{l=1}^m K_{x_l,y_l}\right) H_f^{\alpha^x}(f)
\end{equation}
where $\alpha^x$ denotes the multi-index corresponding to $x=(x_1,\dots,x_m)$. To derive this, one can write $f=K^{1/2}Z$, compute the chaos expansion of $\Phi$ in terms of the orthogonal variables $Z$ using ordinary (univariate) Hermite polynomials, and then use the chain rule to convert back to derivatives with respect to $f$. However \eqref{e:WienerChaosAlternative} seems to be of little use for our purposes: the inner product of multivariate Hermite polynomials has a complicated expression, and moreover at every point it depends on the covariance of the entire vector.
\end{remark}

\subsection{Chaos expansion for level-set functionals}
Since the cluster count is not a smooth functional, we will not be able to apply Proposition \ref{p:correlated_wce} directly. As such, we next provide an appropriate interpretation of the derivatives appearing in \eqref{e:qm} for functionals that depend only on the excursion set $\{f > \ell\}$. For this we introduce the notion of \textit{pivotal intensities}.

\smallskip

For the remainder of the section, we work with an arbitrary function $\Xi: \mathcal{P}(D) \to \R$ defined on the subsets of $D$, and let $\Xi(f) = \Xi(\{f > 0\})$. In particular, if $D = \Lambda_R$ and $\Xi(E)$ is the sum of the number of components of $E$ and $D \setminus E$ that do not intersect $\partial D$, then $\Xi(f-\ell) = N_R(\ell)$. 

\subsubsection{Pivotal events and intensities}
For $y\in D$, the \textit{discrete derivative of $\Xi$ at $y$} is the function $d_y\Xi:\mathcal{P}(D)\to\R$ given by
\begin{displaymath}
    d_y\Xi(E)=\Xi(E\cup\{y\})-\Xi(E\setminus\{y\}).
\end{displaymath}
For $y \in D^m$, we let $\underline{y}$ denote the unique subvector $(y'_1,\ldots,y'_n)$ of all distinct elements of $y$ ordered in some fixed way (say, lexicographically), and write $d_{\underline{y}}$ to denote the iterated derivative $d_{y'_1}\dots d_{y'_n}$. Note that $d_{\underline{y}}$ only depends on $E \setminus \{y_1^\prime,\dots,y_n^\prime\}$, and does not depend on the order in which the derivatives are applied.

\begin{definition}(Pivotal events)
\label{d:piveve}
For $y \in D^m$ and $\sigma\in\R$, we say that a configuration $E\subseteq D$ is \emph{$\sigma$-pivotal at $\underline{y}$} if
\begin{displaymath}
d_{\underline{y}} \Xi(E)=\sigma.
\end{displaymath}
We define $\mathrm{Piv}(\underline{y},\sigma)$ to be the set of all such configurations. We emphasise that $\mathrm{Piv}(\underline{y},\sigma)$ is empty for all but finitely many values of $(\underline{y},\sigma)$. See Figure \ref{f:piv} for an illustration in the case of the cluster count.
\end{definition}

In a slight abuse of notation, given a function $g:D \to\R$ we write $d_{\underline{y}} \Xi(g)$ to abbreviate $d_{\underline{y}} \Xi(\{g > 0\} )$ and $g\in\mathrm{Piv}(\underline{y},\sigma)$ to mean that $\{g >0\} \in\mathrm{Piv}(\underline{y},\sigma)$. 

\begin{figure}[ht]
    \centering
    \hspace{0.5cm}
    \begin{subfigure}[t]{0.45\textwidth}
        \centering
        \begin{tikzpicture}
        \draw[black!20,step=0.5] (0.2,0.2) grid (4.3,4.3);
        \node[] at (4.5,3.5)  {$\mathbb{Z}^d$};
        \draw[gray, dashed] (0.35,0.35) rectangle (4.15,4.15);
        \draw[gray] (4.15,2.75) -- (4.8,2.5);
        \node[right] at (4.8,2.5) {$\Lambda_R$};
        \foreach \x in {(1,1),(1.5,1),(2,1),(1.5,1.5),(1.5,2),(2,2),(2.5,2),(2.5,2.5),(2,2.5),(1.5,3.5),(1.5,3)} {
        \fill \x circle (1.5pt);
        }
        \draw (1.5,2) circle (2pt);
        \draw (3,3.5) circle (2pt);
        \draw (2.5,2) circle (2pt);
        \node[above left] at (1.5,2) {$y_1$};
        \node[above left] at (3,3.5) {$y_2$};
        \node[below right] at (2.5,2) {$y_3$};
        \draw[thick] (1,1)--(1.5,1)--(2,1) (1.5,1)--(1.5,2)--(2.5,2)--(2.5,2.5)--(2,2.5)--(2,2);
        \draw[thick] (1.5,3.5)--(1.5,3);
    \end{tikzpicture}
    \end{subfigure}
    \;\;
    \hspace{-0.2cm}
    \begin{subfigure}[t]{0.45\textwidth}
        \centering
        \begin{tikzpicture}
        \draw[black!20,step=0.5] (0.2,0.2) grid (4.3,4.3);
        \node[] at (4.5,3.5)  {$\mathbb{Z}^d$};
        \draw[gray, dashed] (0.35,0.35) rectangle (4.15,4.15);
        \draw[gray] (4.15,2.75) -- (4.8,2.5);
        \node[right] at (4.8,2.5) {$\Lambda_R$};
        \foreach \x in {(1.5,1.5),(1.5,2),(1.5,2.5),(2,2.5),(2.5,2.5),(3,2.5),(2.5,3),(3,3),(3.5,2.5),(3.5,2),(3.5,1.5),(3,1.5),(2.5,1.5),(2.5,1),(2,1),(2,1.5),(3,1),(1.5,3.5),(2,3.5)} {
        \fill \x circle (1.5pt);
        }
          \draw (3.5,2.5) circle (2pt);
        \draw (1.5,1.5) circle (2pt);
        \node[above right] at (3.5,2.5) {$y_1$};
        \node[below left] at (1.5,1.5) {$y_2$};
        \draw[thick] (1.5,1.5)--(1.5,2.5)--(3.5,2.5)--(3.5,1.5)--(1.5,1.5) (2.5,2.5)--(2.5,3)--(3,3)--(3,2.5) (2,1.5)--(2,1)--(3,1)--(3,1.5) (2.5,1)--(2.5,1.5) ;
         \draw[thick] (1.5,3.5) -- (2,3.5);
    \end{tikzpicture}
    \end{subfigure} 
    \caption{Illustration of pivotal configurations for the cluster count in $\Lambda_R$ (i.e.\ $\Xi(E)$ is the sum of the number of components of $E$ and $\Lambda_R \setminus E$ that do not intersect $\partial \Lambda_R$, where $E$ is the set of black vertices). Left: The configuration is $(-1)$-pivotal at $y_1$, $1$-pivotal at $y_2$, and not pivotal at $y_3$. Right: The configuration is $1$-pivotal at $(y_1,y_2)$.}
   \label{f:piv}
\end{figure}
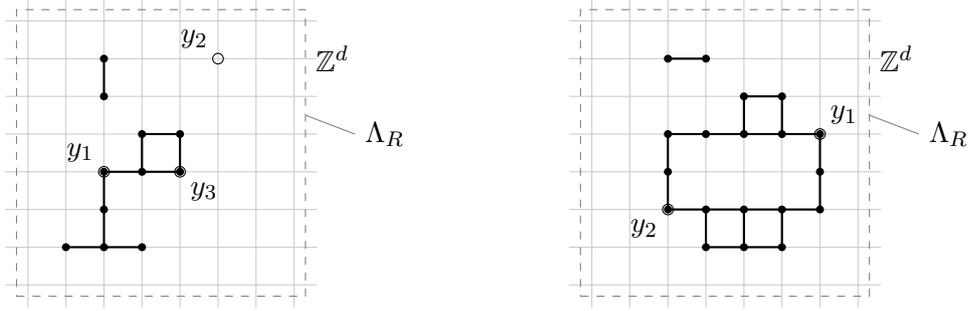

We next define the pivotal intensities. These will depend on an arbitrary fixed vector $\nu:D\to\R$; we will mostly consider the case that $\nu\equiv\ell$. For $y = (y_1,\ldots,y_m) \in D^m$, let $\alpha^y \in \N_0^D$ be its associated multi-index (that records the number of times each distinct $y_i \in D$ appears in $y$), and let $\tilde{\alpha}^y$ be defined by $\tilde{\alpha}^y_i=\max\{\alpha^y_i-1,0\}$. We extend our definition of $d_y\Xi(f-\nu)$ to include repeated points:
    \begin{equation}
    \label{e:itex}
d_y\Xi(f-\nu) := d_{\underline{y}} \Xi(f-\nu) H^{\tilde{\alpha}^y}_f(f)  
    \end{equation}
where $H$ is the multivariate Hermite polynomial defined in \eqref{e:mhp}. This coincides with the previous definition since if $y_1,\ldots,y_m$ are distinct then $\tilde{\alpha}^y= 0$ and $H^{\tilde{\alpha}^y}_f = 1$. We also adopt the notation $g(y):=(g(y_1),\dots,g(y_m))$ for functions $g:D\to\R$.

\begin{definition}[Pivotal intensities]
\label{d:pivint}
For $y \in D^m$, the \textit{pivotal intensity at $y$ (with respect to level $\nu$)} is
 \begin{align}
 \label{e:pivint}
\nonumber P(y):=P(\nu;y):=&\E\left[d_y \Xi(f-\nu) \middle|f(\underline{y})=\nu(\underline{y})\right]\varphi_{f(\underline{y})}(\nu(\underline{y})) \\
= &\sum_{\sigma \neq 0}\sigma \E\left[ \id_{f-\nu\in\mathrm{Piv}(\underline{y},\sigma)}  H^{\tilde{\alpha}^y}_f(f) \middle|f(\underline{y})=\nu(\underline{y}) \right] \varphi_{f(\underline{y})}(\nu(\underline{y}))
\end{align}
 where $\varphi_{f(\underline{y})}$ denotes the density of $f(\underline{y})$. When $\nu\equiv\ell$, we denote this intensity as $P(\ell;y)$.
\end{definition}

We emphasise that in \eqref{e:pivint} the Hermite polynomial $ H^{\tilde{\alpha}^y}_f$ is defined with respect to (the covariance matrix of) the unconditioned field $f$, but its argument has the law of the conditioned field $(f | f(\underline{y})=\nu)$.

\begin{remark}
\label{r:pivint}
Since $H^{\tilde{\alpha}^y}_f$ is a polynomial in $\lvert D\rvert$ variables, it can be challenging to directly analyse its limiting behaviour as the domain size increases. A useful equivalent expression for the pivotal intensity is 
\begin{equation}
\label{e:pivalt}
P(y)  =  \E\left[ d_{\underline{y}}\Xi(f-\nu)  H^{\tilde{\alpha}^y}_{f( \underline{y} ) | f(y^c) = 0 } \big(f(\underline{y})-\E[f(\underline{y})|f(y^c)]\big) \middle|f(\underline{y})=\nu(\underline{y}) \right]\varphi_{f(\underline{y})}(\nu(\underline{y}))
\end{equation}
where $y^c= D\setminus y$ (identifying $y$ as a subset of $D$). This can be further simplified: by Gaussian regression,
\begin{displaymath}
    \E[f(\underline{y})|f(y^c)]=\Sigma_{\underline{y} \to y^c}\Sigma_{y^c}^{-1}f(y^c)
\end{displaymath}
where $\Sigma_{X \to Y}$ denotes the covariance matrix between $X$ and $Y$. We denote by $|\underline{y}|$ the cardinality of $\underline{y}$ (viewed as a subset of $D$). Note that $H^{\tilde{\alpha}^y}_{f( \underline{y} ) | f(y^c) = 0 }$ is a polynomial in $\lvert\underline{y}\rvert$ variables (although its argument in \eqref{e:pivalt} depends on $f$ everywhere). The equivalence of \eqref{e:pivint}--\eqref{e:pivalt} follows since, again by Gaussian regression,
\begin{align*}
   H^{\tilde{\alpha}^y}_{f(\underline{y}), f(y^c)}(\underline{x},x^c)  \varphi_{f(\underline{y}),f(y^c)}(\underline{x},x^c)   & = (-1)^{\lvert\tilde{\alpha}^y\rvert}\partial^{\tilde{\alpha}^y} \varphi_{f(\underline{y}),f(y^c)}(\underline{x},x^c)   \\ 
   & =  (-1)^{\lvert\tilde{\alpha}^y\rvert} \partial^{\tilde{\alpha}^y} \varphi_{f(\underline{y})|f(y^c) = 0}(\underline{x} - \Sigma_{\underline{y} \to y^c}\Sigma_{y^c}^{-1} x^c) \varphi_{f(y^c)}(x^c)   \\
   & =  H^{\tilde{\alpha}^y}_{f( \underline{y} ) | f(y^c) = 0 }(\underline{x} - \Sigma_{\underline{y} \to y^c}\Sigma_{y^c}^{-1} x^c)\varphi_{f(\underline{y}),f(y^c)}(\underline{x}, x^c)  .  
\end{align*}

\end{remark}

Our interest in pivotal intensities lies in their interpretation as the derivatives of $\E[\Xi(f-\ell)]$. We prove the following result later in the section:

\begin{proposition}\label{p:derpivint}
For $y_1,\ldots,y_m \in D$, 
\[ P(\nu;y_1,\ldots,y_m) = \partial_{y_1} \cdots \partial_{y_m} \E[\Xi(f-\nu)]    \] 
\end{proposition}

For later use we state a bound on the pivotal intensities that is uniform in the size of the domain $D$. Let $\| d \Xi\|_\infty := \max_{y, E \subseteq D} |d_y \Xi(E)|$, and let $\lambda_{\textrm{min}} > 0$ denote the smallest eigenvalue of $K$.

\begin{lemma}
\label{l:bounds1} 
For $y_1,\ldots,y_m \in D$,
\[  \lvert d_{\underline{y}}\Xi(E)\rvert\leq 2^{m-1} \| d \Xi\|_\infty  \quad \text{and} \quad | P(y_1,\ldots,y_m) | \le \| d \Xi\|_\infty e^{c m} \sqrt{m!}  \]
where $c > 0$ depends only on $\lambda_{\mathrm{min}}$, $|\underline{y}|$, $\|\nu\|_\infty$, and $\|K\|_\infty$.
\end{lemma}
\begin{proof}
Defining $a_m$ as the supremum of $\lvert d_y\Xi(E)\rvert$ over distinct $y_1,\dots,y_m\in D$ and $E \subseteq D$, by the triangle inequality
    \begin{displaymath}
        a_{m+1}\leq\sup_{y_1,\dots,y_{m+1}}\sup_{E} \ \ \lvert d_{y_1}\dots d_{y_m}\Xi(E\cup\{y_{m+1}\})\rvert+\lvert d_{y_1}\dots d_{y_m}\Xi(E\setminus\{y_{m+1}\})\rvert\leq 2a_m,
    \end{displaymath}
    and iterating we obtain $\lvert d_{\underline{y}}\Xi(E)\rvert\leq 2^{m-1} \| d \Xi\|_\infty$. We also have
    \begin{equation}\label{e:bounds10}
        \varphi_{f(\underline{y})}(\nu)\leq \left(\det\mathrm{Cov}[f(\underline{y})]\right)^{-1/2}\leq \max\{1,\lambda_{\textrm{min}}^{-m/2}\} ,
    \end{equation}
    where we used that the smallest eigenvalue of $\mathrm{Cov}[f(\underline{y})]$ is at least $\lambda_{\textrm{min}}$ (see Lemma \ref{l:lambdamin}). Combining these observations with Remark \ref{r:pivint}, it remains to show that
    \begin{equation}
    \label{e:bounds11}
        \E\left[ \big| H^{\tilde{\alpha}^y}_{f( \underline{y} ) | f(y^c) = 0 } \big(f(\underline{y})-\Sigma_{\underline{y} \to y^c}\Sigma_{y^c}^{-1}f(y^c)\big)  \big| \middle|f(\underline{y})=\nu(\underline{y}) \right] \le e^{c_1 m} \sqrt{m!}
    \end{equation}
    where $c_1 > 0$ depends only on $\lambda_{\textrm{min}}$, $|\underline{y}|$, $\|\nu\|_\infty$, and $\|K\|_\infty$.

     Abbreviate $k = |\underline{y}|$. Applying the pointwise bound on Hermite polynomials in Proposition~\ref{p:mhpmain1}, and since $|\tilde{\alpha}^y|\le m$, the left-hand side of \eqref{e:bounds11} is bounded by
    \[  k^{m/2} \sqrt{m!}  \times \E \Big[ e^{c_2 \sqrt{m} \big( \|Y\|_2 + 1 \big) } \Big]  \]
    where $Y$ is the $k$-dimensional Gaussian vector $f(\underline{y})-\Sigma_{\underline{y} \to y^c}\Sigma_{y^c}^{-1}f(y^c)$ conditioned on $\{f(\underline{y})=\nu(\underline{y})\}$, and $c_2 > 0$ depends only on $\lambda_{\textrm{min}}$. Since we have, for any $s > 0$
   \[ \E[ e^{s \|Y\|_2 }] \le \E[ e^{s k \max_i |Y_i|} ] \le \sum_i  \E[ e^{s k |Y_i|} ] \le 2 k \max_i e^{ s k \lvert\E[Y_i]\rvert + (s^2 k^2 / 2) \textrm{Var}[Y_i] }  \]
 it remains to bound $\lvert\E[Y_i]\rvert$ and $\textrm{Var}[Y_i]$ by constants.
 
By Gaussian regression
\begin{align*}
    \E[Y_i]= \nu_i - \Sigma_{\underline{y}\to y^c }\Sigma^{-1}_{y^c}\Sigma_{\underline{y}\to y^c}^T\Sigma_{\underline{y}}^{-1}\nu(\underline{y}).
\end{align*}
Moreover, by positive definiteness of the Schur complement of a covariance matrix, the product of the first three matrices above is dominated by $\Sigma_{\underline{y}}$, and hence the above expression is bounded by a constant depending only on $\lambda_{\mathrm{min}}$, $|\underline{y}|$, $\|\nu\|_\infty$, and $\|K\|_\infty$. Finally, again by Gaussian regression and Lemma~\ref{l:lambdamin}
\[ \textrm{Var}[Y_i]  \le \textrm{Var}[\Sigma_{\underline{y}_i \to y^c}\Sigma_{y^c}^{-1}f(y^c)] = \Sigma_{\underline{y}_i \to y^c}\Sigma_{y^c}^{-1}\Sigma_{\underline{y}_i \to y^c}^T \le \lvert\underline{y}\rvert\|K\|_\infty^2\lambda_{\mathrm{min}}^{-1} \]
which completes the proof. 
\end{proof}

\subsubsection{Chaos expansion}

We can now formulate the chaos expansion for a level-set functional:

\begin{theorem}[Chaos expansion for level-set functionals]\label{t:cluster_wce}
  The functional $\Xi(f-\nu)$ has the chaos expansion
  \[   \Xi(f-\nu)=\E[\Xi(f-\nu)]+\sum_{m=1}^\infty Q_m[\Xi(f-\nu) ] \]
  where convergence occurs in $L^2$,
    \begin{displaymath}
        Q_m[\Xi(f-\nu)]= \frac{1}{m!}\sum_{x_1,\dots,x_m\in D}P(\nu;x_1,\ldots,x_m)\wick{f(x_1)\cdots f(x_m)},
    \end{displaymath}
    and $P(\nu;x_1,\ldots,x_m)$ is as in Definition \ref{d:pivint}.
\end{theorem}

Comparing to Proposition \ref{p:correlated_wce}, and given Proposition \ref{p:derpivint}, the proof of Theorem \ref{t:cluster_wce} essentially consists of justifying the formal exchange of expectation and derivatives
\[\E[ \partial_{y_1} \dots \partial_{y_m}  \Xi(f-\nu)] = \partial_{y_1} \dots \partial_{y_m} \E[   \Xi(f-\nu)]  \]
for level-set functionals. We do this by working with a suitable smooth approximation:

\begin{lemma}\label{l:derivatives_OU}
    There exists a collection $(\Xi^t)_{t>0}$ of smooth functions $\Xi^t : \R^D \to \R$ such that $\Xi^t(f-\nu)\to \Xi(f-\nu)$ in $L^2$ as $t \to 0$, and, for any multi-index $\alpha\in\N_0^D$,
    \begin{displaymath}
        \E[\partial^\alpha  \Xi^t(f-\nu)]=\partial^\alpha\E[ \Xi^t(f-\nu)]<\infty \, ,  \quad t > 0,
    \end{displaymath}
and
    \begin{displaymath}
        \lim_{t \to 0}\partial^\alpha\E[ \Xi^t(f-\nu)]=\partial^\alpha\E[\Xi(f-\nu)].
    \end{displaymath}
\end{lemma}
\begin{proof}
Since $\Xi(f)$ is a finite linear combination of indicators of quadrants, it is enough to prove the lemma for one quadrant, e.g., the case $\Xi(f) = \id_{f \in Q}$, where $Q = \cap_{x\in D}\{s \in \R^D : s(x) < \nu(x) \}$. We then define $\Xi^t: s \mapsto \P[ s + t Z \in Q] $ where $Z = (Z_x)_{x\in D}$ is an i.i.d.\ Gaussian vector, and the claims follow.
\end{proof}

\begin{proof}[Proof of Theorem \ref{t:cluster_wce}]
Recall that $Q_m$ denotes projection onto the $m$-th homogeneous chaos (of the Gaussian Hilbert space generated by $f$). Fixing $t>0$, we apply Proposition~\ref{p:correlated_wce} to $\Xi^t(f-\nu)$ and find that
    \begin{displaymath}
    Q_m[\Xi^t(f-\nu)]=\frac{1}{m!}\sum_{x_1,\dots,x_m\in D}\E[\partial_{x_1}\dots\partial_{x_m} \Xi^t(f-\nu)]\wick{f(x_1)\cdots f(x_m)}.
    \end{displaymath}
    Then by Lemma~\ref{l:derivatives_OU}
    \begin{align*}
     Q_m[\Xi^t(f-\nu)]\to\frac{1}{m!}\sum_{x_1,\dots,x_m\in D}\partial_{x_1}\dots\partial_{x_m}\E[\Xi(f-\nu)]\wick{f(x_1)\cdots f(x_m)}
    \end{align*}
    in $L^2$ as $t\to 0$. Since $\Xi^t(f-\nu)\to \Xi(f-\nu)$ in $L^2$ as $t\to 0$ and $Q_m$ is just projection onto a closed subspace of $L^2$, we conclude that as $t \to 0$
    \begin{displaymath}
        Q_m[\Xi^t(f-\nu)]\overset{L^2}{\to}Q_m[\Xi(f-\nu)].
    \end{displaymath}
    Combining these observations with Proposition~\ref{p:derpivint} proves the result.
\end{proof}

\subsubsection{Pivotal intensities as derivatives}
Towards proving Proposition \ref{p:derpivint}, we consider the effect on $\E[\Xi(f-\ell)]$ of an arbitrary perturbation of the field:

\begin{lemma}\label{l:pivotal_derivative}
For $g:D\to\R$,
\begin{displaymath}
    \frac{d}{dt} \bigg|_{t=0}\E \big[\Xi( f - \nu +tg  ) \big] = \sum_{y\in D}g(y)\E\left[d_y\Xi(f-\nu)\middle|f(y)=\nu(y)\right]\varphi_{f(y)}(\nu(y)).
\end{displaymath}
\end{lemma}

\begin{proof}
    For $y\in D$ and $t>0$ we define the events
    \begin{align*}
        A_t(y)&=\begin{cases}
            \{(f-\nu)(y)<0\leq (f-\nu+tg)(y)\} &\text{if }g(y)>0,\\
            \{(f-\nu+tg)(y)<0\leq (f-\nu)(y)\} &\text{if }g(y)\leq0,
        \end{cases}
        &A_t&=\bigcup_{y\in D}A_t(y),\\
        A_t^\ast(y)&=A_t(y)\cap\bigcap_{z\in D\setminus\{y\}}A_t^c(z), 
        &A^\ast_t&=\bigcup_{y\in D}A_t^\ast(y).
    \end{align*}
    Since $\Xi$ depends only on the excursion set of the field, we observe that $\Xi(f-\nu+tg)\neq\Xi(f-\nu)$ implies that $A_t$ occurs. We also note that since $f$ is non-degenerate, for any $z\neq y$ the vector $(f(z),f(y))$ has a (Gaussian) density which is bounded on a neighbourhood of $(\nu(z),\nu(y))$. It follows that
    \begin{displaymath}
        \lim_{t\searrow0}\frac{\P(A_t(y)\cap A_t(z))}{t}=0\qquad\text{and hence}\qquad\lim_{t\searrow0}\frac{\P(A_t\setminus A_t^\ast)}{t}=0.
    \end{displaymath}
    Using these two observations, and the fact that $\Xi$ is bounded, we have
    \begin{equation}\label{e:pivotal_derivative}
    \begin{aligned}
        \frac{d}{dt}\bigg|_{t=0}\E \big[\Xi(f-\nu+tg)\big]&        =\lim_{t\searrow 0}\frac{1}{t}\E\big[(\Xi(f-\nu+tg)-\Xi(f-\nu))\ind_{A_t}\big]\\
        &=\lim_{t\searrow 0}\frac{1}{t}\E\big[(\Xi(f-\nu+tg)-\Xi(f-\nu))\ind_{A_t^\ast}\big]\\
        &=\sum_{y\in D}\lim_{t\searrow 0}\frac{1}{t}\E\big[(\Xi(f-\nu+tg)-\Xi(f-\nu))\ind_{A_t^\ast(y)}\big].
    \end{aligned}
    \end{equation}
    By definition of $A_t^\ast(y)$
    \begin{displaymath}
        (\Xi(f-\nu+tg)-\Xi(f-\nu))\ind_{A_t^\ast(y)}=\mathrm{sgn}_{g(y)}d_y\Xi(f-\nu)\ind_{A_t^\ast(y)}
    \end{displaymath}
    where $\mathrm{sgn}_{g(y)}$ denotes the sign of $g(y)$. Therefore conditioning on the value of $f(y)$, and assuming $g(y)>0$, we have
    \begin{align*}
        \lim_{t\searrow0}\frac{1}{t}&\E\left[(\Xi(f-\nu+tg)-\Xi(f-\nu))\ind_{A_t^\ast(y)}\right]\\
        &\qquad\qquad\qquad=\lim_{t\searrow0}\frac{1}{t}\mathrm{sgn}_{g(y)}\E\left[d_y\Xi(f-\nu)\ind_{A_t^\ast(y)}\right]\\
        &\qquad\qquad\qquad=\lim_{t\searrow0}\frac{1}{t}\mathrm{sgn}_{g(y)}\E\left[d_y\Xi(f-\nu)\ind_{A_t(y)}\right]\\
        &\qquad\qquad\qquad=\lim_{t\searrow0}\frac{1}{t}\mathrm{sgn}_{g(y)}\int_{(\nu-tg)(y)}^{\nu(y)}\E\left[d_y\Xi(f-\nu)\middle| f(y)=u\right]\varphi_{f(y)}(u)\;du\\
        &\qquad\qquad\qquad=g(y)\E\left[d_y\Xi(f-\nu)\middle| f(y)=\nu(y)\right]\varphi_{f(y)}(\nu(y)).
    \end{align*}
    In the case that $g(y)\leq 0$, the above calculation is valid provided we integrate over the interval $[\nu(y),(\nu-tg)(y)]$ instead. Substituting this into \eqref{e:pivotal_derivative} and summing over $y$ yields the statement of the lemma.
\end{proof}

\begin{proof}[Proof of Proposition \ref{p:derpivint}]
In the case that $y_1,\ldots,y_m$ are distinct, we argue by induction. Let $\underline{y}_{m-1} = (y_1,\ldots,y_{m-1})$, and  let $\tilde{\P}$ be the probability measure on $f$ under the conditioning $f(\underline{y}_{m-1})=\nu(\underline{y}_{m-1})$. Since $f$ is non-degenerate by assumption, it follows that the Gaussian field induced by $\tilde{\P}$ is non-degenerate on $D\setminus\{y_1,\dots,y_{m-1}\}$. Then by Lemma~\ref{l:pivotal_derivative} applied to $d_{\underline{y}_{m-1}}\Xi$ (viewed as a function on $\mathcal{P}(D\setminus\underline{y}_{m-1})$) and the inductive assumption, we have
    \begin{align*}
        \partial_{y_{1}}\dots\partial_{y_m}\E[\Xi(f-\nu)]&=\partial_{y_{m}}P(\nu;y_1,\dots,y_{m-1})\\
        &=\frac{d}{dt}\bigg|_{t=0}\tilde{\E}\left[d_{\underline{y}_{m-1}}\Xi(f-\nu+t\ind_{y_{m}})  \right]\varphi_{f(\underline{y}_{m-1})}(\nu(\underline{y}_{m-1}))\\
        &=\tilde{\E}\left[d_{\underline{y}}\Xi(f-\nu)\middle|f(y_{m})=\nu(y_m)\right]\\
        &\qquad\qquad\qquad\qquad\times\varphi_{f(y_{m})|f(\underline{y}_{m-1})}(\nu(y_m) | \nu(\underline{y}_{m-1}) )\varphi_{f(\underline{y}_{m-1})}(\nu(\underline{y}_{m-1}))\\
        &=\E\left[d_{\underline{y}}\Xi(f-\nu)\middle|f(\underline{y})=\nu(\underline{y})\right]\varphi_{f(\underline{y})}(\nu(\underline{y}))\\
        &=P(y_1,\dots,y_{m}).
    \end{align*}
    To extend to the general case, we have 
\begin{align*}
   \partial_{y_{1}}\dots\partial_{y_m}\E[\Xi(f-\nu)] &= \partial^{\tilde{\alpha}^y} \Big(\E\left[d_{\underline{y}} \Xi(f-\nu) \middle| f(\underline{y}) = \nu(\underline{y}) \right]\varphi_{f(\underline{y})}(\nu(\underline{y})\Big) \\
   &= (-1)^{\lvert\tilde{\alpha}^y\rvert}\partial^{\tilde{\alpha}^y}_u\Big( \int d_{\underline{y}}\Xi(\underline{x}-\nu(y^c))\varphi_{f(\underline{y}),f(y^c)}(u,\underline{x}) \;d\underline{x}\Big)\Big|_{u=\nu(\underline{y})}.
   \end{align*}
   On the other hand, by the definition of Hermite polynomials
   \begin{align*}
  P(\nu;y_1,\ldots,y_m) & :=  \E \left[d_{\underline{y}}\Xi(f-\nu)
H^{\tilde{\alpha}^y}_f(f)  \middle| f(\underline{y})  = \nu(\underline{y}) \right]\varphi_{f(\underline{y})}(\nu(\underline{y}))  \\
&  = \int d_{\underline{y}}\Xi(\underline{x}-\nu(y^c)) H^{\tilde{\alpha}^y}_{ f(\underline{y}), f(y^c)}(\nu(\underline{y}),\underline{x})  \varphi_{f(\underline{y}),f(y^c)}(\nu(\underline{y}),\underline{x}) \;d\underline{x} \\
&  = (-1)^{\lvert\tilde{\alpha}^{y}\rvert}\int  d_{\underline{y}}\Xi(\underline{x}-\nu(y^c))   \partial^{\tilde{\alpha}^y} \varphi_{f(\underline{y}),f(y^c)}(\nu(\underline{y}),\underline{x}) \;d\underline{x} .
\end{align*} 
It remains to justify passing the derivative $\partial^{\tilde{\alpha}^y}$ through the integral sign. For any multi-index $\beta$, $\partial^\beta\varphi_f(\nu(\underline{y}),\underline{x})$ is equal to a polynomial in $(\nu(\underline{y}),\underline{x})$ times
\begin{displaymath}
    \varphi_{f(\underline{y}),f(y^c)}(\nu(\underline{y}),\underline{x})\leq c_1 e^{-c_2 \|(\nu(\underline{y}),\underline{x})\|_2^2 }
\end{displaymath}
where $c_1,c_2 > 0$ depend only on $f$. Since this product is integrable in $\underline{x}$ uniformly over compacts in $\|\nu(\underline{y})\|_\infty$, the passage is justified by the measure-theoretic Leibniz rule.
\end{proof}

\subsection{An integrated formula for the tail}
We end this section by deriving an `integrated' formula for the variance of the tail of the chaos expansion that depends only on derivatives/pivotal intensities \textit{of a fixed order}. This is proved by iterating an interpolation formula for the covariance between smooth functionals of Gaussian vectors \cite{cha08}.

Let $\tilde{f} : D \mapsto \R$ be an independent copy of $f$. For $t\in[0,1]$ we define the interpolated vector
\begin{equation}
\label{e:inter}
    f^t=tf+\sqrt{1-t^2}\tilde{f}
\end{equation}
which has the same distribution as $f$.

\begin{proposition}\label{p:tailboundsmooth}
    Let $m \ge 1$, $t\in[0,1]$, and let $\Phi:\R^D\to\R$ be $m$-times continuously differentiable  such that $\E[\|\nabla^k\Phi(f)\|_2^2]<\infty$ for every $k \le m$. Then
    \begin{multline*}
        \mathrm{Cov}\Big[\sum_{m' \ge m} Q_{m'}[\Phi(f)], \sum_{m' \ge m} Q_{m'}[\Phi(f^t)] \Big]=\\
        \sum_{x,y\in D^m}\prod_{i=1}^{m}K_{x_i,y_i}\int_0^t\int_0^{t_1}\dots\int_0^{t_{m-1}}\E\left[\partial_{x}\Phi(f)\partial_{y}\Phi(f^{t_{m}})\right]\;dt_{m}\dots dt_1 .
    \end{multline*}
\end{proposition}
\begin{proof}
    The classical interpolation formula for the covariance (see \cite[Lemma~3.4]{cha08}) states that, provided $\Phi$ is continuously differentiable and $\E[\Phi(f)^2+\|\nabla\Phi(f)\|_2^2]<\infty$, for every $t\in[0,1]$
    \begin{equation}\label{e:cov_identity_smooth}
        \mathrm{Cov}[\Phi(f),\Phi(f^t)]=\sum_{x,y\in D}K_{x,y}\int_0^t\E[\partial_x\Phi(f)\partial_y\Phi(f^s)]\;ds.
    \end{equation}
    (The cited result is stated only for the case $t=1$ and uses a different parameterisation of the interpolation $f^t$, but the proof immediately yields this statement.) We then claim that, for every $M \ge 1$,
    \begin{equation}\label{e:tailboundsmooth}
    \begin{aligned}
        \mathrm{Cov}[\Phi(f),&\Phi(f^t)]=\sum_{m=1}^{M-1}\frac{t^m}{m!}\sum_{x,y\in D^m}\prod_{i=1}^mK_{x_i,y_i}\E[\partial_x\Phi(f)]\E[\partial_y\Phi(f)]\\
        &+\sum_{x,y\in D^{M}}\prod_{i=1}^{M}K_{x_i,y_i}\int_0^t\int_0^{t_1}\dots\int_0^{t_{M-1}}\E\left[\partial_x\Phi(f)\partial_y\Phi(f^{t_M})\right]\;dt_{M}\dots dt_1.
    \end{aligned}
    \end{equation}
    This follows from applying \eqref{e:cov_identity_smooth} iteratively to
    \begin{displaymath}
    \E[\partial_x\Phi(f)\partial_y\Phi(f^t)]=\E[\partial_x\Phi(f)]\E[\partial_y\Phi(f^t)]+\mathrm{Cov}[\partial_x\Phi(f),\partial_y\Phi(f^t)]
    \end{displaymath}
    using the fact that $\E[\partial_x\Phi(f^t)]=\E[\partial_x\Phi(f)].$ From the chaos expansions for $\Phi(f)$ and $\Phi(f^t)$ (Proposition~\ref{p:correlated_wce}) and the diagram formula for Wick polynomials (Theorem~\ref{t:df}) we have
    \begin{align*}
        \mathrm{Cov} \big[Q_m[\Phi(f)],Q_m[\Phi(f^t)] \big]&=\frac{t^m}{(m!)^2} \sum_{x, y \in D^m} \Big( \sum_{\sigma \in S_m} \prod_{i=1}^m K_{x_i,y_{\sigma(i)}} \Big) \E[\partial_x\Phi(f)]\E[\partial_y\Phi(f)]\\
    &=\frac{t^m}{m!} \sum_{x, y \in D^m} \Big( \prod_{i=1}^m K_{x_i,y_i} \Big) \E[\partial_x\Phi(f)]\E[\partial_y\Phi(f)],
    \end{align*}
    where $S_m$ denotes the set of permutations of $\{1,\dots,m\}$, and the second equality uses commutativity of partial derivatives. Since the different chaoses are orthogonal, we conclude that the final term in \eqref{e:tailboundsmooth} must be the covariance of $\sum_{m\geq M}Q_m[\Phi(f)]$ and $\sum_{m\geq M}Q_m[\Phi(f^t)]$ yielding the statement of the proposition.
\end{proof}

To apply this formula to a level-set functional $\Xi:\mathcal{P}(D)\to\R$ we need a `joint' variant of the pivotal intensities $P$ with respect to a level $\ell \in \R$: 

\begin{definition}[Joint pivotal intensities]
\label{d:jpi}
For $t\in[0,1)$ and $x,y\in D^m$, the $m+m$ \textit{joint pivotal intensity at $(x,y)$ (with respect to level $\ell \in \R$)} is
\begin{displaymath}
P^t(x;y):=\E\left[d_{x,y}\Xi(f,f^t)\middle|f(\underline{x})=\ell, f^t(\underline{y})=\ell\right]\varphi_{f(\underline{x}),f^t(\underline{y})}(\ell,\ell)
\end{displaymath}
where
\begin{displaymath}
    d_{x,y}\Xi(f,f^t):=d_{\underline{x}}\Xi(f-\ell)d_{\underline{y}}\Xi(f^t-\ell) H^{\tilde{\alpha}^x,\tilde{\alpha}^y}_{f,f^t}(f,f^t) .
\end{displaymath}
\end{definition}

As in the case of a single pivotal intensity, these can be interpreted as derivatives of the moments of the level-set functional:

\begin{proposition}\label{p:derpivintjoint}
    Let $m\in\N$, $x,y\in D^m$ and $t\in[0,1)$ then
    \begin{displaymath}
        \partial^{\alpha^x}_f\partial^{\alpha^y}_{f^t}\E\big[\Xi(f-\ell)\Xi(f^t-\ell)\big]=P^t(x;y).
    \end{displaymath}
\end{proposition}
\begin{proof}
    This follows from Proposition~\ref{p:derpivint} applied to the level-set functional $\Xi(f-\ell)\Xi(f^t-\ell)$ of the non-degenerate Gaussian field $(f,f^t)$ on $D \times D$
\end{proof}

We observe some basic properties of these joint intensities:

\begin{lemma}
\label{l:jpi}
The map $t \mapsto P^t(x;y)$ is continuous on $[0,1)$. Moreover 
\[   \int_0^1\int_0^{t_0}\dots\int_0^{t_{m-2}} | P^{t_{m-1}}(x;y)| \;dt_{m-1}\dots dt_0  < \infty .\]
\end{lemma}
\begin{proof}
For the continuity, observe that for $t\in[0,1)$ the vector $(f,f^t)$ is a non-degenerate Gaussian with covariance that varies continuously in $t$. Hence by Gaussian regression the same is true of the conditional distributions in the above intensity. Then since $d_{x,y}\Xi(f,f^t)$ is bounded by a polynomial (with coefficients depending continuously on $t$) an application of the dominated convergence theorem proves the claimed continuity.

For the integrability, appealing to Proposition \ref{p:mhpmain} we have
\[ \sqrt{ \E\left[  \big( H^{\tilde{\alpha}^x,\tilde{\alpha}^y}_{f,f^t}(f,f^t)  \big)^2 \middle|f(\underline{x})=f^t(\underline{y})=\ell\right]  } \varphi_{f(\underline{x}),f^t(\underline{y})}(\ell,\ell)  \le c_{f,x,y} (1-t)^{ - |\tilde{\alpha}^x +\tilde{\alpha}^y|/2 -|\underline{x} \cap \underline{y}|/2}  . \] 
Since $d_{\underline{x}}\Xi$ is bounded, and 
\[ |\tilde{\alpha}^x + \tilde{\alpha}^y| + |\underline{x} \cap \underline{y}| = (m - |\underline{x}|) + (m - |\underline{y}|) +  |\underline{x} \cap \underline{y}|  = 2m - |\underline{x} \cup \underline{y} | \le  2m-1, \]
applying the Cauchy-Schwarz inequality to the definition of $P^t$ we see that
\[ | P^t(x;y)|   \le  c_{f,x,y,\Xi} (1-t)^{-m+1/2 }  .  \]
 It remains to observe that
\[  \tau_m  := \int_0^1 \int_0^{t_0}\dots\int_0^{t_{m-2}} \frac{1}{ (1-t_{m-1})^{m-1/2 }} \;dt_{m-1}\dots dt_0  \]
is finite. This holds since, swapping the order of integration,
\[  \tau_m  =   \int_0^1 \textrm{Vol}(E_s)  \frac{1}{ ( 1-s)^{m-1/2 }}   \, \,ds \le \int_0^1  \frac{(1-s)^{m-1}}{(1-s)^{m -1/2}} \, ds =  \int_0^1 \frac{1}{ ( 1-s)^{1/2 } } \, ds  = 2 < \infty \] 
where
\begin{equation*} E_s = \big\{ x_0,\ldots,x_{m-2} \in [0,1] : x_0 \in [s,1], x_1 \in [s,x_0], \ldots, x_{m-2} \in [ s, x_{m-3} ] \big\} . \qedhere
\end{equation*}
\end{proof}

We can now state the desired formula:

\begin{proposition}
\label{p:wcerror}
  For every $m \ge 1$,
    \begin{displaymath}
        \Var\Big[\sum_{m' \geq m}Q_{m'}[\Xi(f-\ell)]\Big]=\sum_{x,y\in D^m}\prod_{i=1}^{m} K_{x_i,y_i}\int_0^1\int_0^{t_0}\dots\int_0^{t_{m-2}}P^{t_{m-1}}(x;y)\;dt_{m-1}\dots dt_0.
    \end{displaymath}
\end{proposition}

\begin{proof}
    Let $Z=(Z_x)_{x\in D}$ be an i.i.d.\ standard Gaussian vector which is independent of $f$. For $\epsilon>0$ and deterministic $h:\R^D\to\R$, define $\Xi^\epsilon(h)=\E[\Xi(h+\epsilon Z)]$. Applying Proposition~\ref{p:derpivint}, we deduce that $\Xi^\epsilon$ is smooth and for any $x_1,\dots,x_m\in D$
    \begin{equation}\label{e:WCTail1}
    \begin{aligned}
        \partial_{x_1}\dots\partial_{x_m}\Xi^\epsilon(h-\ell)&=\E[d_x\Xi(h-\ell+\epsilon Z)]\varphi_{\epsilon Z(\underline{x})}(\ell-h(\underline{x}))\\
        &=\E[d_{\underline{x}}\Xi(h-\ell+\epsilon Z)]H^{\tilde{\alpha}^x}_{\epsilon Z(\underline{x})}(\ell-h(\underline{x}))\varphi_{\epsilon Z(\underline{x})}(\ell-h(\underline{x}))
    \end{aligned}
    \end{equation}
    where we have used the independence of $Z(\underline{x})$ and $Z(x^c)$ along with the fact that $d_{\underline{x}}\Xi$ does not depend on the values of its argument at $\underline{x}$. Since $d_{\underline{x}}\Xi$ and $\varphi_{\epsilon Z(\underline{x})}$ are bounded, we conclude that the above random variable is bounded by a polynomial in $h(\underline{x})$ and hence $\E\left[\|\nabla^k\Xi^\epsilon(f)\|_2^2\right]<\infty$ for all $k \in \mathbb{N}$. We may therefore apply Proposition~\ref{p:tailboundsmooth} for any $t\in[0,1]$ and $\epsilon>0$ to obtain
    \begin{equation}\label{e:WCTail2}
    \begin{aligned}
        &\mathrm{Cov}\Big[\sum_{m' \ge m} Q_{m'}[\Xi^\epsilon(f-\ell)],\sum_{m' \ge m} Q_{m'}[\Xi^\epsilon(f^t-\ell)] \Big]=\\
        & \qquad \sum_{x,y\in D^m}\prod_{i=1}^{m}K_{x_i,y_i}\int_0^t\int_0^{t_1}\dots\int_0^{t_{m-1}}\E\big[\partial^{\alpha^x}_f\Xi^\epsilon(f-\ell)\partial^{\alpha^y}_{f^{t_m}}\Xi^\epsilon(f^{t_m}-\ell)\big]\;dt_{m}\dots dt_1.
    \end{aligned}
    \end{equation}
    By Lemma~\ref{l:derivatives_OU} (or, more precisely, its proof) applied to $\Xi(f-\ell)\Xi(f^t-\ell)$ viewed as a function of $(f,f^t)$, and Proposition~\ref{p:derpivintjoint}, we see that for any $s<1$
    \begin{equation}\label{e:WCTail3}
        \E\big[\partial^{\alpha^x}_f\Xi^\epsilon(f-\ell)\partial^{\alpha^y}_{f^{s}}\Xi^\epsilon(f^{s}-\ell)\big]\to\partial^{\alpha^x}_f\partial^{\alpha^y}_{f^{s}}\E\big[\Xi(f-\ell)\Xi(f^{s}-\ell)\big]= P^{s}(x;y)
    \end{equation}
    as $\epsilon\searrow 0$. We next claim that the left hand side of this expression is uniformly bounded over $s\in[0,t]$ for any given $t<1$. Assuming the claim, we may apply the dominated convergence theorem to the right hand side of \eqref{e:WCTail2} and use the $L^2$ convergence of $\Xi^\epsilon$ as $\epsilon\to 0$ to conclude that
    \begin{multline*}
        \mathrm{Cov}\Big[\sum_{m' \ge m} Q_{m'}[\Xi(f-\ell)],\sum_{m' \ge m} Q_{m'}[\Xi(f^t-\ell)] \Big]=\\
        \sum_{x,y\in D^m}\prod_{i=1}^{m}K(x_i-y_i)\int_0^t\int_0^{t_1}\dots\int_0^{t_{m-1}}P^{t_m}(x;y)\;dt_{m}\dots dt_1.
    \end{multline*}
    Taking $t\to 1$ proves the statement of the lemma since $\Xi(f^t)\to\Xi(f)$ in $L^2$ and $P^{t_m}$ is integrable over $[0,1]$ by Lemma~\ref{l:jpi}.

    It remains to prove the claim. Let $\tilde{Z}$ be an independent copy of $Z$ and define $f_\epsilon=f+\epsilon Z$ and $f^s_\epsilon=f^s+\epsilon\tilde{Z}$. By Lemma~\ref{l:derivatives_OU}, Fubini's theorem and Proposition~\ref{p:derpivint}, the left hand side of \eqref{e:WCTail3} is equal to
    \begin{align*}
        \partial^{\alpha^x}_f\partial^{\alpha^y}_{f^{s}}\E\big[&\Xi^\epsilon(f-\ell)\Xi^\epsilon(f^{s}-\ell)\big]\\
        &=\partial^{\alpha^x}_f\partial^{\alpha^y}_{f^{s}}\E\big[\Xi(f-\ell+\epsilon Z)\Xi(f^{s}-\ell+\epsilon\tilde{Z})\big]\\
        &=\E\big[d_{\underline{x}}\Xi(f_\epsilon-\ell)d_{\underline{y}}\Xi(f^{s}_\epsilon-\ell)H^{\tilde{\alpha}^x,\tilde{\alpha}^y}_{f_\epsilon,f^s_\epsilon}(f_\epsilon,f^s_\epsilon)\big|f_\epsilon(\underline{x})=f^s_\epsilon(\underline{y})=\ell\big]\varphi_{f_\epsilon(\underline{x}),f^s_\epsilon(\underline{y})}(\ell,\ell).
    \end{align*}
    By Gaussian regression, the definition of the Hermite polynomials, and the fact that discrete derivatives of $\Xi$ of a given order are bounded, this expression is continuous in $(s,\epsilon)\in[0,1)\times[0,1]$. Hence it is bounded uniformly over $(s,\epsilon)\in[0,t]\times[0,1]$, proving the claim.
\end{proof}

\begin{remark}
    As for Theorem \ref{t:cluster_wce}, it is likely that a version of Proposition \ref{p:wcerror} remains true if $\ell$ is replaced by an arbitrary vector $\nu \in \R^D$, but we do not need such an extension.
\end{remark}


\medskip

\section{Localising the chaotic components}
\label{s:sl}

In this section we consider the chaos expansion of the cluster count of the GFF. Specialising our notation from the previous section, for $D \subset \Z^d$ a finite subset we let $\Xi_D(E)$ be the sum of the number of connected components of $E$ and $D \setminus E$ that do not intersect $\partial D$, so that $\Xi_{\Lambda_R}(f-\ell) = N_R(\ell)$. We then write 
\begin{equation}
    \label{e:cecc}
 N_R(\ell) =  \E[N_R(\ell)] + \sum_{m \ge 1} Q_m[N_R(\ell)] 
 \end{equation}
for the chaos expansion of the cluster count.

\smallskip
The main result of this section is that each component of \eqref{e:cecc} can be approximated by a `semi-local' counterpart whose coefficients $P_\infty$ are stationary and rapidly decaying away from the diagonal:

\begin{proposition}
\label{p:statwce}
Let $\ell \notin \{-\ell_c,\ell_c\}$ and $m \ge 1$. Then there exists a function $P_\infty: (\Z^d)^m \to \R$ such that, as $R \to \infty$
\begin{equation}\label{e:varbound}
\begin{aligned}
\Var \Big[  Q_m[N_R(\ell)] - \frac{1}{m!}\sum_{x_1,\dots,x_m\in \Lambda_R} P_\infty(x_1,\ldots,x_m) &\wick{f(x_1)\cdots f(x_m)} \Big] \\&=  \begin{cases} O(R^d) & \text{if } m=1, \\  o(R^d) & \text{if } m \ge 2. \end{cases}
 \end{aligned}
\end{equation}
Moreover $P_\infty$ is stationary, i.e.\
\begin{displaymath}
    P_\infty(x_1+y,\dots,x_m+y)=P_\infty(x_1,\dots,x_m) \, , \quad \forall y \in \Z^d,
\end{displaymath}
 permutation invariant, symmetric in the sense that $P_\infty(x) = P_\infty(-x)$, and for all $m \ge 2$ there exist $c_1,c_2,\rho > 0$ such that,
\[ | P_\infty(x_1,\dots,x_m)  | \le c_1 e^{-c_2 \mathrm{diam}_\infty(\underline{x})^\rho} \]
where $\mathrm{diam}_\infty$ denotes diameter with respect to the uniform distance $d_\infty$.
\end{proposition}

The error term in \eqref{e:varbound} is of higher order in the case $m=1$ due to boundary effects. In the following propositions, we refine the approximation in \eqref{e:varbound} by including boundary effects, and also give a qualitative version of \eqref{e:varbound} valid at critical levels:

\begin{proposition}
\label{p:varasymp}
Let $\ell \notin \{-\ell_c,\ell_c\}$. Then there exists a function $P_\infty^H: \N \to \R$ such that, as $R \to \infty$,
\[  \Var \Big[  Q_1[N_R(\ell)]  -  \sum_{x\in \Lambda_R}   P_\infty^H \big(d_\infty(x, \partial \Lambda_R) \big) f(x) \Big]  = o(R^d). \]
Moreover the function $P_\infty^H$ satisfies
\[ | P_\infty^H(k) - P_\infty(0) | \le c_1 e^{-c_2 k^\rho} ,\]
where $c_1,c_2,\rho > 0$ are constants and $P_\infty$ is as in Proposition \ref{p:statwce}.
\end{proposition}

\begin{proposition}
\label{p:statwcequal}

There exists $P_\infty(0) \in \R$ such that, as $R \to \infty$,
\[ \Var\Big[  Q_1[N_R(\ell_c)]  - P_\infty(0) \sum_{x\in \Lambda_R} f(x) \Big] = o(R^{d+2}) . \]
\end{proposition}

The implicit constants in \eqref{e:varbound} are not uniform in the chaos order $m$. To control the tail of the chaos expansion we use an alternate `semi-localisation' which arises by considering a truncation of the component count. Define $m_0 = \max\{2,7-d\}$, i.e.\ $m_0$ is the smallest positive integer such that $m_0 (d-2) > d$.

\begin{proposition}
\label{p:statwcetrunc}
Let $\ell \notin \{-\ell_c,\ell_c\}$ and $\eps > 0$. Then there exists $r, M > 0$, and for each $m \ge m_0$ a function $P_{\infty; \le r}: (\Z^d)^m \to \R$, such that as $R \to \infty$, eventually
\begin{equation*}
\Var \Big[ \sum_{m \ge m_0} Q_m[N_R(\ell)] - \sum_{m_0\leq m\leq M} \overline{Q}_m(R,r) \Big] \le  \eps R^d ,
\end{equation*}
where
\begin{displaymath}
    \overline{Q}_m(R,r):=\frac{1}{m!}\sum_{x_1,\dots,x_m\in\Lambda_R}P_{\infty;\leq r}(x_1,\dots,x_m)\wick{f(x_1)\cdots f(x_m)}.
\end{displaymath}
Moreover the functions $P_{\infty; \le r}$ are stationary, permutation invariant, symmetric, supported on the set $\{ x \in (\Z^d)^m : \mathrm{diam}_\infty(\underline{x}) \le r+2\}$, and satisfy $| P_{\infty; \le r}(x)| \le e^{c m}\sqrt{m!}$ for $c = c_{d,\ell,r} > 0$.
\end{proposition}

The remainder of the section is devoted to proving Propositions \ref{p:statwce}--\ref{p:statwcetrunc}.

\subsection{Stationary pivotal intensities}
\label{ss:stat}

In this section we define the functions $P_\infty$, $P_\infty^+$ and $P_{\infty; \le r}$ appearing in Propositions \ref{p:statwce}--\ref{p:statwcetrunc}. We emphasise that these definitions are valid at every level, including criticality.

\subsubsection{Stationary pivotal events and intensities}
Let $\mathrm{Piv}_R(\underline{y},\sigma)$ and $P_R$ be respectively the pivotal events and intensities for the functional $N_R(\ell) = \Xi_{\Lambda_R}(f-\ell)$ given in Definitions~\ref{d:piveve} and \ref{d:pivint}. We shall define $P_\infty$ to be the stationary analogue of $P_R$; for this we first need a stationary version of the pivotal events:

\begin{definition}[Stationary pivotal events]
For $y \in (\Z^d)^m$ and $E \in \mathcal{P}(\Z^d)$ define
\begin{displaymath}
    d_{\underline{y}} \Xi_\infty(E)=\limsup_{R\to\infty}d_{\underline{y}}\Xi_{\Lambda_R}(E).
\end{displaymath}
For $\sigma \in\R$, we say that a configuration $E\subseteq \Z^d$ is \emph{stationary-$\sigma$-pivotal at $\underline{y}$} if
\begin{displaymath}
    d_{\underline{y}} \Xi_\infty(E)=\sigma
\end{displaymath}
and define $\mathrm{Piv}_\infty(\underline{y},\sigma)$ to be the set of all such configurations.
\end{definition}

We next give a stationary extension of the iterated derivatives defined in \eqref{e:itex}. For $y \in (\Z^d)^m$, let $y^c = \Z^d \setminus \underline{y}$ and let $(f(\underline{y})|f(y^c)=0)$ be defined as the projection of $f(\underline{y})$ onto the orthogonal complement of the subspace of $L^2$ spanned by $f(y^c)$. Since $f$ has an i.i.d.\ component (Assumption~\ref{a:sd}), so will the previous projection and hence $(f(\underline{y})|f(y^c)=0)$ is a non-degenerate Gaussian vector. In particular the Hermite polynomial $H^{\tilde{\alpha}^y}_{f(\underline{y}) | f(y^c) =  0}$ is well-defined, where $\tilde{\alpha}^y$ is as in \eqref{e:itex}. Then we define
\[  d_y \Xi_\infty(f-\ell) = d_{\underline{y}} \Xi_\infty(f-\ell) H^{\tilde{\alpha}^y}_{f(\underline{y}) | f(y^c) = 0} \big(f(\underline{y})-\E[f(\underline{y})|f(y^c)]\big) . \]
Note that this corresponds to the equivalent formulation of the Hermite polynomial from Remark~\ref{r:pivint}; this is necessary since the Hermite polynomial $H_{f(\underline{y}), f(y^c)}$ is not well-defined for an infinite vector.

\smallskip
We can now define the stationary pivotal intensities appearing in Proposition \ref{p:statwce}--\ref{p:statwcequal}. 

\begin{definition}[Stationary pivotal intensities]
\label{d:statpivint}
For $y \in (\Z^d)^m$, the \textit{stationary pivotal intensity at $y$} is 
  \begin{displaymath}
P_\infty(y) = \E\left[ d_y\Xi_\infty(f-\ell) \middle|f(\underline{y})=\ell\right]\varphi_{f(\underline{y})}(\ell) .
\end{displaymath}
\end{definition} 

Next we confirm that $P_\infty$ is indeed stationary, permutation invariant, and symmetric. For this we provide a `stabilisation' characterisation of the stationary pivotal events: for large enough domains the discrete derivative is eventually constant. Intuitively this holds because we are only interested in bounded clusters touching the pivotal points, which must be determined on some bounded domain.

\begin{lemma}[Stabilisation]
\label{l:piv_stabilisation}
    Let $E\subseteq\Z^d$, $m \ge 1$ and $y\in(\Z^d)^m$. Let $D\subset\Z^d$ be a finite subset such that every bounded cluster of $E\setminus\underline{y}$ and $E^c\setminus\underline{y}$ which intersects a neighbour of some $y_i$ is contained in $D\setminus\partial D$. Then $d_{\underline{y}}\Xi_\infty(E)=d_{\underline{y}}\Xi_D(E)$.
\end{lemma}
In particular, this implies that the limit superior in the definition of $d_{\underline{y}}\Xi_\infty$ can be replaced by a genuine limit.
\begin{proof}
    Let $D\subset\Z^d$ satisfy the conditions of the lemma and let $D^\prime\supseteq D$. It is enough to show that $d_{\underline{y}}\Xi_{D^\prime}(E)=d_{\underline{y}}\Xi_{D}(E)$. By definition of the cluster count,
    \begin{equation}\label{e:piv_stabilisation}
        \Xi_{D^\prime}(E)=\Xi_{D}(E)+\tilde{\Xi}^+(E)+\tilde{\Xi}^+(E^c)
    \end{equation}
    where $\tilde{\Xi}^+(E)$ denotes the number of clusters of $E$ which are contained in $\mathrm{int}D^\prime:=D^\prime\setminus\partial D^\prime$ and intersect $D^\prime\setminus\mathrm{int}D$.

    By definition of the discrete derivative, $d_{\underline{y}}\tilde{\Xi}^+(E)$ can be expressed as a linear combination of terms of the form
    \begin{displaymath}
        \tilde{\Xi}^+(\tilde{E}\cup\{y_i\})-\tilde{\Xi}^+(\tilde{E}\setminus\{y_i\})
    \end{displaymath}
    for some $i=1,\dots,m$ and some $\tilde{E}$ that agrees with $E$ except possibly at the pivotal points (that is, $\tilde{E}\Delta E\subseteq\underline{y}$). We will show that all such terms are zero and hence that $d_{\underline{y}}\tilde{\Xi}^+(E)=0$.

    Observe that if we add $y_i$ to $\tilde{E}\setminus\{y_i\}$ this point can either join an existing cluster, make a connection between multiple existing clusters or form a new single-point cluster at $y_i$. In all cases we see that $\tilde{\Xi}^+$ cannot increase, so
    \begin{displaymath}
        \tilde{\Xi}^+(\tilde{E}\cup\{y_i\})\leq\tilde{\Xi}^+(\tilde{E}\setminus\{y_i\}).
    \end{displaymath}
    Now consider any cluster of $\tilde{E}\setminus\{y_i\}$ contained in $\mathrm{int}D^\prime$ which intersects $D^\prime\setminus\mathrm{int}D$. By the first condition in the statement of the lemma, this cluster cannot intersect any neighbour of $y_i$ and so it must also form a cluster in $\tilde{E}\cup\{y_i\}$. By taking the union over all such clusters, we have
    \begin{displaymath}
        \tilde{\Xi}^+(\tilde{E}\setminus\{y_i\})\leq\tilde{\Xi}^+(\tilde{E}\cup\{y_i\}).
    \end{displaymath}
    Combined with the previous equation and the earlier observation, we conclude that $d_{\underline{y}}\tilde{\Xi}^+(E)=0$. A near-identical argument, using the second condition in the statement of the lemma, shows that $d_{\underline{y}}\tilde{\Xi}^+(E^c)=0$. Thus applying $d_{\underline{y}}$ to \eqref{e:piv_stabilisation} completes the proof of the lemma.
\end{proof}

\begin{lemma}
    \label{l:sym}
    The pivotal intensities $P_\infty$ are stationary, permutation invariant, and symmetric. 
\end{lemma}

\begin{proof}
From the definition, we see that $P_\infty(y)$ depends only on $\underline{y}$, which is invariant under permutations of $y$. The symmetry of $P_\infty(y)$ follows from the symmetry of the field $f$.

To prove stationarity, let $E\subseteq\Z^d$, $y\in(\Z^d)^m$, and $x\in\Z^d$. By definition of the cluster count and Lemma~\ref{l:piv_stabilisation}
\begin{displaymath}
    d_{\underline{y}+x}\Xi_\infty(E+x)=\lim_{R\to\infty}d_{\underline{y}+x}\Xi_{\Lambda_R}(E+x)=\lim_{R\to\infty}d_{\underline{y}}\Xi_{\Lambda_R-x}(E)=d_{\underline{y}}\Xi_\infty(E).
\end{displaymath}
Then by inspecting the definition of $P_\infty$, and using the fact that $f$ is stationary, it is clear that $P_\infty(y)=P_\infty(y+x)$.
\end{proof}

 \subsubsection{Half-space pivotal intensities}
\label{ss:hspi}

In the case $m=1$ we will also need `half-space' analogues of the stationary pivotal events and intensity. Let $\Z^d_+ := \{ x = (x_1,\ldots,x_d) \in \Z^d : x_1 \ge 0\}$ denote the upper half-space.

\begin{definition}[Half-space pivotal events]
\label{d:hspiveve}
For $y \in \Z^d_+$ and $E \in \mathcal{P}(\Z^d)$, define
\begin{displaymath}
    d_y \Xi^{\mathrm{H}}_\infty(E)=\limsup_{R\to\infty}d_y\Xi_{\Lambda_R \cap \Z^d_+}(E).
\end{displaymath}
For $\sigma \in \R$ we say that a configuration $E\subseteq \Z^d$ is \emph{half-space-$\sigma$-pivotal at $y$} if
\begin{displaymath}
    d_y \Xi^{\mathrm{H}}_\infty(E)=\sigma.
\end{displaymath}
We define $\mathrm{Piv}^{\mathrm{H}}_\infty(y,\sigma)$ to be the set of all such configurations.
\end{definition}

\begin{definition}[Half-space pivotal intensity]
\label{d:hspivint}
For $k \in \N_0$, the \textit{half-space pivotal intensity at height $k$} is
  \begin{displaymath}
P_\infty^{\mathrm{H}}(k) = \E\left[ d_y\Xi^{\mathrm{H}}_\infty(f-\ell) \middle|f(y)=\ell\right]\varphi_{f(y)}(\ell) ,
\end{displaymath}
where $y$ denotes the point $(k,0,\ldots,0)$.
\end{definition} 

\subsubsection{Truncated pivotal intensities}
The function $P_{\infty;\le r}$ in Proposition \ref{p:statwcetrunc} is defined as the analogue of $P_\infty$ for a \textit{truncated cluster count}. More precisely, we let $N_{R;\le r}(\ell)$ denote the sum of $\Xi_{R;\leq r}^+(f-\ell)$ and $\Xi_{R;\leq r}^-(f-\ell)$, defined respectively as the number of clusters of $\{f > \ell\}$ and of $\{f<\ell\}$ in $\Lambda_R$ which are of diameter at most~$r$ (and do not intersect $\partial \Lambda_R)$. We define $N_{R;>r}(\ell)$ analogously for clusters of diameter strictly larger than $r$. We then define $P_{R; \le r}$ (resp.\ $P_{R; > r}$) and $P_{\infty; \le r}$ (resp.\ $P_{\infty; > r}$) in the same way as $P_R$ and $P_\infty$ with the functionals $N_{R; \le r}(\ell)$ (resp.\ $N_{R; > r}(\ell)$) in place of $N_R(\ell)$.

\begin{lemma}
    \label{l:stationarytrunc}
    For every $r,R \geq 1$, $P_{\infty;\leq r}$ is stationary, permutation invariant, symmetric, and both $P_{R;\leq r}$ and $P_{\infty;\leq r}$ are supported on the set $\{ x \in (\Z^d)^m : \mathrm{diam}_\infty(\underline{x}) \le r+2\}$.
\end{lemma}
\begin{proof}
    The proofs of stationarity, permutation invariance, and symmetry are identical to that of Lemma~\ref{l:sym} so we omit the details. For the final claim, by the definition of stationary pivotal intensities (Definitions~\ref{d:pivint} and~\ref{d:statpivint}) it is enough to show that for any $E\subseteq\Z^d$
\begin{displaymath}
    d_{\underline{y}}\Xi^\pm_{R;\leq r}(E)=0
\end{displaymath}
whenever $\mathrm{diam}_\infty(y)>r+2$. Given a configuration $E\subseteq\Z^d$ and a point $y_1\in\Z^d$, we observe that every cluster of $E\setminus\{y_1\}$ must satisfy (at least) one of the following conditions: have diameter greater than $r$, not contain a neighbour of $y_1$, or be contained in $y_1 + \Lambda_{r+1}$. If we add the point $y_1$ to our configuration, then clearly only clusters satisfying the third condition can contribute to a change in $\Xi^+_{R;\leq r}$. Therefore by looking only at the configuration $E$ on $y_1+\Lambda_{r+2}$, we can identify which of these clusters have diameter at most $r$ and hence compute $d_{y_1}\Xi^+_{R;\leq r}(E)$. In other words if $E_1$ and $E_2$ are two configurations which agree on $y_1 + \Lambda_{r+2}$, then $d_{y_1}\Xi^+_{R;\leq r}(E_1)=d_{y_1}\Xi^+_{R;\leq r}(E_2)$. Hence for any point $y_2\notin y_1 + \Lambda_{r+2}$ and any $E\subseteq\Z^d$ we have $d_{y_2}d_{y_1}\Xi^+_{R;\leq r}(E)=0$. If $y=(y_1,\dots,y_m)\in(\Z^d)^m$ has diameter greater than $r+2$, then we can find $y_i$ and $y_j$ at distance greater than $r+2$. Then by definition of the discrete derivative, we can express $d_{\underline{y}}\Xi^+_{R;\leq r}(E)$ as a linear combination of terms of the form $d_{y_i}d_{y_j}\Xi^+_{R;\leq r}(E^\prime)$ for different configurations $E^\prime$. Since each of these are zero, we have proven the desired conditions for $\Xi^+$. The argument for $\Xi^-$ is entirely analogous.
\end{proof}

An advantage of the truncated intensities is that they admit a uniform bound:

\begin{lemma}
\label{l:bounds}
Let $\ell \in \R$ and $r \ge 1$. There exists $c > 0$ depending only on $d$, $\ell$ and $r$ such that, for all $m, R \ge 1$, and $y \in (\Lambda_R)^m$,
\[ \max\big\{ |P_{R;\le r}(y)|, |P_{\infty;\le r}(y) | \big\} \le  e^{c m} \sqrt{m!} . \]
\end{lemma}

\begin{proof}
For any $R>0$, $y\in\Lambda_R$, and $E\subseteq\Lambda_R$, since $y$ has at most $2d$ neighbours, we have $\lvert d_y\Xi_{R;\leq r}(E)\rvert\leq 2d$. The bound then follows from Lemma~\ref{l:bounds1} and the observations that (i) by Lemma \ref{l:stationarytrunc}, $P_{R;\le r}(y)$ and $P_{\infty;\le r}(y)$ are supported on configurations $y$ for which $\underline{y}$ has bounded cardinality, and (ii) the covariance matrix of any subset of the GFF has smallest eigenvalue bounded below by a constant, which follows from considering the i.i.d.\ component (Assumption~\ref{a:sd}).
\end{proof}

\subsection{Truncated arm estimates}
\label{s:tae}
Our control of the pivotal intensities hinges on \textit{truncated arm decay} estimates. For a subset $E \subset \Z^d$, recall the \textit{truncated arm event} $\mathrm{Arm}_R(E)$ that $E$ contains a component which is (i) bounded, and (ii) includes a path connecting a neighbour of $0$ and $\partial \Lambda_R$. It is clear from the definition that for every $\ell \in \R$, as $R \to \infty$,
\begin{equation}
    \label{e:tadqual}
\P[ \mathrm{Arm}_R(\{f > \ell\} ) ] \to 0 .
\end{equation} 
In the off-critical regimes a much stronger statement is true:

\begin{theorem}[Truncated arm decay; \cite{dgrs23}]
\label{t:tad}
For every $\ell \neq \ell_c$, there exists $c,\rho > 0$ such that, for every $R \ge 1$,
\begin{equation}
    \label{e:tadquant}
 \P[ \mathrm{Arm}_R(\{f > \ell\} ) ] \le e^{-c R^\rho} . 
 \end{equation}
 The constants $c,\rho > 0$ can be chosen uniformly over compact subsets of $\R \setminus \{\ell_c\}$.
\end{theorem}

A `two-arm' version of \eqref{e:tadqual} also holds for \textit{non-truncated arm events}: for every $\ell \in \R$, as $R \to \infty$,
\begin{equation}
    \label{e:twoarm}
\P[ \mathrm{TwoArm}_R(\{f > \ell\} ) ] \to 0 
\end{equation} 
where $\mathrm{TwoArm}_R(E)$ is the event that $(E \cap \Lambda_R) \setminus \{0\}$ contains two distinct components which each include a path connecting a neighbour of $0$ and $\partial \Lambda_R$. As is well-known, \eqref{e:twoarm} follows from the a.s.\ uniqueness of the unbounded component of $\{f > \ell\}$ and the `finite-energy' property of fields with an i.i.d.\ component (see \cite[Definition 12.1 and Theorem 12.2]{hj06}).

\smallskip

We shall need `pinned' variants of \eqref{e:tadqual}--\eqref{e:twoarm} for the GFF conditioned at pivotal points. For $E \subset \Z^d$ and $y=(y_1,\dots,y_m)\in(\Z^d)^m$, let $\widetilde{\mathrm{Arm}}_{r;\underline{y}}(E)$ denote the event that $E \setminus \underline{y}$ contains a component which is (i) bounded, and (ii) includes, for at least one $1 \le i \le m$, a path connecting a neighbour of $y_i$ to $y_i + \partial \Lambda_r$. For later use we note that
\begin{equation}
    \label{e:arminclusion}
\big\{ \widetilde{\mathrm{Arm}}_{r;\underline{y}}(\{f > \ell\} )   \big\} \cap \{ f(\underline{y}) \le \ell \}   \quad \Longrightarrow \quad \bigcup_{i=1}^m\bigcup_{z \sim y_i} \{z + \mathrm{Arm}_{r-1}(\{f > \ell\} )  \}
\end{equation} 
where $z + A$ denotes the occurrence of $A$ for the shifted field $f(\cdot-z)$. For $t \in [0,1]$, let $f^t$ be the interpolated field defined in \eqref{e:inter}, that is
\[ f^t = tf + \sqrt{1-t^2} \tilde{f} \]
where $\tilde{f}$ is an independent copy of the GFF.

\begin{proposition}[Pinned truncated arm decay]
\label{p:ptad}
$\,$
\begin{enumerate}
\item For every $\ell \in \R$ and $m \ge 1$, as $r \to \infty$
\[ \sup_{y_1,\ldots,y_m\in\Z^d} \P \big[ \widetilde{\mathrm{Arm}}_{r;\underline{y}}(\{f > \ell\} ) \big| f(\underline{y}) = \ell \big] \to 0 . \]
\item For $\ell \neq \ell_c$, there exists $c,\rho > 0$ such that, for $y_1,\dots,y_m\in\Z^d$ and $r \ge 1$,
\[ \P \big[ \widetilde{\mathrm{Arm}}_{r;\underline{y}}(\{f > \ell\} ) \big| f(\underline{y}) = \ell \big] \le e^{-c r^\rho} . \]

 \item For $\ell \neq \ell_c$, there exists $c,\rho > 0$ such that, for $x_1,\ldots,x_m,y_1,\dots,y_m\in\Z^d$, $r \ge 1$, and $t \in [0,1)$,
\[ \P \big[ \widetilde{\mathrm{Arm}}_{r;\underline{x}}(\{f > \ell\} ) \big| f(\underline{x}) = f^t(\underline{y}) = \ell \big] \le (1-t)^{-m} e^{-c r^\rho}  . \] 
 \end{enumerate}
The constants $c,\rho > 0$ can be chosen uniformly over compact subsets of $\R \setminus \{\ell_c\}$.
\end{proposition}

\begin{proposition}[Pinned two-arm decay]
For every $\ell \in \R$, as $r \to \infty$
\label{p:ptwoarm}
\[  \P \big[ \mathrm{TwoArm}_{r}(\{f > \ell\} ) \big| f(0) = \ell \big] \to 0 . \]
\end{proposition}

Propositions \ref{p:ptad} and \ref{p:ptwoarm} will follow from \eqref{e:tadqual}--\eqref{e:twoarm} and the following estimate:

\begin{proposition}[De-pinning]
\label{p:depinning}
Let $(X,Y)$ be an $(n+m)$-dimensional centred Gaussian vector of unit variance such that $Y$ is non-degenerate, and let $\Sigma$ be its covariance matrix. Then for every $\ell$ and every event $A$ that depends only on $X$,
\[ \P[ A \, | \, Y = \ell ] \varphi_{Y}(\ell) \le c_{\ell,m} \lambda_{\mathrm{min}}^{-m}    P  \max\Big\{ 1, (\log 1/P)^{m/2} \Big\} \]
where $P := \P[A , Y \le \ell ]$ and $\lambda_{\mathrm{min}}$ is the smallest eigenvalue of $\Sigma$. The constant $c_{\ell,m} >0$ can be chosen uniformly for $\ell$ in compact subsets of $\R$.
\end{proposition}

\begin{proof}
Let $\lambda_{\mathrm{min}}(Y)$ be the smallest eigenvalue of the covariance matrix $\Sigma_Y$ of $Y$. We note that $\lambda_{\mathrm{min}} \le \lambda_{\mathrm{min}}(Y) \le \text{Var}(Y_i) = 1$ (see Lemma \ref{l:lambdamin}). For $s \ge 0$, consider the function
\[ h(s)  :=  \P \big[A  \big| Y = \ell-s \big]  , \]
where the conditional density is taken with respect to Lebesgue measure and is well defined since $Y$ is non-degenerate. (The same comment applies to all conditioning below.) We claim that for every $s, T \ge 0$,
\begin{equation}
    \label{e:ptad2}
h(s) \ge   e^{-2m^2 \lambda_{\mathrm{min}}^{-1} s^2} \Big(  e^{- 2 m  \lambda_{\mathrm{min}}^{-1/2}   s T} h(0)   - e^{-T^2/2} \Big)  . 
\end{equation} 
 Assuming \eqref{e:ptad2}, let us complete the proof of the proposition. Fix a $\delta \in (0,1]$. Then
\begin{equation}
    \label{e:ptad3}
   \inf_{s \in [0,\delta]} h(s) \le  \P \big[A \,  \big| \, Y \in [ \ell-\delta,  \ell ] \big]  = \frac{   \P \big[A , Y \in [ \ell-\delta,  \ell ] \big] }{ \P\big[ Y \in [ \ell-\delta,  \ell  ] \big] } \le \frac{\delta^{-m}  P }{ \min_{y \in [\ell-\delta,\ell]} \varphi_Y(y)},
\end{equation}
where $Y\in[\ell-\delta,\ell]$ denotes that each element of the vector $Y$ lies in this interval. We also have
\begin{equation}
\label{e:ptad4}
\min_{y \in [\ell-\delta,\ell]} \varphi_Y(y) / \varphi_Y(\ell) \ge \min_{p \in [0,\delta]} e^{- |p^T \Sigma_Y^{-1} \ell | } e^{- \frac{1}{2}|p^T \Sigma_Y^{-1} p| } \ge e^{- m \lvert\ell\rvert \delta \lambda_{\min}^{-1} - m \delta^ 2 \lambda_{\min}^{-1}/2 }  
\end{equation}
and
\begin{equation}
    \label{e:ptad5}
 \varphi_Y(\ell) \le  \lambda_{\min}^{-m/2} .
 \end{equation}
Combining \eqref{e:ptad2}--\eqref{e:ptad5} we have
\[
h(0) \varphi_Y(\ell) \le e^{2m \lambda_{\mathrm{min}}^{-1/2} \delta T} \Big( e^{  m^2 \delta \lambda_{\mathrm{min}}^{-1} (\lvert\ell\rvert + 5/2) }  \delta^{-m} P  +  \lambda_{\min}^{-m/2}  e^{-T^2/2} \Big) . 
\]
We may assume that $P \le 1/e$ otherwise the result is immediate. Then setting $T = 2 \sqrt{-\log P}$ and $\delta =  \lambda_{\mathrm{min}}/ \sqrt{-\log P}$ we have
\begin{align*}
h(0) \varphi_Y(\ell) & \le e^{4m} \Big( e^{ m^2(\lvert\ell\rvert+5/2)} \lambda_{\mathrm{min}}^{-m} P (\log 1/P)^{m/2} +  \lambda_{\min}^{-m/2}  P^2  \Big) \\
& \le 2 e^{m^2(\lvert\ell\rvert+13/2)}  \lambda_{\mathrm{min}}^{-m} P (\log 1/P)^{m/2} , 
\end{align*} 
which gives the result.

Towards proving \eqref{e:ptad2} we consider the vector
\[ k(i) = \Sigma_{X_i \to Y} \Sigma_Y^{-1} \id_m   \ , \quad i = 1,\ldots,n  ,\]
where $\id_m = (1,\ldots,1)$ is the vector of ones. We claim that 
\begin{equation}
        \label{e:cm}
\kappa:= \|(k,0)\|_{H(X,Y)}  \le  2m \lambda_{\min}^{-1/2} < \infty 
\end{equation} 
where $H(X,Y)$ denotes the RKHS of the vector $(X,Y)$. To prove this, consider that 
\[  \| ( \Sigma_{X \to Y} \Sigma_Y^{-1} \id_m, \id_m) \|_{H(X,Y)} =  \| \id_m \|_{H(Y)} = \| \Sigma_{Y}^{-1/2} \id_m \|_{L^2}  \le \sqrt{m} \lambda_{\mathrm{min}}^{-1/2}  , \]
and moreover, for every $j = 1,\ldots,m$,
\[ \| (0,\id_{\cdot = j})\|_{H(X,Y)} =  \| \Sigma^{-1/2} (0,\id_{\cdot = j}) \|_{L^2} \le \lambda_{\mathrm{min}}^{-1/2}  .  \]
Hence by the triangle inequality
\begin{equation*}
  \kappa  \le  \| (k,\ind_m) \|_{H(X,Y)} + \sum_{j=1,\ldots,m} \| (0, \id_{\{\cdot = j\}} \|_{H(X,Y)}  \le (\sqrt{m}+m) \lambda_{\mathrm{min}}^{-1/2}  
  \end{equation*}
as required. This yields the decomposition
\begin{equation}
    \label{e:decomp}
(X,Y) \stackrel{d}{=} (Z / \kappa) (k,0)   + (\bar{X},\bar{Y}) , \end{equation} 
where $Z$ is a standard Gaussian random variable, and $(\bar{X},\bar{Y})$ is an independent Gaussian vector. Then for every $s \ge 0$ we have, by Gaussian regression,
\begin{align*}
h(s) & := \P \big[X \in A  \, \big| \, Y = \ell-s \big] \\
& = \P \big[X - s k \in A \, \big|  \, Y = \ell \big]  \\
& = \P \big[(X,Y) - s (k,0)  \in A \,   \big| \, Y = \ell \big]  
\end{align*}
where the final equality is since $A$ does not depend on $Y$. Then using \eqref{e:decomp}, the independence of $Z$ and $(\bar{X},\bar{Y})$, and a change of measure 
\begin{align*} 
h(s) &= \P \Big[ (Z-s \kappa)  \kappa^{-1} (k,0)  + (\bar{X},\bar{Y}) \in A \, \big| \, \bar{Y} = \ell \Big] \\
 & = \E \Big[ \frac{ \varphi(Z+s\kappa)}{\varphi(Z)} \id_{Z \kappa^{-1} (k,0)  + (\bar{X},\bar{Y}) \in A} \, \Big| \, \bar{Y}= \ell \Big] \\
& =  e^{ -s^2 \kappa^2 / 2} \E \Big[ e^{ - Z s \kappa }  \id_{Z \kappa^{-1} (k,0)  + (\bar{X},\bar{Y}) \in A} \, \Big| \, \bar{Y}= \ell \Big] .
\end{align*}
 Then for every $T > 0$ we have
\begin{align*}
  e^{s^2 \kappa^2/2}  h(s) &\ge   e^{- s T \kappa } \P \Big[ Z \kappa^{-1} (k,0)  + (\bar{X},\bar{Y}) \in A \, \Big| \, \bar{Y} = \ell \Big]  -  \P[ Z > T ] \\
 & \ge   e^{-  s T \kappa } h(0)   - e^{-T^2/2} ,
 \end{align*}
 which, combined with \eqref{e:cm}, completes the proof of \eqref{e:ptad2}.
\end{proof}

\begin{proof}[Proof of Proposition \ref{p:ptad}]
Let $\ell \in\R$, $r \ge 2$, $m \ge 1$, and $y = (y_1,\dots,y_m)$ be given. Since the GFF has an i.i.d.\ component (Assumption~\ref{a:sd}), it satisfies $\varphi_{f(\underline{y})}(\ell) \ge c_{d,\ell,m}$ for a constant $c_{d,\ell,m} > 0$ that can be chosen uniformly for $\ell$ in a compact subset of $\R$, and also $\lambda_{\mathrm{min}}(f(x)) \ge c_d > 0$ uniformly over any subset $x \subset \Z^d$. Then applying Proposition \ref{p:depinning} (rescaling the GFF by a constant before applying the proposition)
\[\P \big[f \in \widetilde{\mathrm{Arm}}_{r;\underline{y}}(\{f > \ell\} )  \big| f(\underline{y}) = \ell \big]   \le  c_{d,\ell,m}  P  \max \Big\{ 1, \log  (1/P)^{m/2} \Big\}    , \]
    where $P = \P \big[f \in \widetilde{\mathrm{Arm}}_{r;\underline{y}}(\{f > \ell\} ) , f(\underline{y}) \le \ell \big]  $. Combining with \eqref{e:arminclusion} and the union bound, we further have
    \[ P \le 2d m \P[ \arm_{r-1}(\{f>\ell\}) \]
    and so the first two items of the proposition follow immediately from \eqref{e:tadqual} and \eqref{e:tadquant}.

For the third item we argue similarly, except we use the estimates
\[ \lambda_{\mathrm{min}} (f(\underline{x}),f^t(\underline{y}) ) \ge c_d (1-t)  \qquad \text{and} \qquad \varphi_{(f(\underline{x}),f^t(\underline{y}))}(\ell) \ge c_{d,\ell,m}  \]
which follow respectively from the first and third items of Proposition \ref{p:densitybound}.
\end{proof}

\begin{proof}[Proof of Proposition \ref{p:ptwoarm}]
Similarly to the proof of Proposition \ref{p:ptad}, applying Proposition~\ref{p:depinning} we have 
\[\P \big[f \in \mathrm{TwoArm}_{r}(\{f > \ell\} )  \big| f(0) = \ell \big]   \le  c_{d,\ell} P \max \Big\{1, \log  (1/P)^{1/2}  \Big\}   , \]
where  $P = \P [f \in \mathrm{TwoArm}_{r}(\{f > \ell\} ) ] $. Combining with \eqref{e:twoarm}, this gives the result.
\end{proof}

\subsection{Convergence and decay of the pivotal intensities}

We now use the estimates in the previous section to deduce convergence and decay properties of the pivotal intensities. For $0 \le i \le d-1$, let $F_R^i$ denote the union of the $i$-dimensional boundary faces of $\Lambda_R$, so that in particular $\partial \Lambda_R = F_R^{d-1}$.  

\begin{lemma}[Convergence]
\label{l:conv}
 For every $\ell \in \R$ and $m \ge 1$ the following hold:
\begin{itemize}
    \item There exists a function $\gamma$ satisfying $\gamma(r) \to 0$ as $r \to \infty$ such that, for every $R \ge 1$ and $y  \in (\Lambda_R)^m$,
    \[ | P_R(y) - P_\infty(y) | \le \gamma \big(d_\infty(\underline{y},\partial \Lambda_R) \big)  ,\]
   and if $m = 1$
       \[ | P_R(y)  - P^{\mathrm{H}}_\infty(d_\infty(y,\partial \Lambda_R) ) | \le \gamma \big(d_\infty(y,F_R^{d-2}) \big)  .\]
Moreover if $\ell \notin \{-\ell_c,\ell_c\}$, $\gamma$ can be taken as $\gamma(r) = c_1e^{-c_2 r^\rho}$ for some $c_1,c_2,\rho > 0$ which can be chosen uniformly over compact subsets of $\R\setminus\{-\ell_c,\ell_c\}$.
\item For every $R,r \ge 1$ and $y \in (\Lambda_R)^m$,
\[ | P_{R; \le r}(y) - P_{\infty; \le r}(y) | \le c_{m,d}\id_{d_\infty(\underline{y},\partial \Lambda_R) \le r+2} . \]
\end{itemize} 
\end{lemma}

For $t \in [0,1)$, let $P^t_{R;> r}(x;y)$ be the joint pivotal intensity given in Definition \ref{d:jpi} with respect to $N_{R;>r}(\ell)$.

\begin{lemma}[Decay]
\label{l:decay}
 For every $\ell \in \R$, $m \ge 1$, and $\delta > 0$, there exists a function $\gamma$ satisfying $\gamma(r) \to 0$ as $r \to \infty$ such that, for every $R,r \ge 1$, $x,y \in (\Lambda_R)^m$, and $t \in [0,1)$
    \[ \max \{ |P_R(y)| , |P_\infty(y)|  \} \le \gamma \big( \mathrm{diam}_\infty(\underline{y}) \big)  \]
    and
    \[ |P^t_{R; >r}(x;y) | \le  (1-t)^{-m+1/2-\delta} \gamma \big( \max\big\{r, \max_{z = x,y} \mathrm{diam}_\infty(\underline{z}) \big\} \big)  . \]
    Moreover if $\ell \notin \{-\ell_c,\ell_c\}$, $\gamma$ can be taken as $\gamma(r) = c_1 e^{-c_2 r^\rho}$ for some $c_1,c_2,\rho > 0$ which can be chosen uniformly over compact subsets of $\R\setminus\{-\ell_c,\ell_c\}$.
\end{lemma}

\begin{remark}
    Although the constants $c_1,c_2,\rho > 0$ in the definition of $\gamma$ in Lemmas \ref{l:conv} and \ref{l:decay} may depend on $\ell$, $m$, and $\delta$, this will not matter for our later arguments.
\end{remark}

\begin{proof}[Proof of Lemma \ref{l:conv}]
We first show the pointwise convergence $P_R(y) \to P_\infty(y)$ as $R \to \infty$. Let $y\in(\Z^d)^m$ be given, and assume $R$ is sufficiently large so that $y \subseteq \Lambda_R$. We recall the alternate expression for the pivotal intensity given in Remark~\ref{r:pivint}:
\begin{displaymath}
    P_R(y)=\E\big[d_{\underline{y}}\Xi_{\Lambda_R}(f-\ell)H^{\tilde{\alpha}^y}_{f(\underline{y})|f(\Lambda_R\setminus \underline{y})=0}(X_R)\big|f(\underline{y})=\ell\big]\varphi_{f(\underline{y})}(\ell)
\end{displaymath}
where $X_R:=f(\underline{y})-\E[f(\underline{y})|f(\Lambda_R\setminus y)]$. Let $\overline{f}_\ell:=(f|f(\underline{y})=\ell)$ and $X_\infty:=f(\underline{y})-\E[f(\underline{y})|f(y^c)]$. We now make the following claims: 
\begin{align}
    \label{e:conv_4}
    \sup_{R>1}\sup_{y\in(\Z^d)^m}\E&\left[\left(H^{\tilde{\alpha}^y}_{f(\underline{y})|f(\Lambda_R\setminus\underline{y})=0}(X_R)\right)^2\middle|f(\underline{y})=\ell\right]<\infty,\\
    \label{e:conv_1}
    d_{\underline{y}}\Xi_{\Lambda_R}(\overline{f}_\ell-\ell)&\to d_{\underline{y}}\Xi_{\infty}(\overline{f}_\ell-\ell)\quad\text{almost surely as }R\to\infty,\\
    \label{e:conv_2}
    H^{\tilde{\alpha}^y}_{f(\underline{y})|f(\Lambda_R\setminus \underline{y})=0}&\to H^{\tilde{\alpha}^y}_{f(\underline{y})|f(y^c)=0}\quad\text{pointwise as }R\to\infty\text{, and}\\
    \label{e:conv_3}
    (X_R|f(\underline{y})=\ell)&\to(X_\infty|f(\underline{y})=\ell)\quad\text{almost surely as }R\to\infty.
\end{align}
By \eqref{e:conv_4} and the fact that $\lvert d_{\underline{y}}\Xi_{\Lambda_R}\rvert\leq 2^m d$ (courtesy of Lemma~\ref{l:bounds1}), we see that the integrands in the definition of $P_R$ are uniformly integrable. Then by \eqref{e:conv_1}--\eqref{e:conv_3} and Vitali's convergence theorem, as $R\to\infty$
\begin{displaymath}
    P_R(y)\to\E[d_{\underline{y}}\Xi_\infty(f-\ell)H^{\tilde{\alpha}^y}_{f(\underline{y})|f(y^c)=0}(f(\underline{y})-\E[f(\underline{y})|f(y^c)])|f(\underline{y})=\ell]\varphi_{f(\underline{y})}(\ell)=P_\infty(y).
\end{displaymath}

We now verify the four claims. Claim \eqref{e:conv_4} was established in  the proof of Lemma~\ref{l:bounds1}, (specifically in \eqref{e:bounds11}). The latter equation actually bounded the first absolute moment of the Hermite polynomial in \eqref{e:conv_4}, but since the arguments are Gaussian this is equivalent to bounding the second moment up to some constant depending only on $m$.

The stabilisation property of the cluster count (Lemma~\ref{l:piv_stabilisation}) immediately yields \eqref{e:conv_1}. Turning to \eqref{e:conv_2}, we note that $(f(\underline{y})|f(\Lambda_R\setminus\underline{y})=0)$ is simply the $L^2$ projection of $f(\underline{y})$ onto the orthogonal complement of the closed linear span of $f(\Lambda_R\setminus\underline{y})$. Since this subspace is decreasing in $R$, a standard Hilbert space argument shows that
\begin{displaymath}
    (f(\underline{y})|f(\Lambda_R\setminus\underline{y})=0)\to(f(\underline{y})|f(y^c)=0)
\end{displaymath}
in $L^2$ as $R\to\infty$. In particular the covariance matrices of the former converge to that of the latter, which is non-degenerate by Assumption~\ref{a:sd}. By definition \eqref{e:mhp}, multivariate Hermite polynomials are pointwise continuous with respect to the covariance matrix of the underlying (non-degenerate) Gaussian vector. This yields the pointwise convergence stated in \eqref{e:conv_2}.

For \eqref{e:conv_3}, by Gaussian regression and martingale convergence
\begin{align*}
    (X_R|f(\underline{y})=\ell)&=X_R+\mathrm{Cov}[X_R,f(\underline{y})]\mathrm{Cov}[f(\underline{y})]^{-1}(\ell-f(\underline{y}))\\
    &\to X_\infty+\mathrm{Cov}[X_\infty,f(\underline{y})]\mathrm{Cov}[f(\underline{y})]^{-1}(\ell-f(\underline{y}))=(X_\infty|f(\underline{y})=\ell).
\end{align*}
This completes the proof of the claims.

We now establish the quantitative bound on convergence. By Proposition~\ref{p:derpivint}
\begin{displaymath}
    \lvert P_\infty(y)-P_R(y)\rvert\leq\sup_{S>R}\lvert P_S(y)-P_R(y)\rvert=\sup_{S>R}\lvert\partial_{y_1}\dots\partial_{y_m}\E[\Xi_{\Lambda_S}-\Xi_{\Lambda_R}]\rvert
\end{displaymath}
where we have suppressed the argument $f-\ell$ from these level set functionals to ease notation (and will continue to do so below). By definition of the cluster count, we can express
\begin{displaymath}
    \Xi_{\Lambda_S}=\Xi_{\Lambda_R}+\Xi_{\Lambda_S\setminus\Lambda_{R-1}}+\Xi_{\Lambda_S,\partial\Lambda_R}
\end{displaymath}
where the last term above is defined to be the number of clusters contained in $\Lambda_S$ which intersect $\partial\Lambda_R$ (and do not intersect $\partial\Lambda_S$). Since $\Xi_{\Lambda_S\setminus\Lambda_{R-1}}$ counts the clusters in $\Lambda_S\setminus\Lambda_{R-1}$ which do not intersect (either component of) the boundary $\partial(\Lambda_{S}\setminus\Lambda_{R-1})=\partial\Lambda_S\cup\partial\Lambda_R$, it does not depend on $f|_{\Lambda_{R-1}}$. Therefore combining the last two equations and applying Proposition~\ref{p:derpivint} once more shows that
\begin{displaymath}
    \lvert P_\infty(y)-P_R(y)\rvert\leq\sup_{S>R}\lvert \partial_{y_1}\dots\partial_{y_m}\E[\Xi_{\Lambda_S,\partial\Lambda_R}]\rvert =\sup_{S>R}\lvert\E[d_y\Xi_{\Lambda_S,\partial\Lambda_R}|f(\underline{y})=\ell]\rvert\varphi_{f(\underline{y})}(\ell).
\end{displaymath}
Using the alternative expression for pivotal intensities given in Remark~\ref{r:pivint}, the Cauchy-Schwarz inequality, and \eqref{e:conv_4}, the latter expression is at most
\begin{equation}\label{e:conv1}
\begin{aligned}
    c_{\ell,m}\sup_{S>R}\E\big[&(d_{\underline{y}}\Xi_{\Lambda_S,\partial\Lambda_{R}})^2\big|f(\underline{y})=\ell\big]^{1/2}\varphi_{f(\underline{y})}(\ell).
\end{aligned}
\end{equation}
We observe that if $d_{\underline{y}}\Xi_{\Lambda_S,\partial\Lambda_R}\neq 0$ then there must exist a bounded cluster of $\{f>\ell\}\setminus\underline{y}$ (or $\{f<\ell\}\setminus\underline{y}$) connecting a neighbour of some $y_i$ to $\partial\Lambda_R$. Hence, since $\lvert d_{\underline{y}}\Xi_{\Lambda_S,\partial\Lambda_R}\rvert\leq c_{m,d}$, 
\begin{align*}
    \E\big[(d_{\underline{y}}\Xi_{\Lambda_S,\partial\Lambda_{R}})^2\big|f(\underline{y})=\ell\big]\leq c_{m,d}\Big(\P&\left[\widetilde{\mathrm{Arm}}_{d_\infty(\underline{y},\partial\Lambda_R);\underline{y}}(\{f>\ell\})|f(\underline{y})=\ell\right]\\
    &+\P\left[\widetilde{\mathrm{Arm}}_{d_\infty(\underline{y},\partial\Lambda_R);\underline{y}}(\{f<\ell\})|f(\underline{y})=\ell\right]\Big).
\end{align*}
The first statement of the lemma now follows from pinned truncated arm decay (Proposition~\ref{p:ptad}) and boundedness of $\varphi_{f(\underline{y})}(\ell)$ (see \eqref{e:bounds10}).

The statement in the case $m=1$ follows from a very similar argument, so we will only highlight the points of difference. Given $y\in\Lambda_R$, we relabel our coordinate axes so that one of the $(d-1)$-dimensional boundary faces of $\partial\Lambda_R$ which $y$ is closest to lies in $[-R,-R+1]\times\Z^{d-1}$ (by Assumption~\ref{a:rv} this relabelling does not affect the distribution of $f$). Then by stationarity of $f$
\begin{displaymath}
    P_R(y)=\E[d_{y+Re_1}\Xi_{\Lambda_R+Re_1}(f-\ell)|f(y+Re_1)=\ell]\varphi_{f(y)}(\ell)
\end{displaymath}
where $e_1$ denotes the unit vector in the positive direction of the first axis. We note that $\Lambda_R+Re_1\subset\Z^d_+$ and that the first coordinate of $y+Re_1$ is $d_\infty(y,\partial\Lambda_R)$. The remainder of the argument proceeds as before; we express
\begin{displaymath}
    P_\infty^{\mathrm{H}}(d_\infty(y,\partial\Lambda_R))=\lim_{S\to\infty}\E[d_{y+Re_1}\Xi_{\Lambda_S+Se_1}(f-\ell)|f(y+Se_1)=\ell]\varphi_{f(y)}(\ell),
\end{displaymath}
decompose $\Xi_{\Lambda_S+Se_1}$ into different cluster counts and use pinned truncated arm decay to control the derivative of the expected difference.

Finally we consider the pivotal intensities for the truncated cluster count. We let $\Xi^{\leq r}_D$ denote the analogue of $\Xi_D$ for which we count only the clusters of diameter at most $r$. By an analogous argument to that given for $P_R$,
\begin{displaymath}
    \lvert P_{R;\leq r}(y)-P_{\infty;\leq r}(y)\rvert\leq\sup_{S>R}\left\lvert\partial_{y_1}\dots\partial_{y_m}\E\left[\Xi_{\Lambda_S,\partial\Lambda_{R}}^{\leq r}\right]\right\rvert.
\end{displaymath}
Since $\Xi^{\leq r}_{\Lambda_S,\partial\Lambda_R}$ counts only clusters which intersect $\partial\Lambda_R$ and have diameter at most $r$, it does not depend on the value of $f$ at points of distance greater than $r+2$ from $\partial\Lambda_R$. In particular if $d_\infty(y,\partial\Lambda_R)>r+2$ then $\Xi^{\leq r}_{\Lambda_S,\partial\Lambda_R}$ does not depend on at least one of the $y_i$ and so the above expression is zero. Since $d_y\Xi^{\leq r}_{\Lambda_R}$ is bounded by a constant depending only on $m$ and $d$, the same is true of $P_{R;\leq r}$ and $P_{\infty;\leq r}$, which gives the last statement.
\end{proof}

\begin{proof}[Proof of Lemma \ref{l:decay}]
Since $P_R(y)\to P_\infty(y)$ as $R\to\infty$, we have
\begin{align*}
    \max\{\lvert P_R(y)\rvert,\lvert P_\infty(y)\rvert\}&\leq\sup_{S\geq R}\E[d_{\underline{y}}\Xi_{\Lambda_S}H^{\tilde{\alpha}^y}_{f(\Lambda_S)}(f(\Lambda_S))|f(\underline{y})=\ell]\varphi_{f(\underline{y})}(\ell)\\
    &\leq c_{m,\ell,d,\lambda_{\mathrm{min}}}\sup_{S\geq R}\E\left[\big(d_{\underline{y}}\Xi_{\Lambda_S}\big)^2\middle| f(\underline{y})=\ell\right]^{1/2}
\end{align*}
where the second inequality follows from Cauchy-Schwarz and \eqref{e:conv_4}.

We next claim that
\begin{displaymath}
    d_{\underline{y}}\Xi_{\Lambda_S}\neq 0\quad\text{implies}\quad \widetilde{\mathrm{Arm}}_{(\mathrm{diam}_\infty(y)-2)/m;\underline{y}}(\{f>\ell\})\cup\widetilde{\mathrm{Arm}}_{(\mathrm{diam}_\infty(y)-2)/m;\underline{y}}(\{f<\ell\}).
\end{displaymath}
Assuming the claim, the first statement of the lemma follows from pinned truncated arm decay (Proposition~\ref{p:ptad}) and the fact that $\lvert d_{\underline{y}}\Xi_{\Lambda_S}\rvert\leq 2^md$.

We now prove the claim. Fix $S>0$, let $E\subset\Z^d$ be some configuration, and let $\Xi^+_S(E)$ denote the number of clusters of $E$ in $\Lambda_S$ which do not intersect $\partial\Lambda_S$.
If $y_1$ and $y_2$ are not joined by a cluster of $E\setminus\{y_1,y_2\}$ then they lie in separate clusters of $E\cup\{y_1,y_2\}$, which we denote by $C_1$ and $C_2$ respectively. Then
\begin{displaymath}
d_{y_1}d_{y_2}\Xi^+_S(E)=d_{y_2}d_{y_1}\Xi^+_S(C_1\cap E)+d_{y_1}d_{y_2}\Xi^+_S(C_2\cap E)=0
\end{displaymath}
 where the first equality follows because $C_1$ and $C_2$ are separated and the second uses the fact that $d_{y_1}\Xi^+_S(C_1\cap E)$ is unaffected by adding/removing points in $C_2$. If $y_1$ and $y_2$ are joined by a cluster of $E\setminus\{y_1,y_2\}$ which intersects the boundary of $\Lambda_S$ then both $d_{y_1}\Xi^+_S(E\cup\{y_2\})$ and $d_{y_1}\Xi^+_S(E\setminus\{y_2\})$ are equal to minus the number of clusters of $E\setminus\{y_1,y_2\}$ which contain a neighbour of $y_1$ and  do not intersect $\partial\Lambda_S$. Hence $d_{y_2}d_{y_1}\Xi^+_S(E)=0$. Combining these two observations, we conclude that if $d_{y_1}d_{y_2}\Xi^+_S(E)\neq 0$ then $y_1$ and $y_2$ must be joined by a cluster of $E\setminus\{y_1,y_2\}$ which does not intersect $\partial\Lambda_S$.

Now choose $y_1$ and $y_2$ to be two points in $\underline{y}$ with maximum distance and suppose that $d_{\underline{y}}\Xi^+_S(\{f\geq\ell\})\neq 0$. By definition of the discrete derivative, we can expand $d_{\underline{y}}\Xi^+(\{f\geq\ell\})$ into a linear combination of terms of the form $d_{y_1}d_{y_2}\Xi^+(E)$ for $E\subseteq \{f>\ell\}\cup\{y_3,\dots,y_m\}$. Hence there exists such an $E$ for which $d_{y_1}d_{y_2}\Xi^+_S(E)\neq 0$ and so by the previous paragraph there exists some path in a bounded cluster of $E\subseteq\{f\geq\ell\}\cup\{y_3,\dots,y_m\}$ with diameter at least $\mathrm{diam}_\infty(y)-2$. Removing the points $\{y_3,\dots,y_m\}$ divides this into at most $m-1$ sub-paths, at least one of which must have diameter at least $(\mathrm{diam}_\infty(y)-2)/(m-1)$. In particular
\begin{displaymath}
    \widetilde{\mathrm{Arm}}_{(\mathrm{diam}_\infty(y)-2)/m;\underline{y}}(\{f>\ell\})
\end{displaymath}
holds. The same argument applies to $\Xi^-$, the number of clusters of $\{f<\ell\}$. Since $\Xi=\Xi^++\Xi^-$, if $d_{\underline{y}}\Xi\neq 0$ then the same must be true of $\Xi^+$ or $\Xi^-$. Combining these observations completes the proof of the claim.

Turning to the joint pivotal intensity, applying H\"older's inequality to Definition~\ref{d:jpi}
\begin{equation}\label{e:decay1}
\begin{aligned}
    \lvert P^t_{R;>r}(x,y)\rvert\leq\Big(\tilde{\E}^t[\lvert d_{\underline{x}}\Xi_{R;>r}\rvert^4]\Big)^{\frac{1}{4}}\;&\Big(\tilde{\E}^t[\lvert d_{\underline{y}}\Xi^t_{R;>r}\rvert^4]\Big)^{\frac{1}{4}}\cdot\\
    &\Big(\tilde{\E}^t\left[\big(H^{\tilde{\alpha}^x,\tilde{\alpha}^y}_{f(\underline{x}),f^t(\underline{y})}(f(\underline{x}),f^t(\underline{y}))\big)^2\right]\Big)^{\frac{1}{2}}\varphi_{f(\underline{x}),f^t(\underline{y})}(\ell,\ell)
\end{aligned}
\end{equation}
where $\tilde{\E}^t$ denotes expectation conditioning on $f(\underline{x})=\ell$ and $f^t(\underline{y})=\ell$. If $\underline{x}\cap\underline{y}=\emptyset$ then $(f(\underline{x}),f^t(\underline{y}))$ is non-degenerate, uniformly in $\underline{x}$, $\underline{y}$ and $t$. In this case a straightforward argument using pointwise bounds on Hermite polynomials (Proposition~\ref{p:mhpmain1}) and Gaussian regression shows that the last line of \eqref{e:decay1} is bounded by a constant depending only on $m$, $\ell$ and the distribution of $f$ (which is uniform over $\ell$ in any compact interval). If $\underline{x}\cap\underline{y}\neq\emptyset$, then Proposition~\ref{p:mhpmain} states that the second line of \eqref{e:decay1} is bounded by $c(1-t)^{-m+1/2}$ where $c>0$ depends only on $m$ and the distribution of $f$. By Lemma~\ref{l:bounds1} $\lvert d_{\underline{y}}\Xi_{R;>r}\rvert\leq 2^md$, and we then have for any $\delta\in(0,1/2)$
\begin{displaymath}
    \lvert P^t_{R;>r}(x,y)\rvert\leq c_{d,m,\ell} (1-t)^{-m+1/2}\Big(\tilde{\P}^t(d_{\underline{x}}\Xi_{R;>r}\neq 0)\Big)^{\frac{\delta}{2m}}\Big(\tilde{\P}^t(d_{\underline{y}}\Xi^t_{R;>r}\neq 0)\Big)^{\frac{\delta}{2m}},
\end{displaymath}
using the fact that $1/4\ge \delta/(2m)$. From the argument given earlier in the proof, if $d_{\underline{x}}\Xi_{R;>r}\neq 0$, then one of the points in $\underline{x}$ must be the origin of an arm event of length at least $(\mathrm{diam}_\infty(\underline{x})-2)/m$. Since we count only (finite) clusters of length at least $r$, a near-identical argument shows that
\begin{displaymath}
    d_{\underline{x}}\Xi_{R;>r}\neq 0\quad\text{implies}\quad\widetilde{\mathrm{Arm}}_{r\vee(\mathrm{diam}_\infty(x)-2)/m;\underline{y}}(\{f>\ell\})\cup\widetilde{\mathrm{Arm}}_{r\vee(\mathrm{diam}_\infty(x)-2)/m;\underline{y}}(\{f<\ell\}).
\end{displaymath}
A similar property holds for $d_{\underline{y}}\Xi^t_{R;>r}$ and so applying Proposition~\ref{p:ptad} completes the proof of the second statement of the lemma. (Note that this proposition also applies to arm events for $f^t$ if we simply swap the roles of the fields $f$ and $\tilde{f}$ used in the interpolation.)
\end{proof}

\subsection{Localisation of the components}
We now complete the proof of Propositions \ref{p:statwce}--\ref{p:statwcetrunc}:

\begin{proof}[Proof of Proposition \ref{p:statwce}]
The claimed properties of $P_\infty$ follow from Lemmas~\ref{l:sym} and~\ref{l:decay}, so it remains to prove \eqref{e:varbound}. Abbreviate 
\begin{align}
\label{e:em}
E_m & := Q_m[N_R(\ell)] -  \frac{1}{m!}\sum_{x_1,\dots,x_m\in \Lambda_R}\wick{f(x_1)\cdots f(x_m)} P_\infty(x_1,\ldots,x_m) \\
\nonumber & =   \frac{1}{m!}\sum_{x_1,\dots,x_m\in \Lambda_R}\wick{f(x_1)\cdots f(x_m)} \Gamma_R(x_1,\ldots,x_m)  
\end{align}
where $\Gamma_R(x) = P_R(x) - P_\infty(x)$. By the diagram formula (Theorem \ref{t:df})
\begin{align*}
    \Var[E_m]  &= \frac{1}{(m!)^2} \sum_{x, y \in (\Lambda_R)^m} \Big( \sum_{\sigma \in S_m} \prod_{i=1}^m G(x_i - y_{\sigma(i)} ) \Big) \Gamma_R(x) \Gamma_R(y)\\
    &=\frac{1}{m!} \sum_{x, y \in (\Lambda_R)^m} \Big( \prod_{i=1}^m G(x_i - y_i ) \Big) \Gamma_R(x) \Gamma_R(y),
\end{align*}
where the second equality follows from permutation invariance of $\Gamma_R$. By the first item of Lemma~\ref{l:conv} and Lemma~\ref{l:decay}, for every $x_1 \in \Lambda_R$
\[  \sum_{x_2,\ldots,x_m \in \Lambda_R} |\Gamma_R(x)| \le \gamma( d_\infty(x_1,\partial \Lambda_R) )   \] 
where $\gamma(k) = c_1 e^{-c_2 k^\rho}$. Then by essentially elementary arguments (formalised in Lemma~\ref{l:rv3} and the second item of Lemma~\ref{l:rv4}), we have that 
\[ \Var[E_m] = O \big( R^{\max\{ 2(d-1) - m(d-2), d-1\}} (\log R)^{\id_{m(d-2)=d-1}} \big), \]
where the logarithmic term essentially arises as the sum of $|G(x)|^m \sim |x|^{-m(d-2)} = |x|^{-(d-1)}$ over the domain $\Lambda_R$. This yields the result.
\end{proof}

\begin{proof}[Proof of Proposition \ref{p:varasymp}]
The fact that 
\[ | P^H_\infty(k) - P_\infty(0) | \le c_1 e^{-c_2 k^\rho}   \]
follows by setting $y = (-R + k, 0,\ldots,0)$ in the first item of Lemma~\ref{l:conv} and taking $R \to \infty$. Abbreviate
\[ E_1^H  := Q_1[N_R(\ell)] -  \sum_{x \in \Lambda_R} f(x) P_\infty^H( d_\infty(x,\partial\Lambda_R) )   =   \sum_{x \in \Lambda_R} f(x) \Gamma_R(x) \]
where $\Gamma_R(x) =  P_R(x) - P_\infty^H( d_\infty(x,\partial\Lambda_R))$. By Lemma~\ref{l:conv} once more, 
\[| \Gamma_R(x)| \le \gamma( d_\infty(x, F_R^{d-2} ) ) \]
where $\gamma(k) = c_1 e^{-c_2 k^\rho}$. Applying the second item of Lemma \ref{l:rv4}, we have that 
\[ \Var[E_1^H] = O \big( R^{d-2} (\log R) \big)  \]
as required.
\end{proof}

\begin{proof}[Proof of Proposition \ref{p:statwcequal}]
Let $E_1$ and $\Gamma_R$ be defined as in \eqref{e:em}. By the first item of Lemma~\ref{l:conv}, $|\Gamma_R(x)| \le \gamma(k)$ for a function $\gamma$ satisfying $\gamma(k) \to 0$. Applying the first item of Lemma \ref{l:rv4} gives the result.
\end{proof}

Before proving Proposition \ref{p:statwcetrunc} we need a tail estimate for the chaos expansion of the truncated cluster count:

\begin{lemma}\label{l:trunctail}
    Let $\epsilon>0$ and $r>1$ be given. Then for sufficiently large $M\in\N$,
    \begin{displaymath}
        \limsup_{R\to\infty}R^{-d}\Var\Big[\sum_{m>M}Q_m\big[N_{R;\leq r}(\ell) \big]\Big]\leq \epsilon.
    \end{displaymath}
\end{lemma}
\begin{proof}
    Given $R>r>1$ and $x\in\Lambda_R$, we define
    \begin{displaymath}
       \Theta_x =  \Theta_x(R,r):=\frac{1}{\lvert C_x\rvert}\ind_{A(x,r,R)}
    \end{displaymath}
    where $C_x$ is the cluster of $\{f>\ell\}$ or $\{f<\ell\}$ which contains $x$ and $A(x,r,R)$ is the event that $C_x$ does not intersect $\partial\Lambda_R$ and has diameter at most $r$. Note that $\Theta_x$ is determined by $f|_{x+\Lambda_{r+1}}$. Observe also that the identity $N_{R;\le r}(\ell)=\sum_{x\in\Lambda_R}\Theta_x$ holds by definition, from which we have
    \begin{displaymath}
        \Var \big[Q_m[N_{R;\leq r}(\ell)] \big]=\sum_{x,y\in\Lambda_R}\mathrm{Cov}\big[Q_m[\Theta_x],Q_m[\Theta_y]\big].
    \end{displaymath}
    We will control separately the diagonal and off-diagonal contributions to this sum. 

    For the diagonal contribution, fixing a large parameter $L>1$ to be specified later, we have
    \begin{equation}\label{e:trunctail1}
        \sum_{\substack{x,y\in\Lambda_R\setminus\Lambda_{R-r-L-1}\\\lvert x-y\rvert\leq L}}\sum_{m>M}\mathrm{Cov}\big[Q_m[\Theta_x],Q_m[\Theta_y]\big]\leq  \sum_{\substack{x,y\in\Lambda_R\setminus\Lambda_{R-r-L-1}\\\lvert x-y\rvert\leq L}}1 \leq c_d L^d (r+L) R^{d-1}
    \end{equation}
    since $\sup_x\Var[\Theta_x]\leq 1$. We also have    
    \begin{equation}\label{e:trunctail2}
    \begin{aligned}
        \sum_{\substack{\{x,y\}\cap\Lambda_{R-r-L-1}\neq\emptyset\\\lvert x-y\rvert\leq L}}\sum_{m>M}&\mathrm{Cov}\big[Q_m[\Theta_x],Q_m[\Theta_y]\big]\\
        &\qquad\leq (2R+1)^d(2L+1)^d\sup_{x\in\Lambda_{R-r-1}}\Var\left[\sum_{m>M}Q_m[\Theta_x]\right].
    \end{aligned}
    \end{equation}
    Since $\Theta_x$ is determined by $f|_{x+\Lambda_{r+1}}$, for $x\in\Lambda_{R-r-1}$ the distribution of $\Theta_x$ does not depend on $R$. By stationarity of $f$, this distribution also does not depend on $x$, and so
    \begin{equation}\label{e:trunctail3}
        \sup_{x\in\Lambda_{R-r-1}}\Var\left[\sum_{m>M}Q_m[\Theta_x]\right]=\Var\left[\sum_{m>M}Q_m[\Theta_0]\right]\leq \epsilon_M
    \end{equation}
    where $\epsilon_M\to 0$ as $M\to\infty$ uniformly over $R$.

    For the off-diagonal contribution, using the chaos expansion for each $\Theta_x$ (Theorem~\ref{t:cluster_wce}) and the diagram formula, we have
    \begin{equation}\label{e:trunctail4}
    \begin{aligned}
        \sum_{\substack{x,y\in\Lambda_R\\\lvert x-y\rvert> L}}\sum_{m>M}\mathrm{Cov}&[Q_m[\Theta_x],Q_m[\Theta_y]]=\sum_{\substack{x,y\in\Lambda_R\\\lvert x-y\rvert> L}}\sum_{m>M}\frac{1}{m!}\sum_{u,v\in\Lambda_R^m}P_x(u)P_y(u)\prod_{i=1}^mG(u_i-v_i)
    \end{aligned}
    \end{equation}
    where $P_x(u)$ denotes the pivotal intensity associated with $\Theta_x$. Since $\Theta_x$ is determined by $f|_{x+\Lambda_{r+1}}$ and discrete derivatives commute, we have $
        d_{u_1}\dots d_{u_m}\Theta_x=0$ if $d_\infty(u_i,x)>r+1$ for any $i$. Hence by definition of the pivotal intensities, $P_x$ is supported in $x+\Lambda_{r+1}$. Then since $|G(x)| \le c_d |x|^{(d-2)}$, and by the bound on pivotal intensities in Lemma~\ref{l:bounds1}, for every $\rho>0$ \eqref{e:trunctail4} is bounded above by
    \begin{displaymath}
      \sum_{\substack{x,y\in\Lambda_R\\\lvert x-y\rvert> L}}\sum_{m>M}\sum_{\substack{u\in x+(\Lambda_{r+1})^m\\v\in y+(\Lambda_{r+1})^m}}e^{c_{r,\ell}m}\rho^m\prod_{i=1}^m\frac{c_d\lvert u_i-v_i\rvert^{-(d-2)}}{\rho} .
    \end{displaymath}
 Choosing 
    \begin{displaymath}
        \rho=\frac{1}{2 |\Lambda_{r+1}|^2}e^{-c_{r,\ell}} \qquad \text{and}\qquad L>4r 
    \end{displaymath}
    the above sum is bounded by
    \begin{align*}
        \sum_{m>M}2^{-m}\sum_{\substack{x,y\in\Lambda_R\\\lvert x-y\rvert> L}}\Big(\frac{c_d(\lvert x-y\rvert-2(r+1))^{-(d-2)}}{\rho}\Big)^m\leq 2^{-M} R^d \max_{m \ge M} \sum_{|y| > L} \big(c_{d,r,\ell} \lvert y\rvert^{-(d-2)}\big)^m
    \end{align*}
    for some $c_{d,r,\ell}>0$. By choosing $L$ sufficiently large (depending only on $d,r,\ell$), the sum in the above expression can be bounded by $1$. Combining this with \eqref{e:trunctail1}--\eqref{e:trunctail3} proves the statement of the lemma.
\end{proof}

\begin{proof}[Proof of Proposition \ref{p:statwcetrunc}]
The claimed properties of $P_{\infty ; \le r}$ are given in Lemmas \ref{l:stationarytrunc} and \ref{l:bounds}, so it remains to prove the variance bound. Fix a truncation parameter $r>1$ and $\delta\in(0,1/2)$. Recalling that $N_{R;>r}$ and $N_{R;\leq r}$ denote the number of level clusters of diameter greater than $r$ and at most $r$ respectively, by linearity of projection onto each chaos
\begin{equation}\label{e:statwcetrunc1}
 \sum_{m\geq m_0} Q_m[N_R(\ell)]-\sum_{m_0\leq m\leq M}\overline{Q}_m(R,r) = I_1 + I_2 + I_3 
 \end{equation}
 where
 \[ I_1  :=\sum_{m>M}Q_m[N_{R;\leq r}(\ell)] \ , \quad I_2 :=  \sum_{m\geq m_0}Q_m[N_{R;>r}(\ell)] \]
 and
 \[ I_3 :=  \sum_{m_0\leq m\leq M}Q_m[N_{R;\leq r}(\ell)]-\overline{Q}_m(R,r) . \]
The variance of $I_1$ is controlled by Lemma \ref{l:trunctail}, so it remains to bound $I_2$ and $I_3$.

Applying Proposition \ref{p:wcerror} to $N_{R; > r}(\ell)$, and by Lemma~\ref{l:decay},
\[ \textrm{Var}[I_2]  \le  \tau_{m_0,\delta} \sum_{x,y \in (\Lambda_R)^{m_0} } \Big( \prod_{i=1}^{m_0} G(x_i - y_i)  \Big) |\Gamma_R(x)| |\Gamma_R(y)| \]
where, as in the proof of Lemma \ref{l:jpi},
\begin{align*} 
\tau_{m_0,\delta}  & := \int_0^1 \int_0^{t_0}\dots\int_0^{t_{m_0-2}} ( 1-t_{m_0-1})^{-m_0+1/2-\delta } \;dt_{m_0-1}\dots dt_0 \\
& \le \int_0^1  \frac{1}{(1-s)^{1/2+\delta}} \, ds  = c_\delta  < \infty   
\end{align*}
and
\[ |\Gamma_R(x)| \le   c_1 \min\Big\{  e^{-c_2 r^\rho} ,  e^{-c_2 \textrm{diam}_\infty(\underline{x})^\rho}  \Big\}. \] 
Note also that
\[ \bar{\Gamma}_R(x_1) :=  \sum_{x_2,\ldots,x_{m_0} \in (\Lambda_R)^{m_0-1}} |\Gamma_R(x_1,x_2,\ldots,x_{m_0})|  \le c_3 r^{d(m_0-1)}  e^{-c_2 r^\rho} < \infty. \]
Recalling that $m_0 (d-2) > d$, applying Lemma \ref{l:rv3} gives that  
\begin{equation}
    \label{e:varbound3}
\limsup_{R\to\infty}R^{-d}\Var[ I_2 ]   \le  \eps_r  
\end{equation}
for some $\eps_r \to 0$ as $r \to \infty$.

Next, using the chaos expansion for level-set functionals (Theorem~\ref{t:cluster_wce}) and the diagram formula (Theorem~\ref{t:df}), $\textrm{Var}[I_3]$ equals
\begin{displaymath}
   \sum_{m_0\leq m\leq M}\frac{1}{m!}\sum_{x,y\in\Lambda_R^m}\Big(\prod_{i=1}^mG(x_i-y_i)\Big)(P_{\infty;\leq r}(x)-P_{R;\leq r}(x))(P_{\infty;\leq r}(y)-P_{R;\leq r}(y)).
\end{displaymath}
By the second items of Lemmas~\ref{l:conv} and~\ref{l:decay} the terms in the innermost sum will be zero unless each point of $x$ and $y$ is within distance $r$ of $\partial\Lambda_R$ and the diameters of $x$ and $y$ are at most $r$. Hence 
\begin{equation}
\label{e:varbound4}
   \textrm{Var}[I_3] \le \sum_{m_0\leq m\leq M}c_{m,d} \sum_{x_1,y_1\in\Lambda_{R-2r}}(2r+1)^{2d}\big((\lvert x_1-y_1\rvert-2r)\vee 1\big)^{-m(d-2)}\leq c^{\prime}R^{d-1}
\end{equation}
where $c^\prime>0$ may depend on $M$, $d$ and $r$ and we have used the fact that $m_0(d-2)>d$.

Combining \eqref{e:statwcetrunc1}--\eqref{e:varbound4} and Lemma~\ref{l:trunctail}, choosing first $r>1$ and then $M\in\N$ sufficiently large yields the result.
\end{proof}

\medskip

\section{On the cluster density functional}
\label{s:cdf}

In this section we study the cluster density functional $\mu(\ell)$ defined by the law of large numbers \eqref{e:lln} (and proved to exist in Proposition \ref{p:mu} below). Recall that \cite{ps22} has shown that  $\ell \mapsto \mu(\ell)$ is real-analytic on $\R \setminus \{-\ell_c,\ell_c\}$. Our main result gives an expression for the derivatives of $\mu$ in terms of the stationary pivotal intensities $P_\infty$ introduced in Section \ref{s:sl}. This also constitutes an alternative proof that $\mu$ is smooth on $\R \setminus \{-\ell_c,\ell_c\}$.

\begin{proposition}
\label{p:mudiff}
The function $\mu$ is smooth on $\R \setminus \{-\ell_c,\ell_c\}$ and continuously differentiable on $\R$. Moreover if either (i) $m \ge 1$ and $\ell \in \R \setminus \{-\ell_c,\ell_c\}$, or (ii) $m=1$ and $\ell = \pm\ell_c$,
    \begin{equation}\label{e:mudiff}
        \mu^{(m)}(\ell) = (-1)^m\sum_{x_2,\dots,x_m\in\Z^d} P_\infty(0,x_2,\dots,x_m) 
    \end{equation}
    where $P_\infty$ is the stationary pivotal intensity at level $\ell$ in Definition \ref{d:statpivint}.
\end{proposition}

\noindent 
Using this expression, we will establish an important qualitative feature of $\mu'(\ell)$:

\begin{lemma}
\label{l:high}
  There exists $\ell_0>0$ such that for all $\lvert\ell\rvert\geq\ell_0$, $\mu^\prime(\ell)\neq 0$. In particular, the number of critical points of $\mu$ outside any neighbourhood of $\{-\ell_c,\ell_c\}$ is finite.
\end{lemma}

We take up the notation $\Xi_D$ and $P_R$ introduced in Section \ref{s:sl}, and for simplicity we abbreviate $\Xi_R = \Xi_{\Lambda_R}$. It will also be convenient to adjust our notation for $P_R$ by including an argument indicating the level, that is, writing $P_R(\ell;x_1,\dots,x_m)$ instead of $P_R(x_1,\dots,x_m)$, and similarly for $P_\infty$.

\subsection{Existence of the cluster density}
\label{ss:lln}

For completeness we first confirm the existence of~$\mu$, which follows from classical arguments for Bernoulli percolation (see also \cite{ns16} for an extension to smooth Gaussian fields).  

\begin{proposition}\label{p:mu}
    Let $f:\Z^d\to\R$ be a stationary ergodic Gaussian field. Define
    \begin{equation}
        \label{e:muform}
    \mu(\ell) = \E \big[ |C_0(\ell)|^{-1}  \id_{|C_0(\ell)| < \infty} \big] \in [0,1] 
    \end{equation} 
      where $|C_0(\ell)|$ denotes the cardinality of the level-set cluster containing the origin. Then for each $\ell\in\R$,
    \begin{displaymath}
        \lim_{R \to \infty} \frac{N_R(\ell)}{\lvert\Lambda_R\rvert} = \mu(\ell)
    \end{displaymath}
    almost surely and in $L^1$. Moreover if $f$ is non-degenerate on $\Lambda_1$ then $\mu(\ell) \in (0,1)$.

    The same statement holds if we replace $N_R(\ell)$ with $N_R^\pm(\ell)$ and $\mu(\ell)$ with
    \begin{displaymath}
        \mu^\pm(\ell):= \E \big[ |C_0(\ell)|^{-1}  \id_{|C_0(\ell)| < \infty}\ind_{\mathrm{sgn}(f(0)-\ell)=\pm 1} \big].
    \end{displaymath}
\end{proposition}
\begin{proof}
The proof of convergence is identical to that given in \cite[Theorem (4.2)]{gri99}, which applies the ergodic theorem to the decomposition of the cluster count in \eqref{e:stationary_approx}. To verify the second statement, if $f$ is non-degenerate on $\Lambda_1$ there is a positive probability that $f>\ell$ at the origin but $f<\ell$ at each of its neighbours, so that $|C_0(\ell)|=1$, and also a positive probability that $f|_{\Lambda_1}>\ell$, so that $|C_0(\ell)|\ge 2$. Given \eqref{e:muform} this completes the proof.
\end{proof}

\subsection{Proof of Proposition \ref{p:mudiff}}
\label{ss:mup}

We consider non-critical and critical levels separately.

\smallskip
\noindent \textbf{Non-critical levels.}  By Proposition~\ref{p:derpivint}, the function
    \begin{displaymath}
        (t_x)_{x\in\Lambda_R}\mapsto \E\Big[N_R\Big(f-\ell+\sum_{x\in\Lambda_R}t_x\ind_x\Big)\Big]
    \end{displaymath}
    is smooth on $\R^{\Lambda_R}$, and hence by the chain rule and Proposition~\ref{p:derpivint}, for $m\in\N$
    \begin{align*}
        \frac{d^m}{d\ell^m}\E[N_R(f-\ell)]&=\Big(\sum_{x\in\Lambda_R}\frac{d}{dt_x}\Big)^m\E\left[N_R\left(f-\ell-\sum t_x\ind_x\right)\right]\Big|_{t_x=0\;\forall x}\\
        &=(-1)^m\sum_{x_1,\dots,x_m\in\Lambda_R}P_R(\ell;x_1,\dots,x_m).
    \end{align*}
    Given a function $h:\R\to\R$ and $\epsilon>0$, we define
    \begin{displaymath}
        D_\epsilon h(\ell)=h(\ell+\epsilon)-h(\ell).
    \end{displaymath}
    Then by iterating the fundamental theorem of calculus, we have for $\epsilon_1,\dots,\epsilon_m>0$
    \begin{align*}
       &  D_{\epsilon_1}\dots D_{\epsilon_m}\E[N_R(f-\ell)] \\
       & \qquad =(-1)^m\int_0^{\epsilon_1}\dots\int_{0}^{\epsilon_m}\sum_{x_1,\dots,x_m\in\Lambda_R}P_R\Big(\ell+\sum_{i=1}^ms_i;x_1,\dots,x_m\Big)\;ds_m\dots ds_1.
    \end{align*}
    We now assume that $\ell\neq\pm\ell_c$ and let $\delta\in(0,1)$. Using the facts that the pivotal intensities are bounded (Lemma~\ref{l:decay}) and converge to their stationary counterparts (Lemma~\ref{l:conv}), the previous expression is equal to
    \begin{displaymath}
        (-1)^m\int_0^{\epsilon_1}\dots\int_{0}^{\epsilon_m}\sum_{x_1\in\Lambda_{R-2R^\delta}}\sum_{x_2,\dots,x_m\in\Lambda_{R-R^\delta}}P_\infty\Big(\ell+\sum s_i;x_1,\dots,x_m\Big)\;ds_m\dots d s_1 +E_R
    \end{displaymath}
    where the error $E_R$ satisfies
    \begin{displaymath}
        \lvert E_R\rvert\leq c_{d,\ell,m}\Big(\prod_i\epsilon_i\Big) (R^{d-1+\delta} + R^{dm}e^{-cR^{\delta\rho}})
    \end{displaymath}
    and $c,\rho>0$ are taken from Lemma~\ref{l:conv}. Let $\sum^\ast_R:=\sum_{x_1\in\Lambda_{R-2R^\delta}}\sum_{x_2,\dots,x_m\in\Lambda_{R-R^\delta}}$. Since $P_\infty$ is stationary
    \begin{align*}
        \sup_{t\in[0,\sum_i\epsilon_i]}\bigg\lvert\frac{1}{\lvert\Lambda_{R-2R^\delta}\rvert}\sum\nolimits^\ast_R P_\infty\Big(\ell+t;x_1,\dots,x_m\Big)-\sum_{x_2,\dots,x_m\in\Z^d}P_\infty\Big(\ell+t;0,x_2,\dots,x_m\Big)\bigg\rvert\\
        \leq \sup_{t\in[0,\sum_i\epsilon_i]}\sum_{x_2,\dots,x_m\notin\Lambda_{R^\delta}}\Big| P_\infty \big(\ell+t;0,x_2,\dots,x_m \big)\Big|.
    \end{align*}
    By Lemma~\ref{l:decay}, the latter expression decays to zero as $R\to\infty$ provided that $\sum_i\epsilon_i$ is sufficiently small. Combining the last four displayed equations, we have
    \begin{multline*}
        \lim_{R\to\infty}\frac{1}{\lvert\Lambda_R\rvert}D_{\epsilon_1}\dots D_{\epsilon_m}\E[N_R(f-\ell)]\\
        =(-1)^m\int_0^{\epsilon_1}\dots\int_0^{\epsilon_m}\sum_{x_2,\dots,x_m\in\Z^d}P_\infty \Big(\ell+\sum s_i;0,x_2,\dots,x_m \Big)\;ds_m\dots ds_1.
    \end{multline*}
    We now claim that the integrand above is continuous in the level. Assuming this claim, we have
    \begin{align*}
        \mu^{(m)}(\ell)&=\lim_{\epsilon_1,\dots,\epsilon_m\to0}\epsilon_1^{-1}\dots\epsilon_m^{-1}\lim_{R\to\infty}\frac{1}{\lvert\Lambda_R\rvert}D_{\epsilon_1}\dots D_{\epsilon_m}\E[N_R(f-\ell)] \\ 
        &=(-1)^m\sum_{x_2,\dots,x_m\in\Z^d}P_\infty(\ell;0,x_2,\dots,x_m)
    \end{align*}
    completing the proof of the proposition.

    It remains to prove the claim. By Lemma~\ref{l:decay}, the pivotal intensities are bounded uniformly in the level by an expression that is summable over $x_2,\dots,x_m\in\Z^d$. Hence by dominated convergence, it is enough to show that for any fixed $x\in(\Z^d)^m$, the pivotal intensity $P_\infty(\ell;x)$ is continuous in $\ell$ on $\R\setminus\{-\ell_c,\ell_c\}$. For fixed $x$ and $R>1$, by applying Gaussian regression to the definition of the pivotal intensity, it follows that $\ell\mapsto P_R(\ell;x)$ is continuous. Since $P_\infty(\ell,x)$ can be approximated by $P_R(\ell;x)$ uniformly over $\ell$ on compacts subsets of $\R\setminus\{-\ell_c,\ell_c\}$ (Lemma~\ref{l:conv}) it follows that $\ell\mapsto P_\infty(\ell;x)$ is continuous away from $\pm\ell_c$, completing the proof.

\smallskip
\vspace{0.2cm}

\noindent \textbf{Critical levels.} To prove differentiability of $\mu$ at $\pm\ell_c$ we require some additional inputs:

\begin{claim}\label{cl:mu_cont_diff}
    The following functions are continuous on $\R$:
    \begin{enumerate}
        \item $\ell\mapsto\mu(\ell)$;
        \item $\ell\mapsto P_\infty(\ell;0)$.
    \end{enumerate}
\end{claim}

The first point will follow from elementary considerations, while the second is a consequence of the (pinned) two-arm decay in Proposition \ref{p:ptwoarm} (itself related to the uniqueness of the infinite cluster).

\smallskip
Let us complete the proof assuming this claim. Since $\mu(\ell)=\mu(-\ell)$, it suffices to consider $\ell=\ell_c$. Given $\epsilon>0$ sufficiently small, we know that $\mu$ is continuous on $[\ell_c,\ell_c+\epsilon]$ (Claim~\ref{cl:mu_cont_diff}) and differentiable on $(\ell_c,\ell_c+\epsilon)$ (by the non-critical case above). Hence by the mean-value theorem, for some $\tilde{\epsilon}\in(0,\epsilon)$
    \begin{displaymath}
        \frac{\mu(\ell_c+\epsilon)-\mu(\ell_c)}{\epsilon}=\mu^\prime(\ell_c+\tilde{\epsilon})=-P_\infty(\ell_c+\tilde{\epsilon};0).
    \end{displaymath}
    Since $\ell\mapsto P_\infty(\ell;0)$ is continuous, taking $\epsilon$ to zero shows that the right derivative of $\mu$ at $\ell_c$ is $-P_\infty(\ell_c;0)$. An analogous argument shows that the left derivative is the same, and hence $\mu^\prime(\ell_c)=-P_\infty(\ell_c;0)$, which is continuous (Claim \ref{cl:mu_cont_diff}) as required.

\begin{proof}[Proof of Claim~\ref{cl:mu_cont_diff}]
    For the first item, given $R>1$ we choose some ordering $\prec$ of the points in $\Lambda_R$. For $\ell\in\R$ and $\epsilon>0$, we will decompose $N_R(f-\ell-\epsilon)-N_R(f-\ell)$ by sequentially considering the change at each point in turn. Specifically for $x\in\Lambda_R$ we define
    \begin{displaymath}
        E_x:=\{f>\ell+\epsilon\}\cup\{y\succ x : f(y)>\ell\}\quad\text{and}\quad E^\prime_x:=\{f>\ell+\epsilon\}\cup\{y\succeq x : f(y)>\ell\}.
    \end{displaymath}
    Then letting $\Delta_x=\Xi_R(E_x^\prime)-\Xi_R(E_x)$ we have
    \begin{displaymath}
        N_R(f-\ell-\epsilon)-N_R(f-\ell)=\sum_{x\in\Lambda_R}\Delta_x.
    \end{displaymath}
    Since $\lvert\Delta_x\rvert\leq 2d\ind_{f(x)\in[\ell,\ell+\epsilon]}$ we have
    \begin{align*}
        \lvert\mu(\ell+\epsilon)-\mu(\ell)\rvert\leq\lim_{R\to\infty}(2R)^{-d}\sum_{x\in\Lambda_R}\E[\lvert\Delta_x\rvert]\leq 2d\;\P(f(0)\in[\ell,\ell+\epsilon])
    \end{align*}
    which yields continuity of $\mu$.

    For the second item we use arguments similar to those which appear in \cite{akn87}, which studied related questions for independent percolation models. Recall that for $E\subseteq\Z^d$, $\Xi^+_R(E)$ and $\Xi^-_R(E)$ denote the number of clusters in $E$ and $E^c$ respectively that are contained in $\Lambda_R\setminus\partial\Lambda_R$. We let $P^\pm_R$ and $P^\pm_\infty$ denote the pivotal intensities for these functionals. Since $\Xi_R=\Xi_R^++\Xi_R^-$, we have $P_\infty=P_\infty^++P_\infty^-$ and so it suffices to prove continuity of $P_\infty^+$ and $P_\infty^-$.

    We first argue that $P_R^+(\ell;0)$ is non-increasing in $R$ (for fixed $\ell$). Given $E\subseteq\Z^d$ and $R>1$, we let $i_R$ and $b_R$ be the number of clusters of $E\cap\Lambda_R\setminus\{0\}$ which contain a neighbour of $0$ and do or do not intersect $\partial\Lambda_R$ respectively. We call these `$R$-interior clusters' and `$R$-boundary clusters'. Then by definition of $\Xi_R^+$
    \begin{equation}\label{e:diff_crit}
        d_0\Xi^+_R(E)=\begin{cases}
            -i_R+1 &\text{if }b_R=0,\\
            -i_R &\text{if }b_R\geq 1.
        \end{cases}
    \end{equation}
    For any $S>R$, the $R$-interior clusters must also be $S$-interior clusters and so $i_S\geq i_R$. Therefore if $b_R=0$ or $b_S\geq 1$, \eqref{e:diff_crit} implies that $d_0\Xi^+_S(E)\leq d_0\Xi_R^+(E)$. On the other hand, if $b_R\geq 1$ and $b_S=0$ then there must be an $R$-boundary cluster which is contained in an $S$-interior cluster and so $i_S\geq i_R+1$. Once again by \eqref{e:diff_crit}, we have $d_0\Xi^+_S(E)\leq d_0\Xi_R^+(E)$. Since $E$ is arbitrary, we conclude that $P_R^+(\ell,0)=\E[d_0\Xi_R^+(f-\ell)|f(0)=\ell]\varphi_{f(0)}(\ell)$ is non-increasing in $R$, as required. For a fixed $R>1$, $P_R^+(\ell;0)$ is continuous in $\ell$ (this follows from applying a simple dominated convergence argument along with Gaussian regression) and so $P_\infty^+(\ell;0)=\lim_{R\to\infty}P_R^+(\ell;0)$, as a decreasing pointwise limit of continuous functions, must be upper semi-continuous

    Given $E\subseteq\Z^d$, let $\tilde{\Xi}_R^+(E)$ denote the number of clusters of $E\cap\Lambda_R$. This can be thought of as the cluster count on $\Lambda_R$ when assuming `free' boundary conditions whereas $\Xi_R^+$ is the cluster count (minus one) assuming `wired' boundary conditions. We can define the pivotal intensities $\tilde{P}_R$ and $\tilde{P}_\infty$ for this functional analogously to those for $\Xi^+$. The argument of the previous paragraph can be adapted to show that $\tilde{P}_R^+(\ell;0)$ is non-decreasing in $R$; in fact the argument is somewhat easier in this case as $d_0\tilde{\Xi}_R^+(E)$ is one minus the number of clusters of $E\cap\Lambda_R\setminus\{0\}$ which contain a neighbourhood of $0$ and the latter quantity is clearly non-increasing in $R$. Hence $\tilde{P}_\infty^+(\ell;0)$ is a non-decreasing limit of continuous functions and therefore lower semi-continuous. To complete the proof of the claim, we need only show that $P_\infty^+(\ell;0)=\tilde{P}_\infty^+(\ell;0)$ for every $\ell$.
    
    Using convergence of the pivotal intensities to their stationary counterparts and the reverse Fatou lemma
    \begin{align*}
        \lvert P_\infty^+(\ell;0)-\tilde{P}_\infty^+(\ell;0)\rvert\leq\E\Big[\limsup_{R\to\infty}\big\lvert d_0\Xi_R^+(f-\ell)- d_0\tilde{\Xi}_R^+(f-\ell)\big\rvert\;\Big|\;f(0)=\ell\Big]\varphi_{f(0)}(\ell).
    \end{align*}
    We next claim that, conditional on $f(0)=\ell$ there is at most one infinite cluster of $\{f>\ell\}$ which contains a neighbour of $0$. If this were not the case, then for any $R\geq 1$ one could find two disjoint `arms' joining neighbours of $0$ to $\partial \Lambda_R$. It follows from the pinned two-arm decay given in Proposition~\ref{p:ptwoarm} that this event has probability $0$, verifying the claim. Fixing such a realisation, we let $C_\infty\subseteq\Z^d$ denote this cluster (which may be empty) and $C_1,\dots,C_i$, for some $i\leq 2d$, denote the other clusters of $\{f>\ell\}$ which contain a neighbour of $0$. If $R$ is sufficiently large so that $C_1,\dots,C_i\subseteq\Lambda_{R-1}$, then by definition of $\Xi^+$ and $\tilde{\Xi}^+$
    \begin{displaymath}
        d_0\Xi_R^+(f-\ell)=
        \begin{rcases}
        \begin{dcases}
            -i+1 &\text{if }C_\infty=\emptyset\\
            -i &\text{if }C_\infty\neq\emptyset
        \end{dcases}
        \end{rcases}
        =d_0\tilde{\Xi}_R^+(f-\ell).
    \end{displaymath}
    Hence the right-hand side of the previous displayed equation is zero, which completes the proof of continuity for $P_\infty^+$. The proof for $P_\infty^-$ is near-identical, up to changes of sign.
\end{proof}

\subsection{Proof of Lemma \ref{l:high}}
 Since $\mu(\ell)=\mu(-\ell)$ we assume that $\ell>0$. By Proposition~\ref{p:mudiff}
    \begin{displaymath}
        \mu^\prime(\ell)=-P_\infty(\ell;0)=-\lim_{R\to\infty}\E[d_0\Xi_R(f-\ell)|f(0)=\ell]\varphi_{f(0)}(\ell).
    \end{displaymath}
    If $f(y)<\ell$ for all $y\in \Lambda_1\setminus\{0\}$, then $d_0\Xi_R(f-\ell)=1$. Then since $\lvert d_0\Xi_R\rvert\leq 2d$, by the union bound we have
    \begin{displaymath}
        \mu^\prime(\ell)\leq\Big(-\P \big(\cap_{y\in \Lambda_1\setminus\{0\}}\{f(y)<\ell\} \big| f(0)=\ell \big)+2d\sum_{y\in \Lambda_1\setminus\{0\}}\P \big(f(y)>\ell \big|f(0)=\ell \big) \Big)\varphi_{f(0)}(\ell).
    \end{displaymath}
    By Gaussian regression $(f(y)|f(0)=\ell)$ is normally distributed with mean $(G(y)/G(0)) \ell$ and variance that depends only on $y$. Since $G(y) < G(0)$ for every $y \neq 0$ (a general property of stationary ergodic Gaussian fields on $\mathbb{Z}^d$), we have $\P(f(y)<\ell|f(0)=\ell)\to 1$ as $\ell\to\infty$. Hence $\limsup_{\ell\to\infty}\mu^\prime(\ell)/\varphi_{f(0)}(\ell)\leq -1$ which proves the lemma.

\medskip

\section{General bounds on the variance}
\label{s:vb}

In this section we prove general bounds on the variance which hold at all levels (Propositions \ref{p:ex} and \ref{p:varub}). The arguments are similar to (but much simpler than) those appearing in \cite{bmm22} and \cite{bmm24} in the setting of smooth Gaussian fields.

\begin{proposition}[Extensivity of the variance]
\label{p:ex}
For every $\ell \in \R$ there exists $c > 0$ such that, for every $R \ge 1$
\[ \Var[N_R(\ell)] \ge c R^d .\]
\end{proposition}

\begin{proof}
Let $c_d$ be such that $\lvert\partial\Lambda_r\rvert\leq c_dr^{d-1}$ for all $r\geq 1$. Fix $r$ large enough so that
\begin{equation}\label{e:varext1}
    \E[N_r(\ell)]\geq\frac{\mu(\ell)}{2}r^d\geq 2c_dr^{d-1}+3
\end{equation}
which is possible by the law of large numbers for the cluster count and the fact that $\mu(\ell)>0$ (Proposition~\ref{p:mu}). For $R > r$, let $(x_i)_{1\le i\le n}$ be $n>c_rR^d$ points in $\Lambda_R$ that have mutual $d_\infty$ distance at least $3r$ and let $\Lambda^i=x_i+\Lambda_{r-1}$. Recalling \eqref{e:iiddecomp}, one can decompose $f$ as 
\[ f \stackrel{d}{=} f^\prime+\sum_{i \le n} \kappa \tilde{Z}_i ,  \]
where $\kappa>0$, $\tilde{Z}_i$ are i.i.d.\ standard Gaussian vectors supported on $\Lambda^i$, and $f'$ is an independent Gaussian field. Denote $M_i:=\E[N_R(\ell)|\tilde{Z}_1,\dots,\tilde{Z}_i]$. Then using successively the law of total variance, orthogonality of martingale increments, the conditional Jensen inequality, and the tower property
\begin{align*}
    \mathrm{Var}[N_R(\ell)] \ge \mathrm{Var}[M_n] & = \sum_{i=1}^n\E\left[(M_i-M_{i-1})^2\right] \\
    & \geq\sum_{i=1}^n\E\left[(\E[M_i-M_{i-1}|\tilde{Z}_i])^2\right]\\
    &=\sum_{i=1}^n\E\left[(\E[M_n| \tilde{Z}_i]-\E[M_n])^2\right]=\sum_{i=1}^n\Var[\E[N_R(\ell)|\tilde{Z}_i]].
\end{align*}
(This is an instance of the reverse Efron-Stein inequality.) Hence it suffices to show that
\begin{equation}
    \label{e:varext2}
 \Var \big[  \E[ N_R(\ell) | \tilde{Z}_i ] \big] \ge \delta \quad \text{for all } 1 \le i \le n
\end{equation}
where $\delta>0$ is independent of $R$.

Let $\Lambda^i_+=x_i+\Lambda_{r}$, then by definition of the cluster count,
\begin{equation}\label{e:varext3}
N_R(\ell)=\Xi_{\Lambda_R\setminus\Lambda^i}(f-\ell)+\Xi_{\Lambda^i_+}(f-\ell)+\Xi_{\Lambda_R,\partial\Lambda^i_+}(f-\ell)
\end{equation}
where we recall that $\Xi_D$ is the cluster count in $D$, and $\Xi_{\Lambda_R,\partial\Lambda^i_+}$ is the number of clusters contained in $\Lambda_R$ which intersect $\partial\Lambda^i_+$. Since $\Xi_{\Lambda_R\setminus\Lambda^i}$ is independent of $\tilde{Z}_i$, we have $\E[\Xi_{\Lambda_R\setminus\Lambda^i}|\tilde{Z}_i]=\E[\Xi_{\Lambda_R\setminus\Lambda^i}]$. Then using the fact that $\lvert\Xi_{\Lambda_R,\partial\Lambda^i_+}\rvert\leq \lvert\partial\Lambda_+^i\rvert\leq c_dr^{d-1}$, taking conditional and unconditional expectations of \eqref{e:varext3} yields
\begin{displaymath}
    \E[N_R(\ell)| \tilde{Z}_i]-\E[N_R(\ell)]=\E[\Xi_{\Lambda^i_+}| \tilde{Z}_i]-\E[\Xi_{\Lambda^i_+}]+e_r
\end{displaymath}
where $\lvert e_r\rvert\leq 2c_dr^{d-1}$. Now given $s>0$ we define the event $A_s = \{ \|f^\prime|_{\Lambda^i}\|_\infty\leq s\}$ and the event $B_s$ that every entry of $\tilde{Z}_i$ exceeds $(\ell+ s)/\kappa$. On $A_s\cap B_s$, $f|_{\Lambda^i} > \ell$  and so $\Xi_{\Lambda^i_+} \in\{0,1\}$. Moreover by stationarity and \eqref{e:varext1}, $\E[\Xi_{\Lambda^i_+}(f-\ell)] = \E[ N_r(\ell) ] \ge \mu(\ell) r^d / 2$.
Therefore on the event $B_s$
\begin{displaymath}
    \E[N_R(\ell)| \tilde{Z}_i]-\E[N_R(\ell)]\leq 1+(2r+1)^d\P(A_s^c)-\frac{\mu(\ell)}{2}r^d+2c_dr^{d-1}\leq -1
\end{displaymath}
where the final inequality is guaranteed by choosing $s$ sufficiently large (since $\P(A_s^c)\to0$ as $s\to\infty$) and using \eqref{e:varext1} again. Since $\P(B_s)$ is positive and independent of $R$, we have verified~\eqref{e:varext2}.
\end{proof}

\begin{proposition}[General upper bound]
\label{p:varub}
For every $\ell \in \R$,
\[ \limsup_{R\to\infty} \frac{\Var[N_R(\ell)]}{R^{d+2}} \le \frac{d^2 \beta_{d,1}}{ G(0)} \]
where $\beta_{d,k} > 0$ is defined in \eqref{e:beta}. 
\end{proposition}

\begin{proof}
Applying Proposition \ref{p:wcerror}, we have
\[ \Var[N_R(\ell)] = \sum_{x,y \in \Lambda_R}  G(x-y) \int_0^1 
 P^t_R(x;y) \, dt \]
where $P^t_R(x;y)$ denotes the $1+1$ joint pivotal intensity in Definition \ref{d:jpi} applied to $N_R(\ell)$. Since $|d_y N_R(\ell)| \le 2d$ we have
\[ |P^t_R(x;y)| \le (2d)^2 \varphi_{f(x),f^t(y)}(\ell,\ell) \le \frac{(2d)^2}{2\pi G(0)}  
\frac{1}{\sqrt{ 1 - t^2 \mathrm{Corr}(f(x),f(y))^2} }  \le  \frac{2d^2}{\pi G(0)}  
\frac{1}{\sqrt{ 1 - t^2 } } . \]
Since $\int_0^1 ( 1 - t^2)^{-1/2}  \, dt = \pi /2 $, we have 
 \[ \Var[N_R(\ell)] \le \frac{d^2}{G(0)} \sum_{x,y \in \Lambda_R}  G(x-y)  , \]
 and we conclude by combining with \eqref{e:beta}.
\end{proof}


\medskip

\section{Proof of the main results}
\label{s:mr}

In this section we complete the proof of Theorems \ref{t:fluc1}, \ref{t:fluc2}, \ref{t:var} and \ref{t:var2}, which describe the limiting distribution and variance of the cluster count $N_R(\ell)$. Recalling the chaos expansion $N_R(\ell) = \E[N_R(\ell)] + \sum_{m \ge 1} Q_m[N_R(\ell)]$, our strategy is to study the limit of individual chaos projections via their stationary approximations. This can be achieved by generalising the classical methods for studying local functionals (i.e. diagram formulae and the fourth moment theorem) as we show in Appendices~B and~C. To characterise the overall limit, we identify which chaos projections (if any) are dominant based on the order of their variance and control the tail of the chaos projections using truncation (Proposition~\ref{p:statwcetrunc}). The variances of the dominant chaos projections are expressed in terms of the stationary pivotal intensities, or equivalently (Proposition~\ref{p:mudiff}) in terms of the derivatives of the mean functional $\mu$, as in the statement of the theorems.

We abbreviate
\[\widetilde{N}_R(\ell) = \frac{N_R(\ell) - \E[N_R(\ell)] }{ \sqrt{\Var[N_R(\ell)]} } \quad \text{and} \quad  \widetilde{Q}_m =  \frac{ Q_m[N_R(\ell)] }{\sqrt{\Var[Q_m[N_R(\ell)] ] }}.\]
We say that $\widetilde{N}_R(\ell)$ is \textit{asymptotically dominated by $\widetilde{Q}_m$} if $\Var[Q_m[N_R(\ell)] ] \sim \Var[N_R(\ell)]$ as $R \to \infty$. This implies, in particular, that $\widetilde{N}_R(\ell)  - \widetilde{Q}_m \to 0 $ in probability. 
\smallskip Recall that $Z$ denotes a standard Gaussian variable, $\Rightarrow$ denotes convergence in law and $\mathcal{C}^{(i)}$ denotes the levels at which the first non-vanishing derivative of $\mu$ is of order $i$.

\begin{proof}[Proof of Theorems \ref{t:fluc1}, \ref{t:fluc2} and \ref{t:var}]
We divide the analysis into four cases: 
\begin{enumerate}
    \item $\ell\in\mathcal{C}^{(1)}$; 
    \item either (i) $d = 4$ and $\ell\in\mathcal{C}^{(2)}$, or (ii) $d = 3$ and $\ell\in\mathcal{C}^{(3)}$; 
    \item $d = 3$ and $\ell\in\mathcal{C}^{(2)}$; 
    \item all remaining cases.
\end{enumerate}
In the first, second, and third cases we will show that $\widetilde{N}_R(\ell)$ is asymptotically dominated by, respectively, $\widetilde{Q}_1$, $\widetilde{Q}_{6-d}$, and $\widetilde{Q}_2$. In the fourth case we show that all terms $\widetilde{Q}_{m}$ may contribute non-negligibly to $\widetilde{N}_R(\ell)$. 
    
For $m \in \mathbb{N}$, let $P^m_\infty$ denote the function $P_\infty$ in Proposition \ref{p:statwce} (given in Definition \ref{d:hspivint}) for this choice of $m$, and abbreviate $\sum P^m_\infty := \sum_{x_2,\ldots,x_m \in \Z^d}P^m_\infty(0,x_2,\ldots,x_m)$. Since we assume $\ell \neq \{-\ell_c,\ell_c\}$, we shall use without further mention that Proposition \ref{p:mudiff} identifies $\sum P^m_\infty$ as $(-1)^m \mu^{(m)}(\ell)$.

    Recall that $m_0 = \max\{2,7-d\}$ is the smallest positive integer such that $m_0 (d-2) > d$.  
    
\noindent \textbf{Case (1).} By stationarity of $P_\infty$ and the definition of $\beta_{d,k}$ in \eqref{e:beta},
\[
    \Var\Big[\sum_{x\in\Lambda_R}P^1_\infty(x)f(x)\Big]=P_\infty^1(0)^2\sum_{x,y\in\Lambda_R}G(x-y)\sim  \beta_{d,1} P^1_\infty(0)^2  R^{d+2}.
\]
Hence by Proposition \ref{p:statwce} and the Cauchy-Schwarz inequality, as $R \to \infty$
\[
    \Var[Q_1[N_R(\ell)] ]\sim  \beta_{d,1} P^1_\infty(0)^2  R^{d+2} + O(R^{d+1}).
\]
Since $P^1_\infty(0) =-\mu'(\ell) \neq 0$ by assumption, we have
\[ \frac{\Var[Q_1[N_R(\ell)] ]}{  R^{d+2} } \to  \beta_{d,1} (\mu'(\ell))^2 . \]

We next argue that the chaoses of order $m \neq 1$ have smaller order. For $2 \le m < m_0$, combining Proposition \ref{p:statwce} with either Proposition \ref{p:nclt} (if $m (d-2) < d$ and $\sum P_\infty^m \neq 0$),  Proposition \ref{p:clt2} (if $m (d-2) = d$ and $\sum P_\infty^m \neq 0$), or Proposition \ref{p:clt} (if $m (d-2) < d$ and $\sum P_\infty^m = 0$), we have
\begin{equation}
    \label{e:case1e1}
\Var[ Q_m[N_R(\ell) ] ] = o(R^{d+2}) .
\end{equation}
Note that our application of Proposition \ref{p:clt} used that $m(d-2) > d-2$ for $m \ge 2$, and that the Green's function $G$ satisfies \eqref{e:kreg} up to a normalising constant.

To analyse the chaoses of order $m \ge m_0$, fix $\eps > 0$ and recall that
\begin{equation}
    \label{e:barq}
\overline{Q}_m(R,r) := \frac{1}{m!}\sum_{x_1,\dots,x_m\in \Lambda_R} \wick{f(x_1)\cdots f(x_m)} P_{\infty;\le r}(x_1,\ldots,x_m)   
\end{equation} 
where $P_{\infty;\le r}$ is as in Proposition \ref{p:statwcetrunc}. Choosing $r,M>0$ depending on $\epsilon$ as in Proposition~\ref{p:statwcetrunc}, as $R \to \infty$, eventually
\begin{equation}
   \label{e:case1e2}  \Var \Big[ \sum_{m \ge m_0} Q_m[N_R(\ell)]  - \sum_{m_0\leq m\leq M} \overline{Q}_m(R,r) \Big]   \le  \eps R^d.
   \end{equation}
By Proposition~\ref{p:clt3} (recall also that $\|P_{\infty;\le r}\|_\infty \le \sqrt{m!} e^{c_{d,\ell,r} m}$ by Lemma \ref{l:bounds})
\begin{equation}
\label{e:case1e3}
\Var \Big[ \sum_{m_0\leq m \leq M} \overline{Q}_m(R,r) \Big] \le   \sum_{m_0\leq m\leq M} e^{cm}  R^d 
\end{equation}
where $c > 0$ depends only on $d$, $\ell$ and $\eps$. Combining with \eqref{e:case1e1} this shows that, as $R \to \infty$
\[\sum_{ m \ge 2 } \frac{ \Var[ Q_m[N_R(\ell) ]] }{R^{d+2}} \to 0 .\]
Hence $\widetilde{N}_R(\ell)$ is asymptotically dominated by $\widetilde{Q}_1$, which completes the proof since $Q_1[N_R(\ell)]$ is Gaussian by definition.


\noindent \textbf{Case (2).} Recall that $d \in \{3,4\}$, and let $m' = m_0 - 1 = 6-d \in \{2,3\}$. Since $m'(d-2) = d$ and $\sum P^{m'}_\infty = (-1)^{m'} \mu^{(m')} \neq 0$ by assumption, combining Propositions \ref{p:statwce} and \ref{p:clt2} gives that, as $R \to \infty$
\[ \frac{\Var[Q_{m'}[N_R(\ell)] ]}{  R^d (\log R) } \to  \frac{ \beta_{d,m'}  (\sum P^m_\infty)^2}{(m')!}  \qquad \text{and} \qquad \widetilde{Q}_{m'} \Longrightarrow Z .\]
We next argue that the chaoses of order $m \neq m'$ are negligible. As before we have
\[ \Var[Q_1[N_R(\ell)] ] = \beta_{d,1} P^1_\infty(0)^2  R^{d+2} + O(R^d) .\]
Since $P^1_\infty(0) = -\mu'(\ell) = 0$ by assumption, 
\[   \frac{ \Var[ Q_1[N_R(\ell) ] ] }{R^d (\log R)} \to 0. \]
If $d=3$, so that $m' > 2$, we have to analyse the second chaos separately. In that case, since $\sum P^2_\infty = \mu''(\ell) = 0$, combining Propositions \ref{p:statwce} and \ref{p:clt}, as $R \to \infty$
\[   \frac{ \Var[ Q_2[N_R(\ell) ] ] \ }{R^d (\log R)} \to 0.  \]
Finally, we bound the chaoses of order $m \ge m_0$ in the same way as in \eqref{e:case1e2}--\eqref{e:case1e3}. Together, this shows that $\widetilde{N}_R(\ell)$ is asymptotically dominated by $\widetilde{Q}_{m'}$, concluding the proof.

\noindent \textbf{Case (3).} This is similar to the previous case. Since $\sum P^2_\infty = \mu''(\ell)$ by assumption, combining Propositions \ref{p:statwce} and \ref{p:nclt} gives that, as $R \to \infty$
\[ \frac{\Var[Q_{2}[N_R(\ell)] ]}{  R^4  } \to  \frac{ \beta_{3,2}  (\sum P^2_\infty)^2}{2} \qquad \text{and} \qquad \widetilde{Q}_2 \Rightarrow Z', \]
where $Z'$ has order-$2$ Hermite distribution associated to the measure with density $\rho(\lambda) = |\lambda|^{-2}$. On the other hand, as in the previous case we have, as $R \to \infty$
\[  \sum_{m \neq 2} \frac{ \Var[ Q_m[N_R(\ell) ] ] \ }{R^4} \to 0 ,\]
and so $\widetilde{N}_R(\ell)$ is asymptotically dominated by $\widetilde{Q}_{2}$.

\noindent \textbf{Case (4).} By Proposition \ref{p:varasymp}, as $R \to \infty$
\[ \Var \Big[ Q_1[N_R(\ell) ] - \sum_{x \in \Lambda_R} f(x) P_\infty^H( d_{\infty}(x,\partial \Lambda_R)  ) \Big] = o(R^d)  \]
where, since $P^1_\infty(0) = -\mu'(\ell) = 0$ by assumption, $|P_\infty^H(k)| \le c_1 e^{-c_2 k^\rho}$. Applying Lemma \ref{l:rv2}, as $R \to \infty$
\[ \Var[ Q_1[N_R(\ell)] ] \sim \sigma_1^2 R^d , \]
where
 \[ \sigma_1^2  =  c_d \overline{E}_{d,d-2} \Big(\sum_{k \ge 0} P_\infty^H(k) \Big)^2 \in [0,\infty)  ,  \]
 and where $c_d > 0$ is such that $G(x) \sim c_d |x|^{2-d}$, and $\overline{E}_{d,\alpha} \in (0,\infty)$ is defined in \eqref{e:edot}. 
 
 For $2 \le m < m_0$, note that $\sum P_\infty^m = (-1)^m \mu^{(m)} = 0$ by assumption. Then combining Propositions \ref{p:statwce} and \ref{p:clt}, as $R \to \infty$
\[ \Var[ Q_m[N_R(\ell) ] ]  \sim \sigma_m^2 R^d \quad \text{and} \quad Q_m[N_R(\ell) ] / R^{d/2} \Longrightarrow \sigma_m Z \]
for constants $\sigma_m^2 \in [0,\infty)$. 

We now consider the higher orders $m\geq m_0$. Given $\epsilon>0$, we choose $r,M>0$ as in Proposition~\ref{p:statwcetrunc}. Recalling the definition of $\overline{Q}_m(R,r)$ in \eqref{e:barq}, by Proposition \ref{p:clt} for every $m \ge m_0$, as $R \to \infty$
\begin{equation}
   \label{e:clt2}
\frac{ \Var[ \overline{Q}_m(R,r) ] }{ R^d } \to \sigma_{m,r}^2  \qquad \text{and} \qquad  \frac{  \overline{Q}_m(R,r)  }{ R^{d/2}}   \Longrightarrow  \sigma_{m,r} Z 
\end{equation}
for some $\sigma_{m,r}^2 \geq 0$. Given our choices of $M=M_\epsilon$ and $r=r_\epsilon$, we define
\begin{displaymath}
    N_R^{(\epsilon)}(\ell)=\sum_{1\leq m < m_0}  Q_m[N_R(\ell) ]  + \sum_{m_0 \le m \le M}  \overline{Q}_m(R,r)\qquad\text{and}\qquad\widetilde{N}_R^{(\epsilon)}(\ell)=\frac{N_R^{(\epsilon)}(\ell)}{\sqrt{\Var[N_R^{(\epsilon)}(\ell)]}}.
\end{displaymath}
Then since component-wise normal convergence is equivalent to joint convergence for sequences of finite vectors of elements of fixed chaoses (\cite[Theorem 6.2.3]{np12}), we deduce that as $R \to \infty$
\[ \frac{ \Var \big[ N_R^{(\epsilon)}(\ell) \big] } { R^d } \to   \sum_{1\leq m <m_0}\sigma_m^2+\sum_{m_0\leq m\leq M}\sigma_{m,r}^2 =: \sigma^2_\epsilon \qquad\text{and}\qquad
\frac{ N_R^{(\epsilon)}(\ell) }{R^{d/2}} \Longrightarrow \sigma_{\epsilon} Z .  \]
Using the fact that different order chaoses are orthogonal and Proposition~\ref{p:statwcetrunc}
\begin{equation*}
    \limsup_{R\to\infty}\frac{1}{R^d}\left\lvert\Var[N_R(\ell)]-\Var[N_R^{(\epsilon)}(\ell)]\right\rvert\leq \epsilon.
\end{equation*}
Since $\epsilon>0$ was arbitrary, combining the last two equations shows that $\Var[N_R(\ell)]/R^d$ is Cauchy and hence convergent to a limit which we denote by $\sigma^2$. Moreover it follows that $\lim_{\epsilon\to 0}\sigma_\epsilon=\sigma$. Convergence in distribution follows from an elementary argument: for any $x\in\R$ and $\delta_0>0$ by the triangle inequality
\begin{align*}
    \Big\lvert \P\big(\widetilde{N}_R(\ell)\leq x\big)-\Phi(x)\Big\rvert\leq& \sup_{\lvert\delta\rvert\leq\delta_0}\Big\lvert\P\big(\widetilde{N}_R^{(\epsilon)}(\ell)\leq x+\delta\big)-\Phi(x+\delta)\Big\rvert +\left\lvert\Phi(x)-\Phi(x+\delta)\right\rvert\\
    &+\P\left(\left\lvert\widetilde{N}_R(\ell)-\widetilde{N}_R^{(\epsilon)}(\ell)\right\rvert>\delta_0\right)
\end{align*}
where $\Phi$ denotes the standard normal CDF. Choosing $\delta_0>0$ and then $\epsilon>0$ sufficiently small and applying Chebyshev's inequality to the final term shows that this expression can be made arbitrarily small by taking $R\to\infty$. Hence $\widetilde{N}_R(\ell)\Longrightarrow Z$ as required. To conclude we observe that $\sigma > 0$ by the extensivity of the variance (Proposition \ref{p:ex}).
\end{proof}

\begin{remark}
\label{r:be}
It follows from the above proof that
\begin{equation}
\label{e:sigma}
\sigma^2 =  c_d \overline{E}_{d,d-2} \Big(\sum_{k \ge 0} P_\infty^H(k) \Big)^2 + \sum_{m \ge 2} \sigma^2_m \in [0,\infty)
\end{equation}
where $\sigma^2_m:=\lim_{r\to\infty}\sigma^2_{m,r}$. This expression includes a (possibly) non-negligible boundary effect since the value of $\sum_{k \ge 0} P_\infty^H(k) $ may depend on the choice of boundary conditions (see Definition \ref{d:hspivint}).
\end{remark}

\begin{proof}[Proof of Theorem \ref{t:var2}]
The first statement is a combination of Propositions \ref{p:ex} and \ref{p:varub}. To prove the second statement, observe that by Propositions \ref{p:statwcequal} and \ref{p:mudiff}, as $R \to \infty$,
\begin{equation*} \Var[ N_R(\ell_c)  ]  \ge \Var[Q_1[N_R(\ell_c)]] \sim  \beta_{d,1}   (\mu'(\ell_c))^2 R^{d+2} . \qedhere
\end{equation*}
\end{proof}

\medskip
\appendix

\section{Gaussian vectors and multivariate Hermite polynomials}
\label{a:mhp} 

In this appendix we establish basic properties of Gaussian vectors and multivariate Hermite polynomials that were used in Sections \ref{s:ce} and \ref{s:sl}. 

\subsection{Gaussian vectors}
For a non-degenerate Gaussian vector $X$, recall that $\Sigma_X$ and $\varphi_X$ denote its covariance matrix and density respectively.  Let $\lambda_{\mathrm{min}}(X)$ and $\lambda_{\mathrm{max}}(X)$ denote the smallest and largest eigenvalue of $\Sigma_X$ respectively. For a set of indices $I \subset \N$, let $X_I = (X_i)_{i \in I}$. For $t \in [0,1]$, let $X^t = t X + \sqrt{1-t^2} \tilde{X}$, where $\tilde{X}$ is an independent copy of $X$.

\begin{lemma}
\label{l:lambdamin}
    For a non-degenerate Gaussian vector $(X,Y)$,
    \begin{equation}
    \label{e:lambda1} 
    \lambda_{\mathrm{min}}(X,Y) \le \lambda_{\mathrm{min}}(X|Y) \le \lambda_{\mathrm{min}}(X) \le  \|\Sigma_X^{-1}\|^{-1}_\infty \ , \quad \lambda_{\mathrm{max}}(X) \le \mathrm{dim}(X) \| \Sigma_X\|_\infty ,
    \end{equation}
    and for every $t \in [0,1]$,
    \begin{equation}
        \label{e:lambda2}
   \lambda_{\mathrm{min}}(X,Y) \le  \lambda_{\mathrm{min}} \big(X, Y^t \big) . 
    \end{equation} 
\end{lemma}
\begin{proof}
   By Gaussian regression, $Y$ is independent of $X-\Sigma_{X \to Y}\Sigma_Y^{-1}Y$, which has the same covariance as $X|Y$, and so $ \lambda_\mathrm{min}(X)$ is equal to
   \[ \min_{\lvert v\rvert=1}\Var[v^T X] =\min_{\lvert v\rvert=1}\Var[v^T(X-\Sigma_{X \to Y}\Sigma_Y^{-1}Y)]+\Var[v^T \Sigma_{X \to Y}\Sigma_Y^{-1}Y] \geq\lambda_\mathrm{min}(X|Y). \]
Using similar reasoning
    \begin{align}
     \nonumber  \lambda_\mathrm{min}(X,Y)&=\min_{\lvert u\rvert^2+\lvert v\rvert^2=1}\Var[u^TX+v^TY]\\
       \label{e:lambda} &=\min_{\lvert u\rvert^2+\lvert v\rvert^2=1}\Var[u^T(X-\Sigma_{X \to Y}\Sigma_Y^{-1}Y)]+\Var[(u^T\Sigma_{X \to Y}\Sigma_Y^{-1}+v^T)Y].
    \end{align}
    Let $v_\mathrm{min}$ be a unit eigenvector associated with $\lambda_\mathrm{min}(X|Y)$ and define $w=v_\mathrm{min}^T\Sigma_{X \to Y}\Sigma_Y^{-1}$. Choosing
    \begin{displaymath}
        u=\frac{v_\mathrm{min}}{\sqrt{1+\lvert w\rvert^2}},\qquad v=\frac{-w}{\sqrt{1+\lvert w\rvert^2}},
    \end{displaymath}
    which satisfies $\lvert u\rvert^2+\lvert v\rvert^2=1$ and $u^T\Sigma_{X \to Y}\Sigma_Y^{-1}+v^T=0$, and inserting these into \eqref{e:lambda} we see that $ \lambda_\mathrm{min}(X,Y) \le \lvert u\rvert^2\lambda_\mathrm{min}(X|Y)\leq\lambda_\mathrm{min}(X|Y)$. Finally, if $Z$ denotes a Gaussian vector with variance matrix $\Sigma_Z = \Sigma_X^{-1}$, then
  \begin{equation*}
  \lambda_{\mathrm{min}}(X)^{-1} = \lambda_{\mathrm{max}}(Z) =  \max_{\lvert v\rvert=1}\Var[v^T Z] \ge  \| \Sigma_Z \|_{\infty} . 
  \end{equation*}
Along with the classical bound $(\lambda_{\mathrm{max}}(X))^2 \le \sum_{i,j}(\Sigma_X)_{i,j}^2$, this completes the proof of \eqref{e:lambda1}.

To prove \eqref{e:lambda2}, if $(u,v)$ is a unit eigenvector associated with $\lambda_{\mathrm{min}}(X, Y^t)$ then
\begin{align*}
    \lambda_{\mathrm{min}}\big(X, Y^t\big) &= \Var[u^T X + t v^T  Y] +  \Var[ \sqrt{1-t^2} v^T \tilde{Y} ] \\
    & \ge ( \|u\|_2^2 + t^2 \|v\|_2^2)\lambda_{\mathrm{min}}(X,Y) + (1-t^2) \|v\|_2^2   \lambda_{\mathrm{min}}(Y) \\
    & \ge \lambda_{\mathrm{min}}(X,Y)
    \end{align*}
where the final step used that $\lambda_{\mathrm{min}}(Y) \ge  \lambda_{\mathrm{min}}(X,Y)$.
\end{proof}

\begin{lemma}
\label{l:spectral}
For a non-degenerate Gaussian vector $(X,Y)$, let $U$ be orthogonal and $D$ diagonal such that $U^T\Sigma_X U=D$, and let $V$ be the orthogonal matrix
\[    V=\frac{1}{\sqrt{2}}\begin{pmatrix}
    U &-U\\
    U & U
\end{pmatrix} . \] 
Then 
\[ V^T\Sigma_{X,X^t}V=\begin{pmatrix}
    (1+t)D &0\\
    0 & (1-t)D
\end{pmatrix} \ , \qquad  \Sigma_{X,X^t}^{-1} = \frac{1}{1-t^2} \begin{pmatrix}
    \Sigma_X^{-1} & -t \Sigma_X^{-1}\\
    -t \Sigma_X^{-1} & \Sigma_X^{-1}
\end{pmatrix}, \]
and
\[ \Sigma_{Y \to (X,X^t)}\Sigma_{X,X^t}^{-1} = (\Sigma_{Y \to X} U D^{-1} U^{T} ,0 )= (\Sigma_{Y \to X} \Sigma_X^{-1} ,0 ).  \]
\end{lemma}
\begin{proof}
Using that 
\[   \Sigma_{Y \to (X,X^t)} = (  \Sigma_{Y \to X}, t  \Sigma_{Y \to X}) \qquad \text{and} \qquad \Sigma_{X,X^t} = \begin{pmatrix}
    \Sigma_X & t \Sigma_X \\
    t \Sigma_X & \Sigma_X
\end{pmatrix}  \]
the claims follow from straightforward computation.
\end{proof}

\begin{proposition}
\label{p:densitybound}
For a non-degenerate Gaussian vector $X$, indices $I, J \subseteq \{1,\ldots,\mathrm{dim}(X)\}$ and $t \in [0,1)$:
\begin{enumerate}
    \item $\lambda_{\mathrm{min}}(X_I, X_J^t) \ge (1-t) \lambda_{\mathrm{min}}(X)$.
    \item  $ \| \varphi_{X_I,X^t_J}(x,x')\|_\infty  \le c_{|I|,|J|} (1-t)^{-|I\cap J|/2}  \lambda_{\mathrm{min}}(X)^{ - |I|/2- |J|/2} $.
\item For $\ell \in \R$$, \varphi_{X_I,X^t_J}(\ell,\ell)  \ge  c_{|I|,|J|} \|\Sigma_X\|_\infty^{-|I| - |J|} e^{ -2\ell^2 (|I| + |J|)^2  \lambda_{\mathrm{min}}(X)^{-1} }$.
\end{enumerate}
\end{proposition}

\begin{proof} $\,$

$\textbf{(1).}$ By Lemma \ref{l:lambdamin}, $\lambda_{\mathrm{min}}(X_I, X_J^t) \ge \lambda_{\mathrm{min}}(X, X^t)$ and the result follows from Lemma \ref{l:spectral}.

$\textbf{(2).}$
Let $K = I \cap J$, and abbreviate $Z = X_K$, $Z^t = X^t_K$, $Y = X_{I\setminus K}$, and $Y' = X^t_{J \setminus K}$, and similarly $z = x|_K$, $z' = x'|_K$, $y = x|_{I \setminus K}$, and $y' = x'|_{J \setminus K}$. Then
\[  \varphi_{X_I,X^t_J}(x,x') =  \varphi_{Z,Z^t}(z,z')   \varphi_{Y,Y'  | Z,Z^t }(y,y' |z,z').  \]  
We bound these terms separately.

For the first term, observe that, using Lemma \ref{l:spectral} and then \eqref{e:lambda1},
\[ | \Sigma_{Z,Z^t} |  = (1-t^2)^{|K|} |\Sigma_Z|^2 \ge (1-t)^{|K|}  \lambda_{\mathrm{min}}(Z)^{2|K|} \ge (1-t)^{|K|} \lambda_{\mathrm{min}}(X)^{2|K|} , \]
and hence $  \varphi_{Z,Z^t}(z,z') \le | \Sigma_{Z,Z^t} |^{-1/2} \le (1-t)^{-|K|/2} \lambda_{\mathrm{min}}(X)^{-|K|} $.

Similarly for the second term  we have
\begin{equation}
\label{e:lambdamin1}
\lambda_{\mathrm{min}}(Y,Y' | Z,Z^t)  \ge  \lambda_{\mathrm{min}}(X_{(I \cup J) \setminus K} | Z,Z^t)  =  \lambda_{\mathrm{min}}(X_{(I \cup J) \setminus K} | X|_K)  \ge \lambda_{\mathrm{min}}(X)
\end{equation}
where the first inequality is by \eqref{e:lambda2} applied to $X|_{(I \cup J) \setminus K}$ conditional on $(Z,Z^t) = 0$, and the second inequality is by \eqref{e:lambda1}. In particular 
\[  \varphi_{Y,Y'  | Z,Z^t }(y,y' |z,z')  \le |\Sigma_{Y,Y' | Z,Z^t} |^{-1/2} \le \lambda_{\mathrm{min}}(X)^{-|(I \cup J) \setminus K|/2} .  \]
Combining these bounds with the identity $2|K| + |(I \cup J) \setminus K| = |I| + |J|$ gives the statement.

$\textbf{(3).}$ With the notation from the proof of the previous item, we have 
\[  \varphi_{X_I,X^t_J}(\ell,\ell) =  \varphi_{Z,Z^t}(\ell,\ell)   \varphi_{Y,Y'  | Z,Z^t }(\ell,\ell |\ell,\ell) .  \]
Applying Lemma \ref{l:spectral} and then \eqref{e:lambda1} we have
\[ (\ell, \ell)^T \Sigma_{Z,Z^t}^{-1}(\ell,\ell) = \frac{2 \ell^2}{1+t} \sum_{1 \le i,j \le |K|} (\Sigma^{-1}_Z)_{i,j} \le 2 \ell^2 |K|^2 \| \Sigma^{-1}_Z \|_\infty \le  2 \ell^2 |K|^2  \lambda_{\mathrm{min}}(X)^{-1} , \]
and also
\[  |\Sigma_{Z,Z^t}| = (1-t^2)^{|K|} |\Sigma_Z|^2 \le \lambda_{\mathrm{max}}(Z)^{2 |K|} \le  c_{|K|} \|\Sigma_{X} \|_\infty^{2 |K|}  .\]
Similarly we have
\begin{align*}
(\ell, \ell)^T \Sigma_{Y,Y' | Z,Z^t}^{-1} (\ell,\ell) & \le 2 \ell^2 |(I \cup J) \setminus K |^2 \lambda_{\mathrm{min}}(Y,Y' | Z,Z^t)^{-1} \\
& \le 2 \ell^2  |(I \cup J) \setminus K |^2  \lambda_{\mathrm{min}}(X)^{-1} \end{align*}
where the final step used \eqref{e:lambdamin1}, and also
\begin{align*} 
|\Sigma_{Y,Y' | Z,Z^t}| & \le \lambda_{\text{max}}(Y,Y' | Z,Z^t)^{2 |(I \cup J) \setminus K |}  \\
& \le  c_{|I|,|J|}   \|\Sigma_{Y,Y' | Z,Z^t} \|_\infty^{2 |(I \cup J) \setminus K |}   \le c_{|I|,|J|}   \|\Sigma_X \|_\infty^{2 |(I \cup J) \setminus K |}  
\end{align*}
where we used that $\|\Sigma_{Y,Y' | Z,Z^t} \|_\infty \le \| \Sigma_{Y,Y'} \|_\infty \le \|\Sigma_X\|_\infty$ by Gaussian regression. Hence
\[  \varphi_{Z,Z^t}(\ell,\ell) \ge c_{|I|,|J|} \|\Sigma_X\|_\infty^{-|K|} e^{-\ell^2 |K|^2\lambda_{\mathrm{min}}(X)^{-1} } \]
and
\[  \varphi_{Y,Y'  | Z,Z^t }(\ell,\ell |\ell,\ell) ] \ge  c_{|I|,|J|} \|\Sigma_X\|_\infty^{-|(I \cup J) \setminus K |} e^{-\ell^2 |(I \cup J) \setminus K |^2\lambda_{\mathrm{min}}(X)^{-1} }  .   \]
Combining with the inequality $|K| + |(I \cup J) \setminus K| \le |I| + |J|$ gives the result.
\end{proof}

\subsection{Hermite polynomials}
 We next establish  two bounds (Propositions \ref{p:mhpmain1} and \ref{p:mhpmain}) on multivariate Hermite polynomials; see \eqref{e:mhp} for the definition of these polynomials. The first is an elementary pointwise estimate:

\begin{proposition}
\label{p:mhpmain1}
    For a non-degenerate Gaussian vector $X$, multi-index $\alpha\in\N_0^{\mathrm{dim}(X)}$, and $x\in\R^{\mathrm{dim}(X)}$
    \begin{displaymath}
        \lvert H^\alpha_X(x)\rvert\leq \mathrm{dim}(X)^{\lvert\alpha\rvert/2}\sqrt{\lvert\alpha\rvert!}e^{c \sqrt{ \lvert\alpha\rvert} (\| x \|_2 +1)} 
    \end{displaymath}
    where $c>0$ depends only on $\lambda_{\mathrm{min}}(X)$.
\end{proposition}
\begin{proof}
We shall deduce the estimate from a classical bound for univariate Hermite polynomials \cite[Eq.(1.2)]{em90}:
    \begin{equation}
    \label{e:uni}
        \lvert H_n(y)\rvert\leq \sqrt{n!}e^{\sqrt{n}\lvert y\rvert}.
    \end{equation}
    Abbreviate $k = \mathrm{dim}(X)$, let $U$ be a $k\times k$ orthogonal matrix such that $U^T\Sigma_X U$ is diagonal, and let $Y=U^TX$. Let $\partial^\alpha=\partial_{i_1}\dots\partial_{i_p}$ where $p=\lvert\alpha\rvert$. By the definition of the multivariate Hermite polynomials and the chain rule
    \begin{displaymath}
        H_X^\alpha(x)=\frac{\partial^\alpha\varphi_Y(U^Tx)}{\varphi_Y(U^Tx)}=\sum_{j_1=1}^k\dots\sum_{j_p=1}^kU_{i_1,j_1}\dots U_{i_p,j_p}(\partial_{j_1}\dots\partial_{j_p}\varphi_Y)(U^Tx)/\varphi_Y(U^Tx).
    \end{displaymath}
    Since $U$ is orthogonal, the $L^1$ norm of any column is at most $\sqrt{k}$ and so by the triangle inequality
    \begin{equation}\label{e:mhp_eq1}
        \lvert H^\alpha_X(x)\rvert\leq k^{p/2}\sup_{\lvert\beta\rvert=p}\lvert H^\beta_Y(U^Tx)\rvert.
    \end{equation}
    Since the components of $Y$ are independent, for any $\beta\in\N_0^k$
    \begin{displaymath}
        H_Y^\beta(U^Tx)=\prod_{i=1}^kH^{\beta_i}_{Y_i}((U^Tx)_i)=\prod_{i=1}^k\Var[Y_i]^{-\beta_i/2}H_{\beta_i}(\Var[Y_i]^{-1/2}(U^Tx)_i). 
    \end{displaymath}
     Then using \eqref{e:uni} and the fact that $\Var[Y_i]\geq\lambda_{\mathrm{min}}(X)$, we have
    \begin{displaymath}
        \lvert H^\beta_Y(U^Tx)\rvert\leq \lambda_{\mathrm{min}}(X)^{-p/2}\sqrt{p!}e^{\sqrt{p}\lambda_{\mathrm{min}}(X)^{-1/2} \sum_i | (U^Tx)_i | }.
    \end{displaymath}
    Since $\sum_i | (U^Tx)_i | \le \| x \|_2$, combining this with \eqref{e:mhp_eq1} completes the proof.
\end{proof}

The second bound is tailored to our application (see the proofs of Lemmas \ref{l:jpi} and \ref{l:decay}):

\begin{proposition}
\label{p:mhpmain}
For a non-degenerate Gaussian vector $X$, indices $I, J \subseteq \{1,\ldots,\mathrm{dim}(X)\}$, multi-indices $\alpha_I \in \N_0^I$ and $\alpha_J \in \N_0^J$, $t \in [0,1)$, and $\ell \in \R$,
\begin{align}
    \label{e:mhp5}
&  \sqrt{ \E \Big[ \big(H^{\alpha_I,\alpha_J,0,0}_{X_I,X^t_J,X_{I^c},X^t_{J^c}}(\ell,\ell,X_{I^c},X^t_{J^c}) \big)^2 \; \Big| \; X_I =\ell,  X^t_J = \ell \Big] }  \varphi_{X_I,X_J^t}(\ell,\ell)  \\
\nonumber & \qquad \qquad \qquad \qquad  \qquad \qquad \qquad \qquad \qquad \le  c_{X,|I|,|J|,|\bar{\alpha}|}  (1-t)^{-|\bar{\alpha}|/2 - |I\cap J|/2}  
\end{align}
where $\bar{\alpha} = \alpha_I + \alpha_J\in\N_0^{I\cup J}$ and
\[ c_{X,|I|,|J|,|\bar{\alpha}|} := c_{|I|,|J|,|\bar{\alpha}|} \max\big\{1, \|\Sigma_X\|_\infty\big\}^{|\bar{\alpha}| } \max \big\{1, \lambda_{\mathrm{min}}(X)^{-1} \big\}^{2|\bar{\alpha}|+\frac{|I|+|J|}{2}} \]
for a constant $c_{|I|,|J|,|\bar{\alpha}|} > 0$.
\end{proposition}

We build towards the proof of Proposition \ref{p:mhpmain}, beginning with the following:

\begin{proposition}
\label{p:mhp}
For a non-degenerate Gaussian vector $(X_I,Y_J)$, multi-indices $\alpha_I,\alpha'_I \in \N_0^I$ and $\alpha_J,\alpha'_J \in \N_0^J$, and $x \in \R^{|I|}$,
\begin{align}
 \label{e:mhp2} &\E \Big[ H^{\alpha_I,\alpha_J}_{X_I,Y_J}(x,Y_J) H_{X_I,Y_J}^{\alpha_I',\alpha_J'} (x, Y_J) \; \Big| \; X_I = x \Big]   \\
 \nonumber & \qquad = 
 \sum_{ \substack{ \hat{\alpha}_I \le \alpha_I, \hat{\alpha}_I' \le \alpha_I' , \\ |\hat{\alpha}_I + \alpha_J| = |\hat{\alpha}_I' + \alpha_J'| } } \frac{ (\alpha_I)!(\alpha_J)! (\alpha_I')! (\alpha'_J)!   }{(\alpha_I-\hat{\alpha}_I)!(\alpha_I'-\hat{\alpha}_I')!} \sum_{ \theta \in M_{\hat{\alpha}_I,\alpha_J}^{\hat{\alpha}_I',\alpha_J'} } \frac{ (\Sigma_{X_I,Y_J}^{-1})^\theta }{ \theta!}  H_{X_I}^{\bar \alpha}(x) .
\end{align}
 where $\bar \alpha := \alpha_I - \hat{\alpha}_I + \alpha'_I - \hat{\alpha}'_I$, $M_{\hat{\alpha}_I,\alpha_J}^{\hat{\alpha}'_I,\alpha_J'} $ is the set of all $(|I|+|J|) \times (|I|+|J|)$ matrices of non-negative integers whose row sum is $(\hat{\alpha}_I,\alpha_J)$ and whose column sum is $(\hat{\alpha}'_I,\alpha_J')$,  $A^\theta=\prod_{i,j}A_{i,j}^{\theta_{i,j}}$ for a matrix $A$, $\theta! = \prod_{i,j} \theta_{i,j}!$, and $\le$ denotes pointwise ordering over multi-indices. In particular, 
\begin{align}
 \label{e:mhpreduce} &\E \Big[ \big( H^{\alpha_I,\alpha_J}_{X_I,Y_J}(x,Y_J) \big)^2 \; \Big| \; X_I = x \Big]  \le c \max_{ \substack{ \hat{\alpha}_I,\hat{\alpha}_I' \le \alpha_I \\ |\hat{\alpha}_I|=|\hat{\alpha}_I'|}  }  (\lambda_{\mathrm{min}}(X_I,Y_J))^{-|\hat{\alpha}_I+\alpha_J|}    | H_{X_I}^{2\alpha_I - \hat{\alpha}_I -\hat{\alpha}_I'}(x) | .
\end{align} 
where $c > 0$ depends only on $|\alpha_I|$ and $|\alpha_J|$.
\end{proposition}
\begin{proof} 
We adapt the proof of \cite[Proposition 8]{rahman2017wiener} which gives a similar formula for unconditioned Hermite polynomials. To reduce notation we abbreviate $X = X_I$ and $Y = Y_J$. For $(s,s'),(t,t') \in \R^{|I|+|J|}$, consider the expression
\begin{align} 
\label{e:mhp3} & \int e^{(s,s')^T \Sigma_{X,Y}^{-1} (x,y)  -  \frac{1}{2} (s,s')^T \Sigma_{X,Y}^{-1} (s,s')  + (t,t')^T \Sigma_{X,Y}^{-1} (x,y) -  \frac{1}{2} (t,t')^T \Sigma_{X,Y}^{-1} (t,t') } \varphi_{Y |X = x}(y) \, dy   \\
\nonumber & \qquad = e^{(s,s')^T  \Sigma_{X,Y}^{-1} (t,t') } \frac{1}{\varphi_X(x)} \int  \varphi_{X,Y}(x-s-t,y-s'-t') \, dy \\
\label{e:mhp4} & \qquad = e^{(s,s')^T  \Sigma_{X,Y}^{-1} (t,t') } \frac{ \varphi_X(x-s-t)}{\varphi_X(x)} .
\end{align}
By a Taylor expansion
\[ e^{(w,w')^T \Sigma_{X,Y}^{-1} (x,y) - \frac{1}{2} (w,w')^T \Sigma_{X,Y}^{-1} (w,w')} = \frac{ \varphi_{X,Y}(x-w,y-w') }{ \varphi_{X,Y}(x,y) } = \sum_{\alpha,\alpha'} \frac{H_{X,Y}^{\alpha,\alpha'}(x,y)  }{(\alpha)! (\alpha')!}  (w,w')^{\alpha,\alpha'}.\]
Now setting $(w,w^\prime)=(s,s^\prime),(t,t^\prime)$ and substituting this expansion into \eqref{e:mhp3} we see that \eqref{e:mhp2} is equal to
\begin{equation}
    \label{e:c2}
c (\alpha_I)!(\alpha_J)! (\alpha_I')! (\alpha'_J)! 
\end{equation}
where $c$ is the coefficient of $(s,s',t,t')$ of order $(\alpha_I,\alpha_J,\alpha_I',\alpha'_J)$ in the expansion of \eqref{e:mhp3} (or equivalently of \eqref{e:mhp4}). Similarly, we can expand
\[  \frac{ \varphi_X(x-s-t)}{\varphi_X(x)}  = \sum_{\alpha} \frac{ (s+t)^{\alpha} }{(\alpha)! } H_{X}^{\alpha}(x) = \sum_{\alpha,\alpha'} \frac{H_X^{\alpha+\alpha'}(x)}{(\alpha)! (\alpha')!}  s^\alpha t^{\alpha'} .\]
Moreover, according to \cite[Eq.(15)]{rahman2017wiener} we also have the convergent expansion
\[  e^{(s,s')^T  \Sigma_{X,Y}^{-1} (t,t') }  = \sum_{|\hat{\alpha}_X + \hat{\alpha}_Y|=|\hat{\alpha}_X' + \hat{\alpha}_Y'|} \sum_{ \theta \in M_{\hat{\alpha}_X,\hat{\alpha}_Y }^{\hat{\alpha}_X',\hat{\alpha}'_Y} } \frac{ (\Sigma_{X,Y}^{-1})^\theta }{ \theta!} (s,s')^{\hat{\alpha}_X,\hat{\alpha}_Y}(t,t')^{\hat{\alpha}_X',\hat{\alpha}_Y'} . \] 
Combining the previous two displays, and recalling that $\bar \alpha := \alpha_I - \hat{\alpha}_I + \alpha'_I - \hat{\alpha}'_I$, we see that the coefficient of $(s,s',t,t')$ of order $(\alpha_I,\alpha_J,\alpha_I',\alpha'_J)$ in the expansion of \eqref{e:mhp4} is equal to 
\[  \sum_{ \substack{ \hat{\alpha}_I \le \alpha_I, \hat{\alpha}_I' \le \alpha_I' , \\ |\hat{\alpha}_I + \alpha_J| = |\hat{\alpha}_I' + \alpha_J'| } }    \sum_{ \theta \in M_{\hat{\alpha}_I,\alpha_J}^{\hat{\alpha}_I',\alpha_J'} } \frac{ (\Sigma_{X,Y}^{-1})^\theta }{\theta!} \frac{  H^{\bar \alpha}_X(x)  }{{(\alpha_I-\hat{\alpha}_I)!(\alpha_I'-\hat{\alpha}_I')!}} . \] 
Since this must be equal to the coefficient $c$ in \eqref{e:c2}, this establishes \eqref{e:mhp2}. The bound \eqref{e:mhpreduce} follows immediately from \eqref{e:mhp2} and the third inequality in \eqref{e:lambda1}.
\end{proof}

\begin{proposition}
\label{p:mhpbound}
For a non-degenerate Gaussian vector $X$, indices $I, J \subseteq \{1,\ldots,\mathrm{dim}(X)\}$, multi-indices $\alpha_I \in \N_0^I$ and $\alpha_J \in \N_0^J$, $t \in [0,1)$, and $(x,x') \in \R^{|I|+ |J|}$ such that $x_{I \cap J},x'_{I \cap J} \equiv \ell \in \R$,
\begin{align*}
| H_{X_I,X^t_J}^{\alpha_I,\alpha_J}(x,x')| &\varphi_{X_I,X^t_J}(x,x')\\
&\le  c_{|I|,|J|,|\bar{\alpha}|} \max\{1,\|\Sigma_X\|_\infty\}^{|\bar{\alpha}|} (1-t)^{- \frac{|\bar{\alpha}| + |I \cap J|}{2}  }  \min\{1,\lambda_{\mathrm{min}}(X)\}^{-2|\bar{\alpha}| - \frac{|I\cup J|}{2}} 
\end{align*}
where $\bar{\alpha} = \alpha_I+\alpha_J$.
\end{proposition}
\begin{proof}
Let $K = I \cap J$, and abbreviate $Z = X_K$, $Z^t = X^t_K$, $Y = X_{I\setminus K}$, and $Y' = X^t_{J \setminus K}$, and similarly $z = x|_K$, $z' = x'|_K$, $y = x|_{I \setminus K}$, and $y' = x'|_{J \setminus K}$. We let $c = c_{|I|,|J|,|\bar{\alpha}|}$ be a constant that may change from line to line. 

Then we write
\begin{align*}
  | H_{X_I,X^t_J}^{\alpha_I,\alpha_J}(x,x')| \varphi_{X_I,X^t_J}(x,x') &=   \Big|   \partial^{\alpha_I,\alpha_J}  \varphi_{X_I,X^t_J}(x,x') \Big| \\ 
  &=  \Big|   \partial^{\alpha_I,\alpha_J} \Big(   \varphi_{Z,Z^t}(z,z')   \varphi_{Y,Y'  | Z,Z^t }(y,y' | z,z')    \Big) \Big|  \\
  & \le c \max_{\alpha,\alpha' : \alpha + \alpha' = (\alpha_I,\alpha_J) }  \Big|   \partial^{\alpha}    \varphi_{Z,Z^t}(z,z')  \Big| \Big|  \partial^{\alpha'}  \varphi_{Y,Y'  | Z,Z^t }(y,y' | z,z')    \Big) \Big| .
\end{align*}
We then claim that
\begin{equation}
\label{e:mult1}
    \left\lvert \partial^\alpha\varphi_{Z,Z^t}(z,z^\prime) \big|_{z,z^\prime=\ell}\right\rvert\leq c_{|I|,|J|,|\alpha|} (1-t)^{-|\alpha|/2 - |K|/2}\min\{1,\lambda_\mathrm{min}(X)\}^{-\lvert\alpha\rvert-|K|/2}.
\end{equation}
and
\begin{equation}
\label{e:mult2}
    \left\lvert \partial^\alpha \varphi_{Y,Y'  | Z,Z^t }(y,y' | z,z') \right\rvert\leq c_{|I|,|J|,|\alpha|} \max\{1, \|\Sigma_X\|_\infty\}^{|\alpha|}  \lambda_{\mathrm{min}}(X)^{-2|\alpha|-|(I \cup J) \setminus K|/2}.
\end{equation}
which together establish the statement of the proposition.

To prove \eqref{e:mult1}, let $V$ and $D$ be defined as in Lemma \ref{l:spectral} applied to $Z$. Let $W$ denote the Gaussian vector $W \sim \mathcal{N}(0,V^T\Sigma_{Z,Z^t}V)$ which has independent components. Then $\varphi_{Z,Z^t}(z,z^\prime)=\varphi_W(V^T(z,z^\prime)) = \prod_j \varphi_{W_j}((V^T(z,z^\prime))_j)$, and so by the chain rule
\begin{align}
  \nonumber  \big| \partial^\alpha\varphi_{Z,Z^t}(z,z^\prime)\big|_{z,z^\prime=\ell} \big| & = \Big| \sum_{\lvert\beta\rvert=\lvert\alpha\rvert} \Big( \prod_{i=1}^{\lvert\alpha\rvert}V_{\beta_i,\alpha_i} \Big) \partial^\beta\varphi_W(u) \big|_{u  = V^T\underline{\ell}} \Big| \\   
  \label{e:hermite_non_deg_1} & \leq  c_{|I|,|J|,|\alpha|} \|V\|_\infty^{|\alpha|} \prod_j  \big| \partial^{\beta_j} \varphi_{W_j}(u) \big|_{ u =  (V^T\underline{\ell})_{j} } 
\end{align}
where $\underline{\ell}$ is the vector with all elements equal to $\ell$. Letting $\lambda_j$ denote the variance of $W_j$,
it is easily verified that
\begin{equation}
    \label{e:unifbound}
|\partial^\beta \varphi_{W_j}(0) | = \begin{cases}  c_\beta   \lambda_{j}^{-\beta/2-1/2} & \beta \text{ even,} \\ 0 &  \beta \text{ odd,} \end{cases}   \quad \text{and} \quad  \sup_{u \in \R} |\partial^\beta \varphi_{W_j}(u) | \le  c_\beta  \lambda_{j}^{-\beta-1/2}   .
\end{equation}
There are then two cases to consider:
\begin{itemize}
    \item $j >\lvert K\rvert$: Recalling Lemma \ref{l:spectral}, in this case $\lambda_j \ge (1-t)\lambda_\mathrm{min}(Z) \geq (1-t)\lambda_\mathrm{min}(X)$, and $(V^T\underline{\ell})_j=0$. Hence
\[  \big| \partial^{\beta_j}   \varphi_{W_{j}} (u) \big|_{ u =  (V^T\underline{\ell})_{j} }   =  \big| \partial^{\beta_j} \varphi_{W_{j}} (0) \big| \le c_{\beta_j} ((1-t)\lambda_\mathrm{min}(X))^{-\beta_j/2-1/2} . \]
 \item $j \le \lvert K\rvert$: In this case $\lambda_j \ge \lambda_\mathrm{min}(Z) \geq\lambda_\mathrm{min}(X)$, and so 
\[  \big| \partial^{\beta_j}  \varphi_{W_{j}} (u) \big|_{ u =  (V^T\underline{\ell})_{j} } \leq c_{\beta_j} \lambda_\mathrm{min}(X)^{-\beta_j -1/2 } .\]
\end{itemize}
Combining these bounds with \eqref{e:hermite_non_deg_1}, and using that $\|V\|_\infty \le 1$, we prove \eqref{e:mult1}.

To prove \eqref{e:mult2} we similarly let $U$ be orthogonal such that $U^T \Sigma_{Y,Y' | Z,Z^t} U $ is diagonal, and let $W$ denote the Gaussian vector $W \sim \mathcal{N}(0, U^T \Sigma_{Y,Y' | Z,Z^t} U)$. Notice that, by \eqref{e:lambdamin1}, the components of $W$ have variance bounded below by $\lambda_{\mathrm{min}}(X)$. Then 
\[ \varphi_{Y,Y'  | Z,Z^t }(y,y' | z,z') = \varphi_W \big( U^T \big( (y,y') - \Sigma_{(Y,Y')\to (Z,Z^t)} \Sigma_{Z,Z^t}^{-1} (z,z') \big) \big) \] 
and so, recalling \eqref{e:unifbound}, by the chain rule
\begin{align*} 
| \partial^\alpha  \varphi_{Y,Y'  | Z,Z^t }(y,y' | z,z')| & \le c_{|I|,|J|,|\alpha|}  C^{|\alpha|} \sup_{u} |\partial^\alpha \varphi_W(u) | \\
& \le c_{|I|,|J|,|\alpha|} C^{|\alpha|}   \lambda_\mathrm{min}(X)^{-|\alpha| -(|I \setminus K| + |J\setminus K|)/2 } \end{align*}  
where 
\[C := \max \big\{ \|U\|_\infty , \| U^T \Sigma_{(Y,Y')\to (Z,Z^t)} \Sigma_{Z,Z^t}^{-1}  \|_\infty \big\} . \]
Since $\|U\|_\infty \le 1$, and by Lemma \ref{l:spectral} and \eqref{e:lambda1} 
\[ \|\Sigma_{(Y,Y')\to (Z,Z^t)} \Sigma_{Z,Z^t}^{-1}\|_\infty \le c_{|I|,|J|}  \| \Sigma_{X}\|_\infty \| \Sigma_{Z}^{-1} \|_\infty \le  c_{|I|,|J|}  \| \Sigma_{X}\|_\infty \lambda_{\mathrm{min}}(X)^{-1},\]
we deduce \eqref{e:mult2}.
\end{proof}

We are now ready to complete the proof of Proposition \ref{p:mhpmain}:

\begin{proof}[Proof of Proposition \ref{p:mhpmain}]
Applying Proposition \ref{p:mhp}, specifically \eqref{e:mhpreduce}, with the substitutions $X_I \to (X_I,X_J^t)$, $Y_J \to (X_{I^c},X^t_{J^c}),$ $ \alpha_I\to (\alpha_I, \alpha_J)$ and $\alpha_J \to 0$, the left-hand side of \eqref{e:mhp5} is bounded above by  
\[ c_{|\bar{\alpha}|}   \max_{ \substack{  \hat{\alpha}_I, \hat{\alpha}_I' \le \alpha_I, \hat{\alpha}_J,\hat{\alpha}_J' \le \alpha_J \\ |\hat{\alpha}_I+\hat{\alpha}_J|=|\hat{\alpha}_I' +\hat{\alpha}_J'|} }    \lambda_{\mathrm{min}}(X,X^t)^{-|\hat{\alpha}_I + \hat{\alpha}_J|/2} \times \sqrt{ |H_{X_I,X_J^t}^{2\alpha_I - \hat{\alpha}_I - \hat{\alpha}_I', 2\alpha_J - \hat{\alpha}_J-\hat{\alpha}_J'}(\ell,\ell)|}  \varphi_{X_I,X_J^t}(\ell,\ell)  . \]
Since $\lambda_{\mathrm{min}}(X,X^t)  = (1-t)  \lambda_{\mathrm{min}}(X)$, and applying Propositions \ref{p:densitybound} and \ref{p:mhpbound}, the above is bounded by
\begin{align*}
    c\max\big\{1, \|\Sigma_X\|_\infty\big\}^{|\bar{\alpha}| }&\max_{ \substack{  \hat{\alpha}_I, \hat{\alpha}_I' \le \alpha_I, \hat{\alpha}_J,\hat{\alpha}_J' \le \alpha_J \\|\hat{\alpha}_I+\hat{\alpha}_J|=|\hat{\alpha}_I' +\hat{\alpha}_J'|} } (1-t)^{-\frac{\lvert\hat{\alpha}_I+\hat{\alpha}_J\rvert}{2}-\frac{\lvert 2\alpha_I-\hat{\alpha}_I-\hat{\alpha}_I^\prime\rvert}{4}-\frac{\lvert 2\alpha_J-\hat{\alpha}_J-\hat{\alpha}_J^\prime\rvert}{4} -\frac{\lvert I\cap J\rvert}{2}}\\
    &\qquad\qquad\times\min\{1,\lambda_{\mathrm{min}}(X)\}^{-\frac{\lvert\hat{\alpha}_I+\hat{\alpha}_J\rvert}{2}-\lvert 2\alpha_I-\hat{\alpha}_I-\hat{\alpha}_I^\prime\rvert-\lvert 2\alpha_J-\hat{\alpha}_J-\hat{\alpha}_J^\prime\rvert-\frac{|I| +| J|}{2}}
\end{align*}
which implies the result.
\end{proof}

\section{Semi-local additive functionals of stationary Gaussian fields}
\label{a:slaf}

In this appendix we extend the classical theory of local additive functionals of stationary Gaussian fields \cite{dm79,bm83} to \textit{semi-local} additive functions. Let $f$ be a stationary Gaussian field on $\Z^d$ with covariance kernel $K(x-y) = \E[f(x) f(y)]$ satisfying $K(x) \sim |x|^{-\alpha}$ for $\alpha > 0$ and $K \ge 0$. We consider functionals of the form
\[ Q_R =  \frac{1}{m!}\sum_{x_1,\ldots,x_m \in \Lambda_R} \wick{f(x_1)\cdots f(x_m)}  P(x_1,\ldots,x_m) \]
where $m \ge 1$ and $P: (\Z^d)^m \to \R$ is a signed kernel which is stationary, permutation invariant, symmetric in the sense that $P(x) = P(-x)$, and integrable in the sense that $\sum |P| := \sum_{x \in (\Z^d)^{m-1}} |P(0,x)|  < \infty$. This reduces to the `local' case considered in \cite{dm79,bm83} if $P(x_1,\ldots,x_m) = \id_{x_1=\cdots=x_m}$.

\smallskip 
By definition $Q_R$ is an element of the $m$-th homogeneous chaos of the Gaussian Hilbert space generated by $f$. In parallel to the results in \cite{dm79,bm83}, the limit theory of $Q_R$ depends on whether $m \alpha \ge d$ (central limit) or $2 \le m \alpha < d$ (non-central limit). As we show, if $2 \le m \alpha \le d$ the limit theory also depends on whether 
\begin{equation*}
\sum P := \sum_{x \in (\Z^d)^{m-1}} P(0,x) 
\end{equation*}
vanishes; the case $\sum P = 0$ exhibits new behaviour compared to the local setting. 

\smallskip 
For our results we will assume that $P$ has rapid off-diagonal decay in the sense that there exists a $\kappa > d$ such that as $R\to\infty$
\begin{equation}
    \label{e:pdecay}
\Gamma(R) := \sum_{x_2,\ldots,x_m \in \Z^d: \max_i \|x_i\|_\infty > R} |P(0,x_2,\ldots,x_m)| = O(R^{-\kappa}) .
 \end{equation}

\subsection{Central limit}
The following are generalisations of \cite[Theorems 1 and 1']{bm83} which treated the local case. For part of the result we refine our assumption that $K(x) \sim |x|^{-\alpha}$ by supposing that, as $|x| \to \infty$,
\begin{equation}
    \label{e:kreg}
  K(x) = |x|^{-\alpha}  + O \big(|x|^{-\alpha-2} \big)  .
\end{equation} 
Recall that $Z$ denotes a standard Gaussian random variable, and $\Rightarrow$ convergence in law.

\begin{proposition}
\label{p:clt}
Suppose either: (i) $m \alpha > d$ and \eqref{e:pdecay} holds for some $\kappa > 3d$; or (ii) $m\alpha > d-2$, $\sum P = 0$, \eqref{e:pdecay} holds for every $\kappa$, and \eqref{e:kreg} holds. Then there exists a constant $\tau = \tau_{K,P} \in [0,\infty)$ such that, as $R \to \infty$
\[  \frac{ \Var[Q_R]  }{ R^d } \to \tau  \qquad \text{and} \qquad   \frac{Q_R}{R^{d/2}}\quad \Longrightarrow  \quad \sqrt{\tau} Z . \]
\end{proposition}

\begin{proposition}
\label{p:clt2}
Suppose $m \alpha = d$, $\sum P \neq 0$, and \eqref{e:pdecay} holds for some $\kappa > 3d$. Then there exists a constant $c = c_d > 0$ such that, as $R \to \infty$
\[  \frac{ \Var[Q_R]  }{ R^d (\log R)^{\id_{m\alpha = d} } } \to c_d \big(\sum P \big)^2  \qquad \text{and} \qquad \frac{Q_R}{\sqrt{\Var[Q_R]}} \quad \Longrightarrow  \quad Z .\]
\end{proposition}

In the case $m \alpha > d$, we also provide a uniform bound on the variance:

\begin{proposition}
\label{p:clt3}
There exists a constant $c_K > 0$ such that, for every $m \alpha > d$ and $R \ge 1$
\[ \Var[Q_R] \le \frac{1}{m!} e^{cm} \Big( \sum |P| \Big)^2  R^d. \]
In particular if $P$ is supported on $\{ x \in (\Z^d)^m : \textrm{diam}_\infty(\underline{x}) \le r \}$ then
\[ \Var[Q_R] \le \frac{1}{m!} e^{cm} \|P\|^2_\infty  R^d \]
where $c > 0$ depends only on $K$ and $r$.
\end{proposition}

\begin{proof}[Proof of Proposition \ref{p:clt}]
We follow the approach of \cite{np12}, which streamlined the `method of moments' analysis of \cite{bm83}. To compute the variance, by the diagram formula (Theorem~\ref{t:df}) and permutation invariance
\begin{equation}
    \label{e:cltvar}
 \Var[Q_R]  =   \frac{1}{m!}  \sum_{x,y \in (\Lambda_R)^m}  \prod_{i=1}^m K(x_i - y_i ) P(x) P(y) .
\end{equation}
Applying the third item of Lemma \ref{l:rv} we have  $\Var[Q_R] \sim  \tau R^d$.  Again by the diagram formula and the third item of Lemma \ref{l:rv} we also have
\[ \E[ Q_R^4 ] \sim 3 \tau^2 R^{2d}  .\]
Since $Q_R$ is an element of a chaos of fixed order, by the fourth moment theorem \cite[Theorem 5.1.7]{np12} this implies the Gaussian limit.
\end{proof}

\begin{proof}[Proof of Proposition \ref{p:clt2}]
This is the same as for Proposition \ref{p:clt} except applying the second item of Lemma \ref{l:rv} in place of the third item.  
\end{proof}

\begin{proof}[Proof of Proposition \ref{p:clt3}]
By \eqref{e:cltvar} and Lemma \ref{l:rv3},
\[ \Var[Q_R]  \le  \frac{1}{m!} c^m  R^d  \Big( \sum |P| \Big)^2   \sum_{x \in \Z^d} \max\{1,|x|\}^{-m\alpha}  \]
for a constant $c = c_K > 0$, and the result follows.
\end{proof}

\subsection{Non-central limit}

We now consider the non-central limit theory in the case $2 \le m \alpha < d$, following closely the approach of \cite{dm79}. Let $\mu$ denote the spectral measure of $f$, i.e.\ the finite measure on $[-\pi,\pi]^d \subset \R^d$ such that
 \[ K(x) = \mathcal{F}[\mu](x) = \int e^{i \langle x, \lambda \rangle} d\mu(\lambda) .\]
By \cite[Proposition 1]{dm79} there exists a locally finite non-atomic measure $\mu_0$ on $\R^d$ such that
\begin{equation}
    \label{e:mu0}
 \mu_0(\cdot) = \lim_{R \to \infty} \mu_R(\cdot) := \lim_{R \to \infty} R^\alpha  \mu(R^{-1} \cdot) 
 \end{equation}
in the sense of weak convergence on compact sets. 
 
\begin{definition}[Hermite distributions]
\label{d:ros}
 For $2 \le m < d/\alpha$, the \textit{$m$-th order Hermite distribution (associated to $\mu_0$)} is the distribution of
\begin{equation}
    \label{e:z'}
Z' = c \int_{(\R^d)^m} S_0(\lambda_1 + \cdots + \lambda_m)  W_{\mu_0}(d\lambda_1) \cdots W_{\mu_0}(d\lambda_m)   
\end{equation}
where $W_{\nu}$ is the (complex) white noise on the space $L^2_\text{sym}(\nu)$ of Hermitian functions $h$ such that $\int |h|^2 d\nu < \infty$ , $\int$ denotes the multiple Wiener-It\^{o} integral with respect to $W_{\nu}$ (see \cite{dob79} for the definition and basic properties), $S_0 = 2^{-d}\mathcal{F}[\id_{[-1,1]^d}]$, and $c = c_{m, \mu_0} > 0$ is a normalising constant chosen so that $\Var[Z'] = 1$. The Wiener-It\^{o} integral is well-defined since, as verified in \cite{dm79},
\[ \int_{(\R^d)^m} |S_0(\lambda_1 + \cdots + \lambda_m)|^2  \mu_0(d\lambda_1) \cdots \mu_0(d\lambda_m)  < \infty .\]
\end{definition}

In the case that $K$ is the Green's function $G$, $\mu_0(d\lambda)$ has density proportional to $|\lambda|^{\alpha-d}$, and so \eqref{e:z'} is equivalent in law to
\[Z' = c' \int_{(\R^d)^m} S_0(\lambda_1 + \cdots + \lambda_m)  \frac{ W(d\lambda_1) \cdots W(d\lambda_m) }{ |\lambda_1|^{(d-\alpha)/2} \cdots |\lambda_m|^{(d-\alpha)/2} } , \]
where $W$ is the standard white noise in $L^2_{\mathrm{sym}}(\R^d)$.
 
 \smallskip 
 The following is a generalisation of \cite[Theorem 1']{dm79}, which treated the local case:

\begin{proposition}
\label{p:nclt}
Suppose $2 \le m \alpha < d$, $\sum P \neq 0$, and \eqref{e:pdecay} holds for some $\kappa > d$. Then there exists a constant $E_{d,m\alpha} > 0$ such that, as $R \to \infty$,
\[  \frac{\Var[Q_R]}{R^{2d-m\alpha}} \to \frac{ E_{d,m\alpha} ( \sum P)^2 }{m!}  \qquad \text{and} \qquad \frac{Q_R}{\sqrt{\Var[Q_R]}} \quad \Longrightarrow \quad Z'  \]
where $Z'$ has the $m$-th order Hermite distribution associated to $\mu_0$. The constant $E_{d,\alpha}$ is defined for $\alpha < d$ as
\begin{equation}
    \label{e:e}
 E_{d,\alpha} := \int_{x,y \in [-1,1]^d} |x-y|^{-\alpha} \, dx dy  = \int_{\R^d} \mathcal{S}_d(x) |x|^{-\alpha} \, dx  \in (0,\infty),
 \end{equation}
 where $\mathcal{S}_d = \big(\id_{[-1,1]^d} \star \id_{[-1,1]^d}\big)$ with $\star$ denoting convolution, and the second equality in \eqref{e:e} is by the identity $\int f (g \star h) = \int (f \star g) h$.
\end{proposition}
\begin{proof}
The variance asymptotics follow from \eqref{e:cltvar} and the first item of Lemma \ref{l:rv}. To prove the convergence in distribution, we closely follow the proof of \cite[Theorem 1']{dm79}.  Recall the measure $\mu_R$ defined in \eqref{e:mu0}. Noting that
\[ \wick{f(x_1) \cdots f(x_m)}  \, \stackrel{d}{=} \, \int_{(\R^d)^m} e^{ i ( \langle x_1, \lambda_1 \rangle + \cdots \langle x_m, \lambda_m \rangle )} W_\mu(d\lambda_1) \cdots W_\mu(d\lambda_m) \]
we have

\begin{align*}
     & \frac{1}{(\sum P) R^{d-m\alpha/2} } Q_R \\
     & \qquad \stackrel{d}{=} R^{-d}   \frac{1}{m!\sum P} \sum_{x \in (\Lambda_R)^m } P(x) \int_{(\R^d)^m }  e^{ \frac{i}{R} ( \langle x_1, \lambda_1 \rangle + \cdots + \langle x_m, \lambda_m \rangle )} W_{\mu_R}(d\lambda_1) \cdots W_{\mu_R}(d\lambda_m) \\
     & \qquad =  \frac{1}{m!}\int_{(\R^d)^m} S_R(\lambda_1,\ldots,\lambda_m) W_{\mu_R}(d\lambda_1) \cdots W_{\mu_R}(d\lambda_m)
     \end{align*}
     where 
     \[ S_R(\lambda_1,\cdots,\lambda_m) =  R^{-d} \frac{1}{\sum P}  \sum_{x \in (\Lambda_R)^m } P(x) e^{ \frac{i}{R} ( \langle x_1, \lambda_1 \rangle + \cdots + \langle x_m, \lambda_m \rangle )}. \]
     \cite[Lemma 3]{dm79} states that the desired convergence in distribution holds provided that the following two conditions are satisfied:
     \begin{enumerate}
 \item As $R \to \infty$
  \begin{equation}
      \label{e:dm1}
  S_R(\lambda_1,\ldots,\lambda_m) \to S_0(\lambda_1+\cdots+\lambda_m)  
  \end{equation}
  uniformly over compact sets;
  \item  Uniformly over $R \ge 1$
  \begin{equation}
      \label{e:dm2}
\lim_{A \to \infty} \int_{(\R^d)^m \setminus ([-A,A]^{d})^m } |S_R(\lambda_1,\ldots,\lambda_m)|^2 \, d\mu_R(\lambda_1) \cdots d\mu_R(\lambda_m) = 0  .   \end{equation}
\end{enumerate}

To verify \eqref{e:dm1}, we first observe that $S_0(\lambda_1+\ldots+\lambda_m)$ is the Fourier transform of the probability measure $\eta_0$ that is uniformly distributed on $\{x \in ([-1,1]^d)^m : x_1 = x_2 = \ldots x_m \}$, which can be seen from the change of coordinates $(x_1,\dots,x_m)\mapsto(x_1,x_2-x_1,\dots,x_m-x_1)$. By splitting $P$ into its positive and negative parts, we may also assume that $P \ge 0$. Then since the measures 
\[ \eta_R=  R^{-d}  \frac{1}{\sum P} \sum_{x \in (\Lambda_R)^m } P(x) \delta_{x/R}    \]
converge weakly to $\eta_0$, we conclude by using the standard fact that weak convergence of probability measures implies local uniform convergence of their Fourier transforms.

To verify \eqref{e:dm2} define the measures 
\[ \tilde{\mu}_R(\lambda) :=  |S_R(\lambda_1,\ldots,\lambda_m)|^2 \, d\mu_R(\lambda_1) \cdots d\mu_R(\lambda_m)    \]
and
\[ \tilde{\mu}_0(\lambda) := |S_0(\lambda_1 + \cdots + \lambda_m)|^2 d\mu_0(\lambda_1) \cdots d\mu_0(\lambda_m)  .\]
By \eqref{e:dm1} we know that $\tilde{\mu}_R \to \tilde{\mu}_0$ in the sense of weak convergence on compact sets, and it suffices to show that in fact $\tilde{\mu}_R$ converges weakly. For $t \in (\R^d)^m$ define
\[ \varphi_R(t)  = \int_{(\R^d)^m} e^{ \frac{i}{R} (\langle [t_1 R],\lambda_1 \rangle + \cdots + \langle [t_mR], \lambda_m \rangle )} |S_R(\lambda_1,\ldots,\lambda_m)|^2 \, d\mu_R(\lambda_1) \cdots d\mu_R(\lambda_m)   . \]
Expanding out $S_R$, one sees that 
   \[  \varphi_R(t) := \frac{1}{R^{2d-m\alpha}} \frac{1}{(\sum P)^2} \sum_{x,y\in (\Lambda_R)^m } \prod_{i=1}^m K(x_i-y_i+[t_i R])  P(x) P(y) . \] 
By the first item of Lemma \ref{l:rv}, 
 \[ \varphi_R(t) \to   E^m_{d,\alpha}(t) \]
 uniformly, where $E^m_{d,\alpha}(t)$ is defined in \eqref{e:et}. Since $E^m_{d,\alpha}(t)$ is continuous (\cite[Lemma 1]{dm79}), this implies the weak convergence of $\tilde{\mu}_R$ to the Fourier transform of $E^m_{d,\alpha}(t)$ (\cite[Lemma 2]{dm79}) which must therefore be $\mu_0$.
\end{proof}

\section{Semi-local extensions of standard kernel computations}
\label{a:kc}
In this appendix we give semi-local extensions of some computations involving sums over a stationary kernel $K: \Z^d \to \R$ satisfying $K(x) = K(-x)$, $K \ge 0$, and $K(x) \sim |x|^{-\alpha}$ as $|x| \to \infty$ for some $\alpha > 0$. These results were used extensively in Sections \ref{s:sl} and \ref{s:mr} and Appendix \ref{a:slaf}.

\subsection{Stationary asymptotics}
\label{ss:ba}
We first study asymptotics when $K$ is weighted by a signed kernel $P : (\Z^d)^m \to \R$ which is stationary, permutation invariant, symmetric, and integrable (in the sense of Appendix \ref{a:slaf}).

Recall the decay assumptions \eqref{e:pdecay} and \eqref{e:kreg} on $P$ and $K$ respectively from Appendix \ref{a:slaf}, and the constant $E_{d,\alpha}$ from \eqref{e:e}. We extend \eqref{e:e} by defining, for $m\alpha < d$, the function $E^m_{d,\alpha}: (\R^d)^m \to (0,\infty)$
\begin{align}
\nonumber E^m_{d,\alpha}(t) & := \int_{x,y \in [-1,1]^d} |x-y+t_1|^{-\alpha} \cdots |x-y+t_m|^{-\alpha} \, dx dy  \\
\label{e:et}  & = \int_{\R^d} \mathcal{S}_d(x)  |x + t_1|^{-\alpha} \cdots |x+t_m|^{-\alpha} \, dx  
\end{align}
so that $E_{d,m\alpha} =  E^m_{d,\alpha}(0)$. In \cite[Lemma 1]{dm79} it is shown that $t \mapsto E^m_{d,\alpha}(t)$ is continuous. We also extend $E_{d,\alpha}$ by setting
\[ E_{d,d} := \lim_{R \to \infty} \frac{ \sum_{x \in \Lambda_R} |x|^{-d} }{\log R} \in (0,\infty) .\]

For $m \ge 1$ and $4m$ vertices labelled $x_1,\ldots,x_m$, $y_1,\ldots,y_m$, $u_1,\ldots,u_m$ and $v_1,\ldots,v_m$, a \textit{valid diagram} $\gamma$ is a perfect matching of the vertices such that no edge has both endpoints with label having the same letter (these are precisely the Feynman diagrams used to compute the expected product of four Wick polynomials: see Theorem~\ref{t:df}). Let $\mathfrak{D}_m$ denote the set of all such valid diagrams. For $\gamma \in \mathfrak{D}_m$ and points $x,y,u,v \in (\Z^d)^m$, the \textit{value} $v_\gamma(x,y,u,v)$ of $\gamma$ is $\prod_{e \in \gamma} K_e$ where $K_e = K(a-b)$ for each edge $e = (a,b)$ in $\gamma$ (identifying the point $x_i \in \R^d$ with the vertex $x_i$ in the natural way).

Let $[t] \in \Z^d$ denote the integer part of $t \in \R^d$. 

\begin{lemma}
\label{l:rv}
$\,$
\begin{enumerate}

\item Suppose $m \alpha < d$,  $\sum P \neq 0$, and \eqref{e:pdecay} holds for some $\kappa > d$. Then as $R \to \infty$
\begin{equation}
    \label{e:ba1}
  \sum_{x,y \in (\Lambda_R)^m} \prod_{i=1}^m K \big(x_i - y_i + [t_i R] \big) P(x) P(y) \sim  E^m_{d,\alpha}(t)  \Big( \sum P \Big)^2  R^{2d - m\alpha}  
  \end{equation}
uniformly over $t \in (\R^d)^m$. In particular, setting $P(x) = \id_{x_1 = \cdots =x_m}$ and $t=0$,
\[  \sum_{x,y \in \Lambda_R}  K(x-y)^m  \sim  E_{d,m\alpha}   R^{2d - m\alpha} . \]

\item Suppose $m \alpha = d$, $\sum P \neq 0$, and \eqref{e:pdecay} holds for some $\kappa > d$. Then as $R \to \infty$
\begin{equation}
    \label{e:ba2}
    \sum_{x,y \in (\Lambda_R)^m} \prod_{i=1}^m K(x_i - y_i ) P(x) P(y)  \sim  E_{d,d} \big(\sum P \big)^2 R^d (\log R) 
    \end{equation}
and if \eqref{e:pdecay} holds for some $\kappa > 3d$,
\begin{equation}
    \label{e:ba3}
   \frac{1}{(m!)^2} \sum_{\gamma \in \mathfrak{D}_m} \sum_{x,y,u,v \in (\Lambda_R)^m}   v_\gamma(x,y,u,v)  P(x) P(y) P(u) P(v) \sim  3 E_{d,d}^2  \big(\sum P \big)^4 R^{2d} (\log R)^2 .
    \end{equation}

\item Suppose either (i) $m\alpha > d$ and \eqref{e:pdecay} holds for some $\kappa > d$, or (ii) $m \alpha > d-2$, $\sum P = 0$, \eqref{e:pdecay} holds for every $\kappa$, and \eqref{e:kreg} holds. Then there exists a constant  $\tau = \tau_{K,P} \in [0,\infty)$ such that, as $R \to \infty$
\begin{equation}
    \label{e:ba4}
\sum_{x,y \in (\Lambda_R)^m} \prod_{i=1}^m K(x_i - y_i ) P(x) P(y) \sim \tau R^d 
\end{equation}
and if \eqref{e:pdecay} also holds for some $\kappa > 3d$,
\begin{equation}
    \label{e:ba5}
   \frac{1}{(m!)^2} \sum_{\gamma \in \mathfrak{D}_m}  \sum_{x,y,u,v \in (\Lambda_R)^m}  v_\gamma(x,y,u,v) P(x) P(y) P(u) P(v) \sim  3 \tau^2   R^{2d}  .
    \end{equation}
\end{enumerate}
\end{lemma}

\begin{remark}
    \label{r:beta}
    Comparing Lemma \ref{l:rv} with \eqref{e:beta}, we see that 
    \[ \beta_{d,k} = \begin{cases}
        c_d E_{d,k(d-2)} &  \text{if } k(d-2) < d , \\
        c_d E_{d,d} &  \text{if } k(d-2) = d ,
    \end{cases} \]
    where $c_d$ is such that $G(x) \sim c_d |x|^{2-d}$.
\end{remark}

\begin{remark}
Among the statements \eqref{e:ba1}--\eqref{e:ba5}, the permutation invariance of $P$ is only used to prove \eqref{e:ba3} and \eqref{e:ba5}, and symmetry only for the cases $\sum P = 0$ of \eqref{e:ba4}--\eqref{e:ba5}.
\end{remark}

\begin{proof}[Proof of Lemma \ref{l:rv}]
We will use $x_{[i,j]}$ to denote the vector $(x_i,\ldots,x_j)$.

\textbf{(1).} For $R\in\N$, writing $x_1 = [uR]$ and $y_1 = [vR]$, by stationarity the left-hand side of \eqref{e:ba1} can be expressed as
\begin{equation*}
\begin{aligned}
    R^{2d-m\alpha}  \int_{u,v \in [-1,1+1/R]^d}  R^\alpha &K([uR]-[vR] + [t_1R]) \sum_{w,z \in (\Z^d)^{m-1}}    F_R(u,v,t_{[2,m]},w,z) \, du dv
\end{aligned}
\end{equation*}
where $ F_R(u,v,t_{[2,m]},w,z) $ equals
\[   \id_{ x_1+w , y_1 +z  \in ( \Lambda_R)^{m-1}} \prod_{i=1}^{m-1} R^\alpha K\big( x_1-y_1 + w_i - z_i + [t_iR] \big) P(0, w) P(0,z)  . \]
Define the set
\[ A_R(u,v,t) = \big\{ (w,z) \in (\Z^d)^{m-1} :  \|w\|_\infty, \|z\|_\infty  \leq  \min_i \|[uR]-[vR]+[t_iR]\|_\infty^{1-\delta/(2d)}  \big\} \]
where $\delta \in (0,1)$ is such that \eqref{e:pdecay} holds for $\kappa > d + \delta$. Since $K(x) \sim |x|^{-\alpha}$, there exists a constant $c_1 > 0$ depending only on $K$ such that, for all $u,v,t$,
\[ (w,z) \in A_R(u,v,t) \quad \Longrightarrow  \quad  R^\alpha K\big(w_i - z_i + x_1-y_1+ [t_iR] \big) \le  c_1 \max\{1,|u-v+t_i |\}^{-\alpha} . \]
Moreover for fixed $u,v,t$ such that $u-v+t_i\neq 0$, as $R \to \infty$
\[ R^\alpha K\big(w_i - z_i + x_1-y_1+ [t_iR] \big) \to    |u-v+t_i |^{-\alpha} \]
uniformly on $A_R(u,v,t)$.

Fix truncation parameters $\eps,r > 0$. Define 
\[ E^m_{d,\alpha}(t; \eps) := \int_{x,y \in [-1,1]^d, \|x-y+t_i\|_\infty \ge \eps} |x-y+t_1|^{-\alpha} \cdots |x-y+t_m|^{-\alpha} \, dx dy  . \]
Since $m \alpha < d$, an application of H\"{o}lder's inequality shows that $E^m_{d,\alpha}(t; \eps) \to E^m_{d,\alpha}(t)$ as $\eps \to 0$ uniformly over $t$. Moreover define $\sum_{\le r} P := \sum_{\|x\|_\infty \le r} P(0,x)$ which satisfies $\sum_{\le r} P \to \sum P$ as $r \to \infty$. 

We first consider the contribution from the set $\| u-v + t_i \|_\infty \ge \eps \, \forall i$, and $\|w\|_\infty, \|z\|_\infty \le r$. Note that there exists $R_0 > 0$ depending only on $\eps,r$ such that, if $R \ge R_0$, then $(w,z) \in A_R(u,v,t)$ holds on this set. Then by boundedness and compactness, 
we have that
\[\int_{ \substack{u,v \in [-1,1+1/R]^d \\ \|u - v + t_i\|_\infty \ge \eps}}   R^\alpha K([uR]-[vR]+[t_1R]) \!\! \sum_{ \substack{ w, z\in (\Z^d)^{m-1} \\ \|w\|_\infty, \|z\|_\infty \le r}}   F_R(u,v,t_{[2,m]},w,z) \, du dv  \]
converges to $E^m_{d,\alpha}(t; \eps)  (\sum_{\le r} P)^2$ as $R \to \infty$, uniformly over $t$. Since this limit converges to $E^m_{d,\alpha}(t)  (\sum P)^2$ as $\eps \to 0$ and $r \to \infty$, uniformly over $t$, it remains to show the negligibility of the contribution from the sets  (i)  $\exists i \text{ s.t. } \|u - v + t_i\| \le  \eps$ and $(w,z) \in A_R(u,v,t)$, (ii) $\max\{ \|w\|_\infty, \|z\|_\infty \} > r $ and $(w,z) \in A_R(u,v,t)$, and (iii)  $(w,z) \notin A_R(u,v,t)$.

Let $c_i > 0$ be constants independent of $R$ and $t$ that may change from line to line. Since $P$ is bounded and $m \alpha < d$, the contribution from the first set is bounded by
\begin{align*}
&  \int_{ \substack{u,v \in [-2,2]^d \\ \exists i \text{ s.t. } \|u - v + t_i\| \le \eps} }  R^\alpha K([uR]-[vR]+[t_1R]) \sum_{w, z \in (\Z^d)^{m-1}}  \id_{A_R(u,v,t)}   \big| F_R(u,v,t_{[2,m]},w,z) \big| \, du dv  \\
& \qquad \le  c_2  \int_{ \exists i \text{ s.t. } \|u - v + t_i\| \le \eps}   \prod_i  \max\{1,|u-v+t_i |\}^{-m \alpha}\, dudv \le c_\eps
\end{align*}
for some $c_\eps \to 0$ as $\eps \to 0$, independent of $R$ and $t$.

Similarly the contribution from the second set is bounded by
\begin{align*}
&  \int_{ u,v \in [-2,2]^d}   R^\alpha K([uR]-[vR]+[t_1R]) \sum_{\substack{ w, z \in (\Z^d)^{m-1} \\ \max\{ \|w\|_\infty, \|z\|_\infty \} > r  } }  \id_{A_R(u,v,t)}   \big| F_R(u,v,t_{[2,m]},w,z) \big| \, du dv  \\
& \qquad \le  c_2 \big(\sum |P| \big) \Gamma(r)  \int_{ u,v \in [-2,2]^d  } \prod_i  \max\{1,|u-v+t_i |\}^{-m \alpha} \, dudv \\
& \qquad \le c_3 \Gamma(r) .
\end{align*}

Finally, since $K$ and $P$ are bounded, and recalling that \eqref{e:pdecay} holds for $\kappa > d + \delta$, we have
\begin{align*}
&  \int_{u,v \in [-2,2]^d}  R^\alpha K([uR]-[vR]+[t_1R]) \sum_{w, z \in (\Z^d)^{m-1}}  \id_{A_R(u,v,t)^c}   \big| F_R(u,v,t_{[2,m]},w,z) \big| \, du dv \\
& \qquad \le  c_2 R^{m\alpha}  \int_{u,v \in [-2,2]^d} \Gamma \big( (c_3 R \min_i\|u-v+t_i\|_\infty)^{1-\delta/(2d)} \big)\, dudv \\
& \qquad  \le  c_4  R^{m\alpha}  \int_{u,v\in[-2,2]^d}\ind_{\min_i\| u-v+t_i\|_\infty\leq 1/(c_3 R)}\\
&\qquad\qquad\qquad\qquad\qquad\qquad\qquad+(R\min_i\|u-v+t_i\|_\infty)^{-(d+\delta/4)}\ind_{\min_i\|u+v-t_i\|_\infty\geq 1/(c_3 R)}\;dudv\\
&\qquad\leq c_5 R^{m\alpha}\left(\int_{\|u\|\leq 1/(c_3 R)}\;du+\int_{\|u\|\geq 1/(c_3 R)}(R\|u\|)^{-(d+\delta/4)}\;du\right)\leq c_6 R^{m\alpha-d} .
\end{align*}
Taking $\eps \to 0$ and $r \to \infty$ completes the proof.

\textbf{(2).} Let $r \in (0,R)$ and define the set 
\[ B_{R,r} = \big\{ (x,y) \in (\Lambda_R)^m : x_1 \in \Lambda_{R - r}, x_{[2,m]} \in (x_1 + \Lambda_r)^{m-1}, y \in (x_1 + \Lambda_r)^m   \big\} . \]
By stationarity, the contribution to \eqref{e:ba2} from $B_{R,r}$ is equal to $\tau_r |\Lambda_{R - r}|$ where
\begin{equation}
    \label{e:tau}
 \tau_r :=   \sum_{x_{[2,m]} \in (\Lambda_r)^{m-1}, y \in (\Lambda_r)^m}  K(y_1) \prod_{i=2}^m K\big(x_i - y_i \big) P(0, x_{[2,m]}) P(y) .  
 \end{equation}
On the other hand, the contribution outside $B_{R,r}$  is at most 
\[ e_{2R} |\Lambda_R \setminus \Lambda_{R - r}|  + (e_{2R} - e_r) |\Lambda_R|   \] 
where
\begin{equation}
    \label{e:er}
e_r :=   \sum_{x_{[2,m]} \in (\Lambda_r)^{m-1}, y \in  (\Lambda_r)^{m} }  K(y_1) \prod_{i=2}^m K\big(x_i - y_i \big) |P(0, x_{[2,m]})| |P(y)|  .  
\end{equation}
Let $c_d > 0$ be defined as 
\[ c_d := \lim_{R \to \infty} (\log R)^{-1} \sum_{x \in \Lambda_R} |x|^{-d} .\]
Then we claim that as $r\to\infty$
\begin{equation}
\label{e:logcase}
\tau_r \sim c_d \big(\sum P \big)^2 \log r  \quad \text{and} \quad e_r \sim c_d \big(\sum |P| \big)^2 \log r  .
\end{equation}
Now fix $\eps \in (0,1)$ and set $r = \eps R$. Given \eqref{e:logcase}, the left-hand side of \eqref{e:ba2} is asymptotic to
\[ c_d \big(\sum P \big)^2  |\Lambda_{R - \eps R}| (\log \eps R)   + O(1) \Big(  |\Lambda_R \setminus \Lambda_{R - \eps R} | (\log 2R)  + |\Lambda_R|   \big( \log (2R) - \log (\eps R) \big) \Big) ,\]
which gives \eqref{e:ba2} by sending $\eps \to 0$. 

To show \eqref{e:logcase}, for $y_1 \in \Z^d$ define the set
\[ A(y_1) = \{  (x,y) \in (\Z^d)^{m-1} :  \max\{  \|x\|_\infty,  \|y\|_\infty \} \le  \| y_1 \|_\infty^{1-\delta/(2d)}   \} \]
where $\delta \in (0,1)$ is such that \eqref{e:pdecay} holds for $\kappa > d + \delta$. Observe that, as $|y_1| \to \infty$
\[ R^\alpha K\big(x_i - (y_1 + y_i ) \big) \sim    |y_1|^{-\alpha} \]
uniformly for $(x_{[2,m]},y_{[2,m]}) \in A(y_1)$. Then 
\begin{align}
\label{e:logcase2}
\nonumber & \sum_{y_1 \in \Lambda_{r/2}, (x_{[2,m]},y_{[2,m]}) \in  A(y_1)}  K(y_1) \prod_{i=2}^m K\big(x_i - (y_1 + y_i) \big) P(0, x_{[2,m]}) P(0, y_{[2,m]}])  \\
& \qquad \qquad \qquad \qquad \qquad \qquad \qquad \qquad \sim   \sum_{y_1 \in \Lambda_{r/2}} |y_1|^{-d}  (P'_{y_1})^2 
\end{align}
where 
\[ P'_{y_1} : = \sum_{x_{[2,m]} \in \big( \Lambda_{ \|y_1\|_\infty^{1-\delta/(2d)} } \big)^m } P(0,x_{[2,m]} )  . \]
Since $P'_{y_1} \to \sum P$ as $|y_1| \to \infty$, \eqref{e:logcase2} is asymptotic to 
\[ \big(\sum P \big)^2 \sum_{y_1 \in \Lambda_{r/2}} |y_1|^{-d} \sim c_d \big(\sum P\big)^2 \log r   . \]
On the other hand, the difference between $\tau_r$ and \eqref{e:logcase2} is at most
\[ c   \big(\sum |P| \big)^2 \sum_{ y_1 \in \Lambda_r \setminus \Lambda_{r/2} } |y_1|^{-d}  + \sum_{y_1 \in \Lambda_r } \Gamma \big(\|y_1\|_\infty^{1-\delta/(2d)} \big)  \]
which is bounded. This shows that $\tau_r$ satisfies \eqref{e:logcase}, and the proof for $e_r$ is identical.

We now turn to \eqref{e:ba3}. We say that a valid diagram $\gamma \in \mathfrak{D}_m$ is \textit{regular} if the endpoint labels partition $\{x,y,u,v\}$. For example the valid diagram with edges $\{ (x_1,y_1), (x_2,y_2),(u_1,v_1),(u_2,v_2)\}$ is regular, but the one with edges $\{ (x_1,y_1), (x_2,u_2),(y_2,v_2),(u_1,v_1)\}$ is not. Letting $\mathfrak{R}_m \subseteq \mathfrak{D}_m$ be the subset of regular valid diagrams, by permutation invariance of $P$ we have
\begin{align*}
&  \frac{1}{(m!)^2}   \sum_{\gamma \in \mathfrak{R}_m} \sum_{x,y,u,v \in (\Lambda_R)^m} v_{\gamma}(x,y,u,v) P(x) P(y) P(u) P(v)  \\
& \qquad  =  3 \sum_{x,y,u,v \in (\Lambda_R)^m} \prod_{i=1}^m K(x_i-y_i) \prod_{i=1}^m K(u_i - v_i) P(x) P(y) P(u) P(v) .
\end{align*}
Combining with \eqref{e:ba2}, to complete the proof of \eqref{e:ba3} it suffices to show that the contribution to \eqref{e:ba3} from every non-regular valid diagram is of negligible order $o(R^{2d} (\log R)^2)$. 

If $m=1$ all valid diagrams are regular so we suppose that $m \ge 2$ and fix a non-regular valid diagram $\gamma$. Let $x,y \in (\Lambda_R)^m$ and suppose that $\max\{ \|x_{[2,m]}-x_1 \|_\infty, \| y_{[2,m]}-y_1 \|_\infty \} \le \|x_1-y_1\|_\infty / 4$. Then every edge in $\gamma$ of the form $(x_i,y_i)$ contributes a factor of at most $c |x_i - y_i|^ {-\alpha}$ to $v_\gamma(x,y,u,v)$. Arguing similarly for all pairs, and using that \eqref{e:pdecay} holds for $\kappa > 3d$, the contribution to \eqref{e:ba3} from $x_1,y_1,u_1,v_1 \in \Lambda_R$ is at most 
\[ W(x_1,y_1,u_1,v_1) = c_4 \prod_{(w_1,z_1)} |w_1-z_1|^{-k_{w,z} \alpha}   \]
where the sum is over all six of the ordered pairs of $\{x_1,y_1,u_1,v_1\}$, $c_4 > 0$ depends only on $K$ and $P$, and $k_{w,z} \ge 0$ is the number of edges in the diagram whose endpoints have labels $w$ and $z$, and hence $k_{w,z}$ sum to $2m$ and satisfy $k_{x,y} + k_{x,u} + k_{x,v} = m$. It remains to show that $\sum_{x_1,y_1,u_1,v_1 \in \Lambda_R} W(x_1,y_1,u_1,v_1) = o(R^{2d} (\log R)^2)$, which is similar to computations carried out in, e.g., \cite[p. 435]{bm83} and \cite[p. 132]{np12}. To give the main idea, by repeatedly using the inequality $s^a t^b \le \max\{ s^{a+b}, t^{a+b} \} \le s^{a+b} + t^{a+b}$, which holds for $s,t,a,b \ge 0$, one sees that
\begin{equation}
    \label{e:w}
\sum_{x_1,y_1,u_1,v_1 \in \Lambda_R} W(x_1,y_1,u_1,v_1) \le c_5 R^d \sum_{x \in \Lambda_{2R}} |x_1|^{-m\alpha}  \sum_{u \in \Lambda_{2R} } |u|^{-q \alpha}   \sum_{v \in \Lambda_{2R} } |v|^{-r \alpha}   
\end{equation}
for some $q,r \ge 0$ and $q + r = m$. Moreover since $\gamma$ is non-regular, one has $q,r > 0$. To conclude, recalling that $m \alpha = d > \max\{q\alpha,r\alpha\}$, \eqref{e:w} is bounded by 
\[ c_6 R^d (\log R) R^{d-q\alpha} R^{d-r\alpha} = O \big(R^{2d} (\log R) \big) . \]

\textbf{(3).} We first consider the case $m\alpha > d$, which is similar to the previous item. Let $\tau_r$ and $e_r$ be defined as in \eqref{e:tau} and \eqref{e:er}. We will prove that the following limits exist
\begin{equation}
\label{e:limits}
 \tau_\infty := \lim_{\tau \to \infty} \tau_r \qquad \text{and} \qquad  e_\infty := \lim_{r \to \infty} e_r. 
 \end{equation} 
 Given \eqref{e:limits}, fixing $r > 0$ and arguing as in the previous item shows that, as $R \to \infty$, the left-hand side of \eqref{e:ba4} is 
\[ \tau_r  |\Lambda_{R - r}| + O \big( |\Lambda_R \setminus \Lambda_{R - r} | \big)  + O\big( (e_{2R} - e_r) |\Lambda_R| \big).\]
 Sending $r \to \infty$ gives \eqref{e:ba4} for $\tau = \tau_\infty$ (we must have $\tau \ge 0$ since \eqref{e:ba1} is non-negative by \eqref{e:cltvar}).

It is sufficient to show \eqref{e:limits} for $e_r$ since this implies absolute convergence of the series defining $\tau_r$. For $y_1 \in \Z^d$ define the set
\[ A(y_1) = \{  x,y \in (\Z^d)^{m-1} :  \max\{  \|x\|_\infty,  \|y\|_\infty \} \le  \| y_1 \|_\infty / 4   \} . \]
Since $K(x) \sim |x|^{-\alpha}$, there exists a $c_1 > 0$ such that, for all $y_1$, 
\[ (x_{[2,m]},y_{[2,m]}) \in A(y_1) \quad \Longrightarrow  \quad   K\big(x_i - (y_1+y_i)  \big) \le  c_1 \min\{1,  |y_1|^{-\alpha} \} . \]
Since $P$ is bounded, and recalling \eqref{e:pdecay}, 
\[ e_\infty < c_2 \sum_{y_1 \in \Z^d} \Big( \min\{1,  |y_1|^{-m\alpha} \}  + \Gamma(\| y_1 \|_\infty / 4 ) \Big)  \]
is finite, as required. 

The proof of \eqref{e:ba5} is essentially the same as that of \eqref{e:ba3}, except since we assume $m\alpha > d$, \eqref{e:w} is bounded instead by
\[ c_3 R^d R^{\max\{0,d-q\alpha\}}(\log R)^{\id_{q\alpha=d}} R^{\max\{0,d-r\alpha\}}(\log R)^{\id_{r\alpha=d}} = o(R^{2d}) \]
as required.

We turn to the case in which $m \alpha > d-2$, $\sum P = 0$, \eqref{e:pdecay} holds for all $\kappa$, and \eqref{e:kreg} holds. Let $\delta > 0$ be such that $\alpha + 2 - 2\delta > d$, let $\beta \in (0,1)$ be such that $\beta(d+1+2md\delta) < 1$, and let $r = R^\beta$. Define the set 
\[ B_R = \big\{ x,y \in (\Lambda_R)^m : x_1 \in \Lambda_{R - 2r}, y_1 \in x_1 + \Lambda_r, x_{[2,m]} \in (x_1 + \Lambda_{r^\delta})^{m-1}, y_{[2,m]}\in (y_1 + \Lambda_{r^\delta})^{m-1}  \big\} . \]
By stationarity, the contribution to \eqref{e:ba4} from $B_R$ is equal to $\tau_r |\Lambda_{R - 2r}|$ where
\[  \tau_r :=   \sum_{y_1 \in \Lambda_{r}} \sum_{x_{[2,m]}, y_{[2,m]} \in (\Lambda_{r^\delta})^{m-1}}  K(y_1) \prod_{i=2}^m K\big(x_i - (y_1+ y_i) \big) P(0, x_{[2,m]}) P(0,y_{[2,m]}]) .  \]
On the other hand, since $K$ and $P$ are bounded, the contribution outside $B_R$  is at most $E_1 + E_2 + E_3$ where
\[ E_1 = \Big| \sum_{(x,y) \in B'_R}  K(x_1-y_1) \prod_{i=2}^m K\big(x_1 - y_1 + x_i - y_i \big) P(0, x_{[2,m]}) P(0,y_{[2,m]}) \Big|  \]
with 
\[B'_R =  \big\{ x,y \in (\mathbb{Z}^d)^m : x_1,y_1 \in \Lambda_R, \|y_1 - x_1\|_\infty > r, x_{[2,m]},y_{[2,m]} \in (\Lambda_{r^\delta})^{m-1}  \big\} , \] 
\[ E_2 \le c_1 |\Lambda_R|^2 \Gamma(r^\delta)  =o(R^{d}) ,\]
and 
\[ E_3 \le  c_2 |\Lambda_R \setminus \Lambda_{R-2r}|  |\Lambda_r||\Lambda_{r^\delta}|^{2m} =     o(R^d)  ,\]
where to bound $E_2$ and $E_3$ we used respectively \eqref{e:pdecay} and our choice of $\beta$. By \eqref{e:kreg},
\[ K\big(x_1 - y_1 + x_i - y_i \big) = |x_1-y_1|^{-\alpha} + \frac{ - \alpha \langle x_i-y_i , x_1-y_1\rangle }{|x_1-y_1|^{\alpha+2}} + O \big( (r^\delta)^2 |x_1-y_1|^{-\alpha-2} \big) \]
 uniformly for every $(x,y) \in B'_R$, and so $ K(x_1-y_1)  \prod_{i=2}^m K (x_1 - y_1 + x_i - y_i )$ is
 \[  |x_1-y_1|^{-m\alpha} +  \frac{ - \alpha \sum_{i=2}^m \langle x_i-y_i , x_1-y_1\rangle }{|x_1-y_1|^2}  + O \big( r^{2\delta} |x_1-y_1|^{-\alpha-2} \big) .  \]
  Pairing up each $x_{[2,m]}$ with $-x_{[2,m]}$ to eliminate the first order correction, $E_1$ is at most a constant times
  \[  |\Lambda_R|^2 \Big( \sum P - \Gamma(r^\delta ) \Big)^2  +   |\Lambda_R| r^{2\delta}  \sum_{x \in \Z^d \setminus \Lambda_r} |x|^{-\alpha-2}   .\]
  Since we assume $\sum P = 0$ and \eqref{e:pdecay} holds for all $\kappa$, and since $\alpha + 2 - d > 2\delta $, we have $E_1 = o(R^{d})$. 
  
To complete the proof of \eqref{e:ba4} it remains to show that $\tau_r \to \tau_\infty < \infty$. Using that $K$ and $P$ are bounded we have
\[ |\tau_{r+1} - \tau_r | \le E_4 + O(|\Lambda_{r+1}| \Gamma(r^\delta)) \]
where, using \eqref{e:kreg} as above
\begin{align*}
E_4 &= \Big| \sum_{y_1 \in \partial \Lambda_{r+1}} \sum_{x_{[2,m]}, y_{[2,m]} \in (\Lambda_{r^\delta})^{m-1}}  K(y_1) \prod_{i=2}^m K\big(x_i - (y_1+ y_i) \big) P(0, x_{[2,m]}) P(0,y_{[2,m]}]) \Big| \\
& \le c_3 r^{d-1}  \Big( \sum P - \Gamma(r^\delta) \Big)^2 + c_3 r^{2\delta} \sum_{y_1 \in \partial \Lambda_{r+1}} |y_1|^{-\alpha-2} = O(r^{d-1-\alpha-2+2\delta}) .
\end{align*} 
By \eqref{e:pdecay}, and since $\alpha + 2 - d > 2\delta $, we conclude that $\tau_r$ is a Cauchy sequence.

The proof of \eqref{e:ba5} is similar to in previous cases and we omit the details.
\end{proof}

\subsection{Boundary asymptotics}

We next study a boundary variant of the first item of Lemma~\ref{l:rv}, for which we need a stronger decay assumption on the kernel. If $\alpha < d-1$, define
\begin{equation}
\label{e:edot}
 \overline{E}_{d,\alpha} = \int_{x,y \in \partial [-1,1]^d} |x-y|^{-\alpha } \in (0,\infty) .
 \end{equation}

\begin{lemma}
\label{l:rv2}
Suppose $\alpha < d-1$ and let $\gamma(k) : \N_0 \to \R$ be such that $\gamma(k) k^\kappa \to 0$ for every $\kappa > 0$. Then as $R \to \infty$, 
\begin{equation}
    \label{e:erv2}
\sum_{x,y \in \Lambda_R} K(x-y ) \gamma \big( d_\infty(x, \partial \Lambda_R ) \big) \gamma \big( d_\infty(y, \partial \Lambda_R ) \big) \sim  \overline{E}_{d,\alpha}   \Big( \sum \gamma \Big)^2  R^{2(d-1)-\alpha}     . 
\end{equation} 
\end{lemma}
\begin{proof}
Fix $\delta > 0$ such that $\delta^2 + d\delta < d-1 -\alpha$. For $i = 0,\ldots,d-1$, recall that $F_R^i$ is the union of the $i$-dimensional boundary faces of $\Lambda_R$. Define the subset $D_R =\{x \in \Lambda_R : d_\infty(x,\partial\Lambda_R) < d_\infty(x,F_R^{d-2})^\delta \}$. Each $x \in D_R$ can be uniquely projected onto its nearest boundary point $\pi_R(x) \in \partial \Lambda_R$. Define
\[ B_R = \big\{(x,y) \in D_R^2 : \max\{ |x- \pi_R(x)| , |y-\pi_R(y) | \}  \le |x-y|^\delta \big\} .\]

Writing $x_1 = [uR]$ and $y_1 = [vR]$, the contribution to \eqref{e:erv2} from $(x,y) \in B_R$ can be expressed as
\[ R^{2(d-1)-\alpha}  \int_{u,v \in \partial[-1,1]^d}     F_R(u,v) \, du dv \]
where $F_R(u,v) $ equals
\[  \sum_{(x,y) \in B_R : \pi_R(x) = x_1, \pi_R(y) = y_1}  R^\alpha K(x-y) \gamma(|x-x_1|)  \gamma(|y-y_1|). \]
Since $K(x) \sim |x|^{-\alpha}$, there exists a constant $c_1 > 0$ such that, for all $u,v$, 
\[  R^\alpha K(x - y ) \le  c_1 \max\{1,|u-v |\}^{-\alpha}  \quad \text{and} \quad R^\alpha K(x-y) \sim    |u-v|^{-\alpha} \]
uniformly on $\{(x,y) \in B_R : \pi_R(x) = x_1, \pi_R(y) = y_1\}$.  Since $\sum |\gamma| < \infty$  and  $|u-v|^{-\alpha}$ is integrable on $u,v \in \partial[-1,1]^d$, by dominated convergence
\[ \int_{u,v \in \partial[-1,1]^d}     F_R(u,v) \, du dv \to   \overline{E}_{d,\alpha}   \Big( \sum \gamma \Big)^2  \]
as $R \to \infty$. 

On the other hand, the contribution to \eqref{e:erv2} from $(x,y) \notin B_R$ is bounded by $E_1 + E_2$ where
\[ E_1 \le 2\|\gamma\|_\infty \sum_{x,y \in \Lambda_R  : d_\infty(y, \partial \Lambda_R) > R^{\delta^2}} K(x-y) \sup_{k\geq R^{\delta^2}}\lvert \gamma(k)\rvert = o(1) \]
and, using that $K$ and $\gamma$ are bounded,
\[ E_2 \le c_{K,\gamma} \sum_{x,y \in \{w \in \Lambda_R : d_\infty(w, \partial \Lambda_R) \le R^{\delta^2}  \} } \id_{d_\infty(x,y) \le  R^\delta } \le c'_{K,\gamma} R^{d-1} R^{\delta^2} R^{d\delta} = o(R^{2(d-1)-\alpha}) \]
by our choice of $\delta$.
\end{proof}

\subsection{General bounds}
Finally we establish general bounds on some related quantities. For simplicity, in this section we work with the kernel $K_\alpha(x) = \max\{1, |x|^{-\alpha}\}$.

\smallskip
For $R \ge 1$ and $m \ge 1$, let $\Gamma_R: (\Lambda_R)^m \to \R$ be permutation invariant. First we give a reduction to the $m=1$ case:

\begin{lemma}
\label{l:rv3}
For $m \ge 2$ and $x_1 \in \Lambda_R$ define 
\[ \bar{\Gamma}_R(x_1) =  \sum_{x_2,\ldots,x_m \in \Lambda_R} |\Gamma_R(x_1,x_2,\ldots,x_m)| . \]
Then
\[ \sum_{x,y \in (\Lambda_R)^m} \prod_{i=1}^m K_\alpha(x_i - y_i ) \Gamma_R(x) \Gamma_R(y) \le m  \sum_{x,y \in \Lambda_R} K_{m \alpha}(x-y)  \bar{\Gamma}_R(x) \bar{\Gamma}_R(y) . \]
\end{lemma}
\begin{proof}
This follows from the inequality $\prod_{i=1}^ms_i\leq \max_i s_i^m\leq\sum_{i=1}^ms_i^m$ for $s_1,\dots,s_m\geq 0$.
\end{proof}

Next we consider bounds in the case $m=1$. For $i = 0,\ldots,d-1$ recall that $F_R^i$ is the union of the $i$-dimensional boundary faces of $\Lambda_R$, and let $\gamma_i(k)$ be such that $0 \le \Gamma_R(x) \le  \gamma_i(d_\infty(x,F_R^i))$. 

\begin{lemma}
\label{l:rv4}
$\,$
\begin{enumerate}
\item If $i = d-1$ and $\gamma_{d-1}(k) \to 0$ as $k \to \infty$, then as $R \to \infty$
\[ \sum_{x,y \in \Lambda_R} K_\alpha(x-y ) \Gamma_R(x) \Gamma_R(y)  = o \big(R^{\max\{ 2d-\alpha, d \} } (\log R)^{\id_{\alpha = d} }  \big). \]
\item  If $\sum_{k \ge 0} \gamma_i(k) k^{d-i-1} < \infty$, then there exists $c = c_{K,i} > 0$ such that
\[ \sum_{x,y \in \Lambda_R} K_\alpha(x-y ) \Gamma_R(x) \Gamma_R(y)  \le c \Big( \sum_{k \ge 0} \gamma_i(k) k^{d-i-1}  \Big)^2 R^{\max\{ 2i-\alpha, i \} } (\log R)^{\id_{\alpha = i} }  . \]
\end{enumerate}
\end{lemma}

\begin{proof}
$\,$
\textbf{(1).} Fix $\eps > 0$ and let $r = r_\eps > 0$ be such that $|\gamma(k)| < \eps$ if $k > r$. Then  $\sum_{x,y \in \Lambda_R} K_\alpha(x-y ) \Gamma_R(x) \Gamma_R(y)$ is at most
\begin{align*}
    &  \|\gamma\|_\infty^2 \sum_{x,y \in \{w: d_{\infty}(w, \partial \Lambda_R) \le r \} } K_\alpha(x-y)  + \eps \|\gamma\|_\infty \sum_{x,y \in \Lambda_R} K_\alpha(x-y)   \\
    & \qquad \le c_{d,\gamma}(rR^{d-1}+\epsilon R^d)\sum_{x\in\Lambda_{2R}}K_\alpha(x)\\
    &\qquad\le c'_{d,\gamma}(r/R+\eps) R^{\max\{2d-\alpha, d\} } (\log R)^{\id_{\alpha = d} } ,
\end{align*}
and taking $\eps \to 0$ gives the result. 

\textbf{(2).} Let $W_k$ denote the subset of $\Lambda_R$ such that $d_\infty(x,F_R^i) = k$. Then 
\[ \sum_{x,y \in \Lambda_R} K_\alpha(x-y ) \Gamma_R(x) \Gamma_R(y) \le  \sum_{k_1,k_2 = 0}^R \gamma_i(k_1) \gamma_i(k_2)  \sum_{x \in W_{k_1}, y \in W_{k_2} } K_\alpha(x-y) . \]
Since $| W_k|  \le c_{d,i} R^i k^{d-i-1}$, and by the monotonicity of $K_\alpha$, 
\[   \sum_{x \in W_{k_1}, y \in W_{k_2} } K_\alpha(x-y) \le c'_{d,i} R^i k_1^{d-i-1} k_2^{d-i-1}  \sum_{y \in \Lambda_{2R} \cap (\Z^i \times \{0\}^{d-i}) } K_\alpha(y)    .\]
Combining we have
\[  \sum_{x,y \in \Lambda_R} K_\alpha(x-y ) \Gamma_R(x) \Gamma_R(y) \le c'_{d,i}  \Big( \sum_{k=0}^\infty \gamma_i(k)  k^{d-i-1}  \Big)^2  R^i \sum_{y \in \Lambda_{2R} \cap (\Z^i \times \{0\}^{d-i}) } K_\alpha(y) \]
and we conclude since 
\begin{equation*}
    \sum_{y \in \Lambda_{2R} \cap (\Z^i \times \{0\}^{d-i} ) } K_\alpha(y)  \le c_{K,i} R^{\max\{ i -\alpha,0 \} } (\log R)^{\id_{\alpha=i}} .  \qedhere
    \end{equation*}
\end{proof}

\bigskip

\bibliographystyle{halpha-abbrv}
\bibliography{scgff.bib}

\end{document}